\documentclass[10pt, a4paper, reqno, english]{amsart}

\usepackage{hyperref}
\usepackage{amssymb}
\usepackage[all]{xy}
\usepackage{enumerate}

\newcommand\A{\mathcal{A}}

\newcommand\AAA{\mathbb{A}}
\newcommand\CC{\mathbb{C}}
\newcommand\PP{\mathbb{P}}
\newcommand\ZZ{\mathbb{Z}}
\newcommand\ZZp{\ZZ_{>0}}
\newcommand\ZZnn{\ZZ_{\ge 0}}
\newcommand\QQ{\mathbb{Q}}
\newcommand\RR{\mathbb{R}}
\newcommand\RRnn{\RR_{\ge 0}}

\newcommand\xx{\mathbf{x}}
\newcommand\yy{\mathbf{y}}
\newcommand\I{\mathcal{I}}
\newcommand\R{\mathcal{R}}
\newcommand{\ADE}{\mathbf{ADE}}
\newcommand{\Aone}{{\mathbf A}_1}
\newcommand{\Atwo}{{\mathbf A}_2}
\newcommand{\Athree}{{\mathbf A}_3}
\newcommand{\Afour}{{\mathbf A}_4}
\newcommand{\Afive}{{\mathbf A}_5}
\newcommand{\Dfour}{{\mathbf D}_4}
\newcommand{\Dfive}{{\mathbf D}_5}
\newcommand{\Esix}{{\mathbf E}_6}
\newcommand{\tS}{{\widetilde S}}
\newcommand{\abs}[1]{\left\|#1\right\|_\infty}
\newcommand{\ee}{\boldsymbol{\eta}}
\newcommand{\e}{\eta}
\newcommand{\te}{{\tilde\eta}}
\newcommand\dd{\,\mathrm{d}}
\newcommand\Where{\,\Big|\,}
\newcommand\WHERE{\,\Bigg|\,}
\newcommand{\ex}[1]{*+<5pt>[o][F]{E_{#1}}}
\newcommand{\li}[1]{*+<3pt>[F]{E_{#1}}}
\newcommand{\classrep}{\mathcal{C}}
\newcommand{\classtuple}{\mathbf{C}}
\newcommand{\OO}{\mathcal{O}}
\newcommand{\eI}{I}
\newcommand{\eII}{\mathbf{I}}
\newcommand{\N}{\mathfrak{N}}
\newcommand{\id}[1]{\mathfrak{#1}}
\newcommand{\aaa}{\id a}
\newcommand{\bbb}{\id b}
\newcommand{\ideal}{\id a}
\newcommand{\p}{\id p}
\newcommand{\kc}{\id{k}_\id{c}}
\newcommand{\kb}{\id{k}_\id{b}}
\newcommand{\idealclass}{\mathfrak{k}}
\newcommand{\bO}{\boldsymbol{\mathcal{O}}}
\newcommand{\Ao}{\Pi_1}
\newcommand{\At}{\Pi_2}
\newcommand{\rk}{\rho}
\newcommand{\eeA}{\boldsymbol{\eta}_{\boldsymbol{A}}}
\newcommand{\eeB}{\boldsymbol{\eta}_{\boldsymbol{B}}}
\newcommand{\eeC}{\boldsymbol{\eta}_{\boldsymbol{C}}}
\newcommand{\eIIA}{\eII_{\boldsymbol{A}}}
\newcommand{\eIIB}{\eII_{\boldsymbol{B}}}
\newcommand{\eIIC}{\eII_{\boldsymbol{C}}}
\newcommand\rto{\dashrightarrow}

\DeclareMathOperator\Pic{Pic}
\DeclareMathOperator\vol{vol}
\DeclareMathOperator\diam{diam}
\DeclareMathOperator\Cox{Cox}

\newtheorem{theorem}{Theorem}
\newtheorem{lemma}[theorem]{Lemma}
\newtheorem{prop}[theorem]{Proposition}
\newtheorem{claim}[theorem]{Claim}
\newtheorem{corollary}[theorem]{Corollary}
\theoremstyle{definition}
\newtheorem{definition}[theorem]{Definition}
\newtheorem{remark}[theorem]{Remark}
\newtheorem*{ack}{Acknowledgements}

\numberwithin{theorem}{section}
\numberwithin{equation}{section}

\begin{document}

\setcounter{tocdepth}{1}

\title{Counting imaginary quadratic points via universal torsors}

\author{Ulrich Derenthal}

\address{Mathematisches Institut, Ludwig-Maximilians-Universit\"at
  M\"unchen, Theresienstr. 39, 80333 M\"unchen, Germany}

\email{ulrich.derenthal@mathematik.uni-muenchen.de}

\author{Christopher Frei}

\address{Institut f\"ur Mathematik A, Technische Universit\"at Graz,
  Steyrergasse 30, 8010 Graz, Austria}

\email{frei@math.tugraz.at}

\date{April 11, 2013}

\begin{abstract}
  A conjecture of Manin predicts the distribution of rational points
  on Fano varieties.  We provide a framework for proofs of Manin's
  conjecture for del Pezzo surfaces over imaginary quadratic fields,
  using universal torsors. Some of our tools are formulated over
  arbitrary number fields.

  As an application, we prove Manin's conjecture over imaginary
  quadratic fields $K$ for the quartic del Pezzo surface $S$ of
  singularity type $\Athree$ with five lines given in $\PP_K^4$ by the
  equations $x_0 x_1 - x_2 x_3 = x_0 x_3 + x_1 x_3 + x_2 x_4 = 0$.
\end{abstract}


%
%

\maketitle

\tableofcontents

\section{Introduction}

Let $S$ be a del Pezzo surface defined over a number field $K$ with
only $\ADE$-singularities, let $H$ be a height function on $S(K)$
given by an anticanonical embedding, and let $U$ be the subset
obtained by removing the lines in $S$. If $S(K)$ is Zariski-dense in
$S$, we are interested in the counting function
\begin{equation}\label{eq:def_NUH}
  N_{U, H}(B) := |\{\xx \in U(K) \mid H(\xx)\leq B\}|.
\end{equation}
In this setting, Manin's conjecture \cite{MR89m:11060, MR1032922} (generalized in
\cite{MR1679843} to include our singular del Pezzo surfaces) predicts
an asymptotic formula of the form
\begin{equation}\label{eq:manin_formula}
  N_{U, H}(B) = c_{S, H}B(\log B)^{\rk-1}(1 + o(1)),
\end{equation}
where $\rho$ is the rank of the Picard group of a minimal
desingularization of $S$. The positive constant $c_{S, H}$ was made
explicit by Peyre \cite{MR1340296} and Batyrev--Tschinkel
\cite{MR1679843}.

Over $\QQ$, Manin's conjecture is known for several del Pezzo surfaces
and some other classes of varieties. To our knowledge, all currently
known cases of Manin's conjecture over number fields beyond $\QQ$
concern varieties with a suitable action of an algebraic group and can
be proved via harmonic analysis on adelic points (e.g., flag varieties
\cite{MR89m:11060}, toric varieties \cite{MR1620682}, and equivariant
compactifications of additive groups \cite{MR1906155}; this includes
some del Pezzo surfaces, classified in \cite{MR2753646}).

In this article, we provide a framework for proofs of the above
formula over imaginary quadratic fields for del Pezzo surfaces without
such a special structure.  Where no additional efforts are required,
our results are formulated for arbitrary number fields.

These methods are then applied to prove Manin's conjecture for the del
Pezzo surfaces over arbitrary imaginary quadratic fields $K$ of degree
$4$ and type $\Athree$ with five lines, with respect to their
anticanonical embeddings in $\PP_K^4$ given by the equations
\begin{equation}\label{eq:def_S}
  x_0 x_1 - x_2 x_3 = x_0 x_3 + x_1 x_3 + x_2 x_4 = 0\text.
\end{equation}
This is the first proof of Manin's conjecture over number fields
beyond $\QQ$ for varieties where the harmonic analysis approach cannot
be applied.

Similar applications of our framework allow the treatment of at least the
split quartic del Pezzo surfaces of types $\Athree+\Aone$, $\Afour$, $\Dfour$,
$\Dfive$ over imaginary quadratic fields \cite{arXiv:1304.3352}.

\subsection{Background}
Apart from the general results mentioned above for varieties with
large group actions, Manin's conjecture is known over $\QQ$ for
projective hypersurfaces whose dimension is large enough compared to
their degree, via the Hardy--Littlewood circle method 
\cite{MR0150129, MR1340296}.

For low-dimensional varieties without such actions of algebraic
groups, Manin's conjecture is known so far only in isolated cases over
$\QQ$, for heights given by specific anticanonical
embeddings. In particular, the case of del Pezzo surfaces has
been investigated from the beginning (e.g., see
\cite[Proposition~5.4]{MR1032922} and \cite[\S 8--11]{MR1340296} for
some toric del Pezzo surfaces of degree $\ge 6$ over $\QQ$,
\cite[Appendix]{MR89m:11060}, \cite{MR1681100} for computational
evidence in degree $3$ over $\QQ$).

The most important technique is the use of universal torsors, which
were invented by Colliot-Th\'el\`ene and Sansuc (see
\cite{MR89f:11082}, for example) and first applied to Manin's conjecture by
Salberger (see \cite{MR1679841}, \cite{MR1679842}). The testing ground
was a new proof in the case of split toric varieties over $\QQ$
\cite{MR1679841}.

The central milestones beyond toric varieties were the first examples
of possibly singular del Pezzo surfaces of degrees $5$
\cite{MR1909606}, $4$ \cite{MR2373960}, $3$ \cite{MR2332351}, and $2$
\cite{arXiv:1011.3434} that are not covered by \cite{MR1620682} or
\cite{MR1906155}. A long series of further examples followed, all of
them over $\QQ$, each dealing with difficulties not encountered
before. Also all higher-dimensional results involving universal
torsors concern varieties over $\QQ$ (specific cubic hypersurfaces of
dimension $3$ \cite{MR2329549} and $4$ \cite{arXiv:1205.0190}).

A relatively general strategy has emerged for split singular del Pezzo
surfaces over $\QQ$ whose universal torsors are open subsets of affine
hypersurfaces, as classified in \cite{math.AG/0604194}. This is
summarized in \cite{MR2520770}. In that basic form, it turns out to be
sufficient for quartic del Pezzo surfaces over $\QQ$ of types $\Dfive$
\cite{MR2320172}, $\Dfour$ \cite{MR2290499}, $\Afour$
\cite{MR2543667}, $\Athree+\Aone$ \cite{MR2520770} and $\Athree$ with
five lines (see Theorem~\ref{thm:a3_main_Q}).

For the cubic surfaces of types $\Esix$ \cite{MR2332351}, $\Dfive$
\cite{MR2520769} and $\Afive+\Aone$ \cite{arXiv:1205.0373} over $\QQ$,
the strategy of \cite{MR2520770} goes through when combined with
significant further analytic input. In other cases such as
\cite{arXiv:1207.2685}, larger deviations from \cite{MR2520770} seem
necessary.

Over number fields beyond $\QQ$, we have the classical result of
Schanuel \cite{MR557080} for projective spaces (which are toric) that
can be interpreted as a basic case of the universal torsor approach,
and a new proof of Manin's conjecture via universal torsors for the
toric singular cubic surface of type $3\Atwo$ (\cite{arXiv:1105.2807}
over imaginary quadratic fields of class number $1$ and
\cite{arXiv:1204.0383} over arbitrary number fields).

Our goal is to generalize the universal torsor approach towards
Manin's conjecture to non-toric varieties over number fields other than
$\QQ$. The two main general challenges arise from the unavailability
of unique factorization (if the class number is greater than $1$) and
from difficulties in regard to counting lattice points (if $K$ has
more than one Archimedean place, whence the unit group of its ring
of integers is infinite). Furthermore, the existing results over $\QQ$
often combine the universal torsor method with subtle applications of
deep results from analytic number theory that are only available over
$\QQ$ in their full strength. To mitigate these additional
difficulties, it seems natural to focus on singular quartic del Pezzo
surfaces first.

\subsection{Results}
Our main results are the techniques presented in Sections
\ref{sec:passage}--\ref{sec:further_summations}, which are described
in slightly more detail below.

They allow a rather straightforward treatment of the split quartic del
Pezzo surfaces of types $\Athree$ with five lines, $\Athree+\Aone$, $\Afour$,
$\Dfour$, $\Dfive$ over imaginary quadratic fields. They should
also be enough for some del Pezzo surfaces of higher degree (e.g., the
ones treated over $\QQ$ in \cite{arXiv:0710.1583} of type $\Atwo$ in
degree $5$, \cite{MR2769338} of type $\Atwo$ and \cite{MR2559866} of
type $\Aone$ with three lines in degree $6$). We expect that an
application to the cubic cases mentioned above or to other quartic del
Pezzo surfaces (such as the ones treated over $\QQ$ in \cite{MR2961294}
of type $\Athree$ with four lines, \cite{MR2853047} of types $3\Aone$
and $\Atwo+\Aone$, \cite{MR2874644}, \cite{MR2980925} of types
$2\Aone$ with eight lines, and the smooth quartic del Pezzo surfaces
of \cite{MR2838351}) would require additional work.

In Section \ref{sec:A3} we demonstrate how to apply our techniques by
proving the following case of Manin's conjecture.

Let $K \subset \CC$ be an imaginary quadratic field with ring of
integers $\OO_K$, discriminant $\Delta_K$, class number $h_K$, and
with $\omega_K := |\OO_K^\times|$ units. On $\PP_K^4(K)$, we use the
(exponential) Weil height given by
\begin{equation}\label{eq:height}
  H(x_0 : \cdots : x_4) := \frac{\max\{\abs{x_0}, \dots, \abs{x_4}\}}{\N(x_0\OO_K+\dots+x_4\OO_K)}\text,
\end{equation}
where $\abs{x_i} := |x_i|^2$ for the usual complex absolute value
$|\cdot |$ and $\N\aaa$ denotes the absolute norm of a fractional ideal
$\aaa$.

Let $S \subset \PP^4_K$ be the del Pezzo surface of degree $4$ defined
by \eqref{eq:def_S}.  Up to isomorphism, it is the unique split del
Pezzo surface that contains a singularity of type $\Athree$ and five
lines.

\begin{theorem}\label{thm:a3_main}
  Let $K$ be an imaginary quadratic field. Let $U$ be the complement of the
  lines in the del Pezzo surface $S \subset \PP^4_K$ defined over $K$
  by \eqref{eq:def_S}.  For $B \geq 3$, we have
  \begin{equation*}
    N_{U,H}(B) = c_{S, H} B(\log B)^5 + O(B(\log B)^4\log \log B),
  \end{equation*}
  with
  \begin{equation*}
    c_{S,H} =\frac{1}{4320}\cdot \frac{(2\pi)^6 h_K^6}{\Delta_K^4\omega_K^6}\cdot \prod_\p \left(1-\frac{1}{\N\p}\right)^6\left(1+\frac{6}{\N\p}+\frac{1}{\N\p^2}\right) \cdot \omega_\infty\text,
  \end{equation*}
  where $\p$ runs over all nonzero prime ideals of $\OO_K$ and
  \begin{equation*}
   \omega_\infty = \frac{12}{\pi}\int_{\max\{\abs{z_0z_2^2},
    \abs{z_1z_2^2}, \abs{z_2^3}, \abs{z_0z_1z_2},
    \abs{z_0z_1(z_0+z_1)}\}\leq 1}\dd z_0 \dd z_1 \dd z_2.
  \end{equation*}
\end{theorem}
Since $S$ is split, its minimal desingularization $\tS$ is a blow-up
of $\PP^2_K$ in five rational points in almost general position, hence
$\rk = \text{rk}\Pic(\tS) = 6$, so our result agrees with Manin's
conjecture. See Theorem~\ref{thm:a3_main_Q} for the analogous result
over $\QQ$.

It would be interesting to see an explicit application of
\cite[Theorem 1.1]{arXiv:1210.1792} giving Manin's
conjecture for the family of fourfolds over $\QQ$ obtained by
Weil restriction of our surfaces over varying imaginary quadratic
fields $K$.

\subsection{Techniques and plan of the paper}\label{sec:plan}
What follows is a short description of our main results and how they
should be applied to prove Manin's conjecture for some split del Pezzo
surfaces $S$ over imaginary quadratic fields. How this works in the
specific case of $S$ defined by \eqref{eq:def_S} is shown in our proof
of Theorem~\ref{thm:a3_main} in Section \ref{sec:A3}.

In Section \ref{sec:arithmetic_functions}, we investigate sums of two
classes of arithmetic functions over general number fields.

In Section \ref{sec:lattice_points}, we consider the problem of
asymptotically counting lattice points in certain bounded subsets of
$\CC = \RR^2$ given by inequalities of the form $\abs{f_i(z)} \leq
\abs{g_i(z)}$, with polynomials $f_i$, $g_i \in \CC[X]$. We use the
notion of sets of \emph{class m} introduced by Schmidt
\cite{MR1330740} and reduce our counting problems to a classical
result of Davenport \cite{MR0043821}. Moreover, we prove a tameness
result for parametric integrals over semialgebraic functions, which
can be applied to show that certain volume functions arising in
partial summations do not oscillate too much.

In Section \ref{sec:passage}, we describe a strategy to
parameterize, up to a certain action of a power of the unit group,
$K$-rational points on $U$ of bounded height by points $(\e_1, \ldots,
\e_t)$ on a universal torsor $\mathcal{T}$ over a minimal
desingularization $\tS$ of $S$ with coordinates $\e_i$ in certain
fractional ideals $\OO_i$ of $K$ and satisfying certain coprimality and
height conditions. If $K$ is $\QQ$ or imaginary quadratic, we propose
a parameterization (Claim \ref{claim:passage}) that is closely related
to the geometry of $\tS$. We expect this to work whenever
$\mathcal{T}$ is an open subset of a hypersurface in affine space $\AAA^t_K$
provided that
the anticanonical embedding $S \subset \PP^4_K$ is chosen favorably.
In \cite{math.AG/0604194}, all such del Pezzo surfaces are classified
and suitable models are given.

It is usually straightforward to prove Claim \ref{claim:passage} in
special cases by induction over a chain of blow-ups of $\PP^2_K$
giving $\tS$. Using the structure of $\Pic(\tS)$, we show that certain
steps in this induction hold in general.  To deal with the lack of
unique factorization in $\OO_K$, we apply arguments introduced by
Dedekind and Weber.

In Section \ref{sec:first_summation}, we provide the tools to sum the
result of our parameterization in Section \ref{sec:passage} over two
variables $\e_{t-1}, \e_t$, using our lattice point counting results
from Section \ref{sec:arithmetic_functions}. Unavailability of unique
factorization leads to difficulties of a technical nature. The results
of this and the next section are specific to imaginary quadratic
fields.

In Section \ref{sec:second_summation}, we provide a general tool to
sum the main term in the result of Section \ref{sec:first_summation}
over a further variable $\e_{t-2}$. Depending on the form of the equation
defining the universal torsor $\mathcal{T}$ in a specific application,
this result will be applied in two different ways.

In applications to specific del Pezzo surfaces, it still remains to
estimate the error terms in the first and second summations. This is
straightforward for some singular del Pezzo surfaces of degree $4$ and
higher, but much harder for del Pezzo surfaces of lower degree that
are smooth or have mild singularities. To handle additional cases, the
most elementary trick is to choose different orders of summations
depending on the relative sizes of the variables. Our results are
compatible with this trick, and indeed it is heavily applied in the
proof of Theorem~\ref{thm:a3_main} (with four different orders of
summations; fortunately, two of them can be handled by symmetry).

In Section \ref{sec:further_summations}, we prove a result handling
the summations over all the remaining variables $\e_1, \dots,
\e_{t-3}$ at once, under certain assumptions on the main term after
the second summation.  The results in this section are formulated in
terms of ideals instead of elements, which appears to be the natural
way to generalize the respective versions over $\QQ$. It seems
interesting to point out that in our applications, we find an
opportunity to pass from sums over elements to sums over ideals right
after the second summation (cf. Lemma
\ref{lem:A3_second_summation_ideals_M11} and Lemma
\ref{lem:A3_second_summation_ideals_M12} in the $\Athree$-case).

\subsection{Notation}\label{sec:notation}
The symbol $K$ will always denote a fixed number field, which is in
some sections arbitrary and in some sections imaginary quadratic or
$\QQ$. We denote the degree of $K$ by $d$, and the number of real
(resp. complex) places of $K$ by $s_1$ (resp. $s_2$). By $\classrep$,
we denote a fixed system of integral representatives for the ideal
classes of $K$, i.e., $\classrep$ contains exactly one integral ideal
from each class.

When we use Vinogradov's $\ll$-notation or Landau's $O$-notation,
the implied constants may always depend on $K$. In cases where they
may depend on other objects as well, we mention this, for example
by writing $\ll_C$ or $O_C$ if the constant may depend on $C$.

In addition to the notation introduced before Theorem
\ref{thm:a3_main}, we use $R_K$ to denote the regulator of $K$ and
$\I_K$ to denote the monoid of nonzero ideals of $\OO_K$. The symbol
$\aaa$ (resp. $\p$) always denotes an ideal (resp. nonzero prime
ideal) of $\OO_K$, and $v_\p(\aaa)$ is the non-negative integer such that
$\p^{v_\p(\aaa)}\mid \aaa$ and $\p^{v_\p(\aaa)+1}\nmid \aaa$. We extend this
in the usual way to fractional ideals (with $v_\p(\{0\}) := \infty$),
and for $x \in K$, write $v_\p(x) := v_\p(x\OO_K)$ for the usual
$\p$-adic exponential valuation.

We say that $x \in K$ is \emph{defined modulo $\aaa$}
(resp. \emph{invertible modulo $\aaa$}) if $v_\p(x)\geq 0$
(resp. $v_\p(x)=0$) for all $\p \mid \aaa$. If $x$ is defined modulo
$\aaa$, then it has a well-defined residue class modulo $\aaa$, and we
write $x \equiv_\aaa y$ if the residue classes of $x$, $y$
coincide, or equivalently, $v_\p(x-y)\geq v_\p(\aaa)$ for all $\p\mid \aaa$.

Sums and products indexed by (prime)
ideals always run over nonzero (prime) ideals. For simplicity, we
define
\begin{equation*}
  \rho_K := \frac{2^{s_1}(2\pi)^{s_2}R_K}{\omega_K \sqrt{|\Delta_K|}}\text.
\end{equation*}
By $\tau_K(\aaa)$ (resp. $\omega_K(\aaa)$), we denote the number of
distinct divisors (resp. distinct prime divisors) of $\aaa \in \I_K$, and $\mu_K$
is the M\"obius function on $\I_K$. Moreover, $\phi_K$ is Euler's
$\phi$-function for $\I_K$, and $\phi_K^*(\aaa) := \phi_K(\aaa)/\N\aaa =
\prod_{\p \mid \aaa}(1-1/\N\p)$.

\begin{ack}
  We thank Martin Widmer for his kind suggestion to prove Lemma
  \ref{lem:omin} via o-minimality and Antoine Chambert-Loir for the
  reference \cite{MR1644093}.

  The first-named author was supported by grant DE 1646/2-1 of the
  Deutsche Forschungsgemeinschaft. The second-named author was
  partially supported by a research fellowship of the Alexander von
  Humboldt Foundation. This collaboration was supported by the Center
  for Advanced Studies of LMU M\"unchen.
\end{ack}

\section{Arithmetic functions}\label{sec:arithmetic_functions}
In this section, $K$ can be any number field of degree $d \geq 2$ (for
$d=1$, see \cite{MR2520770}).  We will need to deal with sums
involving certain coprimality conditions, which are encoded by
arithmetic functions of the following type, analogous to
\cite[Definition~6.6]{MR2520770}.

\begin{definition}
  Let $\bbb \in \I_K$ and $C_1$, $C_2$, $C_3 \geq 1$. Then $\Theta(\bbb,
  C_1, C_2, C_3)$ is the set of all functions $\vartheta : \I_K \to
  \RRnn$ such that there exist functions $A_\p : \ZZnn \to \RRnn$
  satisfying
  \begin{equation*}
    \vartheta(\aaa) = \prod_{\p}A_\p(v_\p(\aaa))
  \end{equation*}
  for all $\aaa \in \I_K$, where 
  \begin{enumerate}
  \item for all $\p$ and $n \geq 1$,
    \begin{equation*}
      |A_\p(n) - A_\p(n-1)| \leq
      \begin{cases}
        C_1&\text{ if }\p^n \mid \bbb,\\
        C_2 \N\p^{-n}&\text{ if }\p^n \nmid \bbb\text;
      \end{cases}
    \end{equation*}
  \item for all $\aaa \in \I_K$, we have $\prod_{\p \nmid \aaa}A_\p(0)
    \leq C_3$.
  \end{enumerate}
  We say that the functions $A_\p$ \emph{correspond to} $\vartheta$.
\end{definition}

The following lemma, which is entirely analogous to \cite[Proposition
6.8]{MR2520770}, describes some elementary properties of the functions defined above.

\begin{lemma}\label{lem:6.8}
  Let $\vartheta \in \Theta(\bbb, C_1, C_2, C_3)$ with corresponding
  functions $A_\p$. Then
  \begin{enumerate}
  \item For any $\aaa \in \I_K$,
    \begin{equation*}
      (\vartheta * \mu_K)(\aaa) = \prod_{p \nmid a}A_\p(0)\prod_{\p | \aaa}(A_\p(v_\p(\aaa)) - A_\p(v_\p(\aaa)-1))\text.
    \end{equation*}
  \item For any $t \geq 0$,
    \begin{equation*}
      \sum_{\N\aaa \leq t}|(\vartheta * \mu_K)(\aaa)|\cdot \N\aaa \ll_{C_2} \tau_K(\bbb)(C_1C_2)^{\omega_K(\bbb)}C_3 t \log(t+2)^{C_2-1}\text.
    \end{equation*}
  \item If $\vartheta$ is not the zero function and $\id q \in \I_K$,
    then the infinite sum and the infinite product
    \begin{equation*}
      \sum_{\substack{\aaa \in \I_K\\\aaa + \id q = \OO_K}}\frac{(\vartheta * \mu_K)(\aaa)}{\N\aaa}\text{ and }\prod_{\p \nmid \id q}\left(\left(1-\frac{1}{\N\p}\right)\sum_{n=0}^\infty\frac{A_\p(n)}{\N\p^n}\right)\prod_{\p \mid \id q}A_\p(0)
    \end{equation*}
    converge to the same real number.
  \end{enumerate}
\end{lemma}
\begin{proof}
  The proof of \cite[Proposition 6.8]{MR2520770} holds almost
  verbatim in our case.
\end{proof}

For $\vartheta \in \Theta(\bbb, C_1, C_2, C_3)$ and $\id q\in \I_K$,
we define
\begin{equation*}
  \A(\vartheta(\aaa),\aaa, \id q) := \sum_{\substack{\aaa \in \I_K\\\aaa + \id q = \OO_K}}\frac{(\vartheta * \mu_K)(\aaa)}{\N\aaa}
\end{equation*}
and $\A(\vartheta(\aaa), \aaa) := \A(\vartheta(\aaa), \aaa,
\OO_K)$. Lemma \ref{lem:6.8}, \emph{(3)}, provides an alternative form. In
the simple case when $\vartheta$ has corresponding functions $A_\p$
satisfying $A_\p(n) = A_\p(1)$ for all prime ideals $\p$ and all $n
\geq 1$, we have
\begin{equation}\label{eq:arith_functions_average_simple}
  \A(\vartheta(\aaa),\aaa, \id q) = \prod_{\p \nmid \id q}\left(\left(1-\frac{1}{\N\p}\right)A_\p(0) + \frac{1}{\N\p}A_\p(1)\right)\prod_{\p \mid \id q}A_\p(0)\text.
\end{equation}
The following Proposition shows that $\A(\vartheta(\aaa), \aaa)$ can
be seen as an average value.

\begin{prop}\label{prop:6.9}
  Let $\idealclass$ be an ideal class of $K$. For $\vartheta \in
  \Theta(\bbb,C_1,C_2,C_3)$, we have
  \begin{equation*}
    \sum_{\substack{\aaa\in\idealclass\cap\I_K\\\N\aaa \le t}}
    \vartheta(\aaa) = \rho_K\A(\vartheta(\aaa),\aaa) t +
    O_{C_2}(\tau_K(\bbb)(C_1C_2)^{\omega_K(\bbb)}C_3 t^{1-1/d})\text,
  \end{equation*}
  for $t \geq 0$.
\end{prop}

\begin{proof}
This follows immediately from Lemma \ref{lem:6.8}, \emph{(1)}, and
Lemma \ref{lem:6.2} below.
\end{proof}

\begin{lemma}\label{lem:3.4}
  Let $C \geq 0$, $c_\vartheta > 0$, and let $\vartheta: \I_K \to \RR_{\geq 0}$ such that, for $t \geq 0$,
  \begin{equation*}
    \sum_{\N\aaa \le t} \vartheta(\aaa) \leq c_\vartheta t(\log(t+2))^C.
  \end{equation*}
  For any $\kappa \in \RR$ and $1 \leq t_1 \leq t_2$, we have
  \begin{equation*}
    \sum_{t_1 \leq \N\aaa \le t_2} \frac{\vartheta(\aaa)}{\N\aaa^\kappa} \ll_{C,\kappa}c_\vartheta \cdot
    \begin{cases}
      t_2^{1-\kappa}(\log(t_2+2))^C &\text{ if }\kappa < 1\text,\\
      \log(t_2+2)^{C+1} &\text{ if }\kappa=1\text,\\
      t_1^{1-\kappa}(\log(t_1+2))^C \ll_{C,\kappa}1 &\text{ if }\kappa>1\text.\\
    \end{cases}
  \end{equation*}
\end{lemma}
\begin{proof}
  Apply \cite[Lemma~3.4]{MR2520770} to $\vartheta'(n)
  :=c_\vartheta^{-1} \sum_{\N\aaa = n}\vartheta(\aaa)$.
\end{proof}

The next lemma is similar to \cite[Lemma 6.2]{MR2520770}, but for
ideals. It completes the proof of Proposition \ref{prop:6.9}.

\begin{lemma}\label{lem:6.2}
  Let $\idealclass$ be an ideal class of $K$, and let $\vartheta: \I_K
  \to \RR$ such that
  \begin{equation*}
    \sum_{\N\aaa \le t} |(\vartheta*\mu_K)(\aaa)|\cdot \N\aaa \ll c_\vartheta t(\log(t+2))^{C}\text,
  \end{equation*}
  for some $C\geq 0$, $c_\vartheta > 0$ and for all $t \geq 0$. Then
  \begin{equation*}
    \sum_{\substack{\aaa \in \idealclass\cap \I_K\\\N\aaa \le t}} \vartheta(\aaa) = \rho_K \sum_{\aaa \in \I_K} \frac{(\vartheta*\mu_K)(\aaa)}{\N\aaa} t +O_C(c_\vartheta t^{1-1/d}).
  \end{equation*}
\end{lemma}

\begin{proof}
  By Lemma \ref{lem:3.4}, $\sum_{\aaa \in \I_K}\frac{(\vartheta *
    \mu_K)(\aaa)}{\N\aaa} \ll_C c_\vartheta$, so the lemma holds for $t <
  1$. Now assume that $t \geq 1$. Since $\vartheta = (\vartheta * \mu_K)
  * 1$, we have
  \begin{equation*}
    \sum_{\substack{\aaa\in\idealclass\cap \I_K\\\N\aaa \le t}} \vartheta(\aaa) = \sum_{\substack{\aaa\in\idealclass\cap\I_K\\\N\aaa \le t}}\sum_{\bbb \mid \aaa}(\vartheta*\mu_K)(\id b) = \sum_{\N\bbb \leq t}(\vartheta*\mu_K)(\bbb)\sum_{\substack{\aaa'\in[\bbb]^{-1}\idealclass\cap \I_K\\\N\aaa' \leq t/\N\bbb}}1\text.
  \end{equation*}
  By the ideal theorem (e.g., \cite[VI, Theorem 3]{MR1282723}), the
  inner sum is $\rho_K t/\N\bbb + O((t/\N\bbb)^{(d-1)/d})$, so our sum
  is equal to
  \begin{equation*}
    \rho_K \sum_{\bbb \in \I_K} \frac{(\vartheta*\mu_K)(\bbb)}{\N\bbb} t + O\left(t \sum_{\N\bbb > t}\frac{|(\vartheta*\mu_K)(\bbb)|}{\N\bbb} + t^\frac{d-1}{d}\sum_{\N\bbb \leq t}\frac{|(\vartheta*\mu_K)(\bbb)|}{\N\bbb^{\frac{d-1}{d}}}\right)\text.
  \end{equation*}
  By Lemma \ref{lem:3.4}, the first part of the
  error term is $\ll_C c_\vartheta(\log(t+2))^C$ and the second part
  is $\ll_C c_\vartheta t^{1-1/d}$.
\end{proof}

We introduce a class of multivariate arithmetic functions, similar to
\cite[Definition~7.8]{MR2520770}. When fixing all variables but one,
these functions are a special case of the ones discussed above.

\begin{definition}\label{def:thetarprime}
  Let $C \geq 1$, $r \in \ZZnn$. Then $\Theta_r'(C)$ is the set of
  all functions $\theta : \I_K^r \to \RR_{\geq 0}$ of the following shape: with
  $J_\p(\aaa_1, \dots, \aaa_r):=\{i \in \{1, \dots, r\} : \p \mid \aaa_i\}$,
  we have
  \begin{equation*}
    \theta(\aaa_1, \dots, \aaa_r) = \prod_{\p} \theta_\p(J_\p(\aaa_1, \dots, \aaa_r)),
  \end{equation*}
  for functions $\theta_{\p} : \{ J \mid J \subset \{1, \ldots, r\}\} \to [0,1]$ with
  \begin{equation*}
    \theta_\p(J) \ge 
    \begin{cases}
      1-C\N\p^{-2} &\text{ if } |J| = 0,\\
      1-C\N\p^{-1} &\text{ if } |J| = 1.
    \end{cases}
  \end{equation*}
\end{definition}

Let $\theta \in \Theta_r'(C)$, fix $\aaa_1$, $\ldots$, $\aaa_{r-1}$,
and let $\vartheta(\aaa_r) := \theta(\aaa_1, \ldots, \aaa_{r-1},
\aaa_r)$. Then the factors $\theta_\p(J_\p(\aaa_1, \dots, \aaa_{r-1},
\aaa_r))$ depend only on $v_\p(\aaa_r)$, and we immediately obtain
$\vartheta(\aaa_r) \in \Theta(\prod_{\p \mid \aaa_1\cdots
  \aaa_{r-1}}\p, 1, C, 1)$. The following result follows immediately
from Proposition \ref{prop:6.9}.

\begin{corollary}\label{cor:6.9}
  Let $\theta \in \Theta_r'(C)$ and $\aaa_1,\ldots, \aaa_{r-1} \in
  \I_K$. For $t \ge 0$, we have
  \begin{equation*}
    \sum_{\N\aaa_r \le t}\theta(\aaa_1, \ldots, \aaa_r) =
    \rho_Kh_K\A(\theta(\aaa_1, \ldots, \aaa_r), \aaa_r) t +
    O_C((2C)^{\omega_K(\aaa_1\cdots\aaa_{r-1})}t^{1-1/d}).
  \end{equation*}
\end{corollary}

By \eqref{eq:arith_functions_average_simple},
\begin{equation*}
  \A(\theta(\aaa_1, \dots, \aaa_r),\aaa_r) = \prod_\p \theta_\p^{(r)}(J_\p(\aaa_1, \dots, \aaa_{r-1}))\text,
\end{equation*}
with
\begin{equation*}
  \theta_\p^{(r)}(J) := \left(1-\frac{1}{\N\p}\right) \theta_\p(J) + \frac{1}{\N\p} \theta_\p(J \cup \{r\}).
\end{equation*}
If $r \geq 1$, we conclude that $\A(\theta(\aaa_1, \dots,
\aaa_r),\aaa_r) \in \Theta_{r-1}'(2C)$. This allows us to define, for
$l \in \{1, \ldots, r\}$,
\begin{equation*}
  \A(\theta(\aaa_1, \dots, \aaa_r),\aaa_r, \dots, \aaa_l) := \A(\cdots\A(\A(\theta(\aaa_1, \ldots, \aaa_r), \aaa_r), \aaa_{r-1}) \cdots, \aaa_l)\text.
\end{equation*}
The following lemma is easily proved by induction (see also
\cite[Corollary 7.10]{MR2520770}).

\begin{lemma}\label{lem:repeated_average}
  Let $\theta \in \Theta_r'(C)$. Then
  \begin{equation*}
    \A(\theta(\aaa_1, \dots, \aaa_{r}),\aaa_{r}, \dots, \aaa_l) = \prod_\p \theta_\p^{(r, \ldots, l)}(J_\p(\aaa_1, \ldots, \aaa_{l-1}))\text,
  \end{equation*}
  where, for $J \subset \{1, \ldots, l-1\}$,
  \begin{equation*}
    \theta_\p^{(r, \ldots, l)}(J) := \sum_{L \subset \{l, \dots, r\}} \left(1-\frac{1}{\N\p}\right)^{r+1-l-|L|}\left(\frac{1}{\N\p}\right)^{|L|} \theta_\p(J \cup L).
  \end{equation*}
  In particular, for $l = 1$,
  \begin{equation}\label{eq:arith_full_average}
    \A(\theta(\aaa_1, \ldots, \aaa_r), \aaa_r, \ldots, \aaa_1) = \prod_{\p}\sum_{L \subset \{1, \ldots, r\}}\left(1-\frac{1}{\N\p}\right)^{r-|L|}\left(\frac{1}{\N\p}\right)^{|L|}\theta_\p(L)\text.
  \end{equation}
\end{lemma}

 For our error estimates, we frequently need the following lemma.
\begin{lemma}\label{lem:omega}
  Let $C \ge 0$. For $t \geq 0$, we have
  \begin{equation*}
    \sum_{\N\aaa \le t} (C+1)^{\omega_K(\aaa)} \ll_C t(\log(t+2))^C.
  \end{equation*}
\end{lemma}

\begin{proof}
  This is clear if $t < 1$, so assume $t \geq 1$.  Write
  $\vartheta(\aaa) := (C+1)^{\omega_K(\aaa)}$. For any $\id p$,
  we have
  \begin{equation*}
    (\vartheta * \mu_K)(\id p^n)=
    \begin{cases}
      1 &\text{ if }n=0\text,\\
      C &\text{ if }n = 1\text,\\
      0 &\text{ if }n \geq 2\text.
    \end{cases}
  \end{equation*}
  Since $\vartheta = (\vartheta * \mu_K) * 1$,
  \begin{equation*}
    \sum_{\N\aaa \leq t}\vartheta(\id a) = \sum_{\N\aaa \leq t}\sum_{\bbb
      \mid \aaa}(\vartheta*\mu_K)(\bbb) \ll t \sum_{\N\bbb \leq t}\frac{(\vartheta*\mu_K)(\bbb)}{\N\bbb}
    \leq t \prod_{\N\id p \leq t}\left(1 + \frac{C}{\N\id p}
    \right)\text,
  \end{equation*}
  where $\id p$ runs over all nonzero prime ideals of $\OO_K$ with
  norm bounded by $t$. By the prime ideal theorem
  (e.g. \cite[Corollary 1 after Proposition 7.10]{MR1055830}) and
  Abelian partial summation, we obtain
  \begin{equation*}
    \prod_{\N\id p \leq t}\left(1 + \frac{C}{\N\id p} \right) \leq
    \exp\left(\sum_{\N\id p \leq t}\frac{C}{\N\id p} \right)\ll_C
    (\log(t+2))^{C}\text.\qedhere
  \end{equation*}
\end{proof}

The following lemma allows us to replace certain sums with
integrals. It is a crucial tool for the results in Sections
\ref{sec:second_summation} and \ref{sec:further_summations}.

\begin{lemma}\label{lem:3.1}
  Let $\idealclass$ be an ideal class of $K$ and $\vartheta: \I_K \to
  \RR$ be a function such that
  \begin{equation}
    \sum_{\substack{\aaa\in\idealclass \cap \I_K\\\N\aaa \le t}} \vartheta(\aaa) - ct \ll \sum_{i=1}^mc_i t^{b_i} \log(t+2)^{k_i}\text,
  \end{equation}
  with $m \in \ZZp$, $c>0$, $c_i$, $b_i \geq 0$, $k_i \in \ZZnn$,
  holds for all $t \geq 0$.

  Let $1 \leq t_1 \leq t_2$, and let $g : [t_1,t_2] \to \RR$ such that
  there exists a partition of $[t_1, t_2]$ into at most $R(g) \geq 1$
  intervals on whose interior $g$ is continuously differentiable and
  monotonic.  Moreover, we assume that there are $a \leq 0$, $c_g \geq
  0$ such that $g(t) \ll c_g t^a$ for all $t \in [t_1, t_2]$.  Then
  \begin{equation}\label{eq:lem_3_1_sum}
    \sum_{\substack{\aaa\in\idealclass\cap\I_K\\t_1 < \N\aaa \le t_2}}
    \vartheta(\aaa)g(\N\aaa) = c\int_{t_1}^{t_2}g(t) \dd t +
    \mathcal{E}(t_1, t_2),
  \end{equation}
  where
  \begin{equation}\label{eq:lem_3_1_errorbound}
    \mathcal{E}(t_1, t_2) \ll_{a, b_i, k_i} R(g)\sum_{i=1}^m
    c_gc_i
    \begin{cases}
      t_2^{b_i}\log(t_2+2)^{k_i} &\text{ if }a = 0,\\
      \sup\limits_{t_1 \le t \le t_2}(t^{a+b_i}\log(t+2)^{k_i}) &\text{ if }a+b_i \ne 0,\\
      \log(t_2+2)^{k_i+1} &\text{ if }a+b_i=0.
    \end{cases}
  \end{equation}
  An analogous formula holds for
  $\sum_{\substack{\aaa\in\idealclass\cap\I_K\\t_1 \leq \N\aaa \le t_2}}
  \vartheta(\aaa)g(\N\aaa)$.
\end{lemma}

\begin{proof}
  For any  $t \in \ZZ \cap [t_1, t_2]$, $\varepsilon \in (0,1)$, we have
  \begin{equation*}
    \sum_{\substack{\aaa\in\idealclass\cap\I_K\\\N\aaa =
        t}}\vartheta(\aaa)= \sum_{\substack{\aaa\in\idealclass \cap
        \I_K\\\N\aaa \le t}} \vartheta(\aaa) -
    \sum_{\substack{\aaa\in\idealclass \cap \I_K\\\N\aaa \le
        t-\varepsilon}} \vartheta(\aaa)\ll c \varepsilon + 
    \sum_{i=1}^mc_i t^{b_i} \log(t+2)^{k_i}\text,
  \end{equation*}
  Letting $\varepsilon \to 0$, we see that the contribution of the
  ideals $\aaa$ with $\N\aaa = t$ is dominated by the error term.

  Hence, it is enough to consider the case $R(g) = 1$ and to assume
  that $g$ is continuously differentiable and monotonic on
  $[t_1,t_2]$. We denote
  \begin{equation*}
    E(t) := \sum_{\substack{\aaa\in\idealclass \cap \I_K\\\N\aaa \le t}} \vartheta(\aaa) - ct
  \end{equation*}
  and start with a similar strategy as in the proof of
  \cite[Lemma~3.1]{MR2520770}. Let $S(t_1, t_2)$ be the sum on the
  left-hand side of (\ref{eq:lem_3_1_sum}). With Abel's summation
  formula and integration by parts, we obtain
  \begin{equation*}
    S(t_1, t_2) =  c\int_{t_1}^{t_2}g(t) \dd t + E(t_2)g(t_2) - E(t_1)g(t_1) - \int_{t_1}^{t_2}E(t)g'(t)\dd t\text.
  \end{equation*}
  By linearity, we may assume that $m = 1$, so $|E(t)| \leq c_1
  t^{b_1}\log(t+2)^{k_1}$. Clearly, the $E(t_i)g(t_i)$ satisfy
  (\ref{eq:lem_3_1_errorbound}). Then
  \begin{equation}\label{eq:lem_3_1_start_estimate}
    \int_{t_1}^{t_2}E(t)g'(t) \dd t \ll c_1\left|\int_{t_1}^{t_2}
      t^{b_1}\log(t+2)^{k_1} g'(t) \dd t\right|\text.
  \end{equation}
  The bound for $a=0$ follows by estimating the integrand by
  $t_2^{b_1}\log(t_2+2)^{k_1}g'(t)$. Moreover, if $b_1=k_1=0$, the
  term on the right-hand side of \eqref{eq:lem_3_1_start_estimate} is
  clearly $\ll c_1 |[g(t)]_{t_1}^{t_2}| \ll c_gc_1 t_1^a$. Otherwise,
  we use integration by parts to further estimate the integral by
  \begin{align*}
    &\ll_{b_1, k_1} c_1\left|[t^{b_1}\log(t+2)^{k_1} g(t)]_{t_1}^{t_2}\right|+c_1\left|\int_{t_1}^{t_2}t^{b_1-1}\log(t+2)^{k_1} g(t) \dd t\right|\\
    & \ll c_g c_1 \sup_{t_1 \le t \le t_2} t^{a + b_1}\log(t+2)^{k_1}
    + c_g c_1 \int_{t_1}^{t_2}t^{a+b_1-1}\log(t+2)^{k_1}\dd t\text.
  \end{align*}
  A simple computation shows that the last integral is
  $\ll \log(t_2+2)^{k_1+1}$ if $a+b_1=0$, and $k_1$-fold
  integration by parts shows that it is $\ll_{a, b_1, k_1}
  |[t^{a+b_1}\log(t+2)^{k_1}]_{t_1}^{t_2}|$ otherwise.
\end{proof}

\section{Lattice points and integrals}\label{sec:lattice_points}
Whenever we talk about integrals or lattices, we identify $\CC$ with
$\RR^2$ via $z \mapsto (\Re z, \Im z)$. For a lattice $\Lambda$ in
$\RR^n$ (by which we mean the $\ZZ$-span of $n$ linearly independent
vectors in $\RR^n$) and a ``nice'' bounded subset $S \subset \RR^n$,
one usually approximates $|\Lambda \cap S|$ by the quantity
$\vol(S)/\det(\Lambda)$. To make this precise, we need to define
``nice'' sets in our context. We follow an approach developed by
Davenport \cite{MR0043821} and Schmidt \cite{MR1330740}. For a
comparison with a different approach using
Lipschitz-parameterizability, see \cite{MR2846337}.

\begin{definition}[{\cite[p. 347]{MR1330740}}]\label{def:class_m}
  A compact subset $S \subset \RR^n$ is \emph{of class $m$} if every
  line intersects $S$ in at most $m$ single points and intervals and
  if the same holds for all projections of $S$ on all linear subspaces.
\end{definition}

In particular, the sets of class $1$ are the compact convex sets. In
our applications, we consider sets as in the following lemma.

\begin{lemma}\label{lem:preimage_of_unit_disc}
  Let $l$, $D \in \ZZp$. For $j \in \{1,\ldots, l\}$, let $f_j$, $g_j \in
  \CC[X]$ be polynomials of degree at most $D$, and let $\prec_j \in
  \{\leq, =\}$. Moreover, assume that the set
  \begin{equation*}
    S := \{z \in \CC \mid \abs{f_j(z)} \prec_j \abs{g_j(z)} \text{ for all }1 \leq j \leq l\}
  \end{equation*}
  is bounded. Then $S$ is of class $m$, for some effective constant
  $m$ depending only on $l$ and $D$.
\end{lemma}

\begin{proof}
  The set $S$ is clearly closed, so it is compact. Write $z = x + i
  y$, with $x$, $y \in \RR$. Then $S$ is defined by the polynomial
  (in)equalities $h_j(x, y) \prec_j 0$, $1 \leq j \leq l$, with
  \begin{equation*}
    h_j(X, Y) := f_j(X+iY)\overline{f_j}(X-iY) - g_j(X+iY)\overline{g_j}(X-iY),
  \end{equation*}
  where $\overline{\vphantom{g}\ }$ denotes complex conjugation of the
  coefficients. Hence, $h_j \in \RR[X, Y]$ and $\deg h_j \leq 2D$. We
  conclude that $S$ has $O_{l,D}(1)$ connected components (see
  e.g. \cite[Proposition 4.13]{coste2000introduction}). Therefore,
  every projection of $S$ to a linear subspace has $O_{l,D}(1)$
  connected components, that is single points and intervals.

  The intersection of $S$ with a line is defined by the (in)equalities
  $h_j(x,y)\prec_j 0$ and a linear equality, so once again it has
  $O_{l,D}(1)$ connected components, that is single points and
  intervals.
\end{proof}

Let $K \subset \CC$ be an imaginary quadratic field, and let $S
\subset \CC$ be as in Lemma \ref{lem:preimage_of_unit_disc}. We use
the following lemma, inspired by \cite[Lemma 1]{MR1330740}, to count
the elements of a given fractional ideal of $K$ that lie in $S$.

\begin{lemma}\label{lem:lattice_points}
  Let $\id a$ be a fractional ideal of an imaginary quadratic field
  $K\subset \CC$, let $\beta \in K$, and let $S \subset \CC$ be a
  subset of class $m$ that is contained in the union of $k$ closed
  balls $B_{p_i}(R)$ of radius $R$, centered at arbitrary points $p_i
  \in \CC$. Then
  \begin{equation*}
    |(\beta + \id a) \cap S| = \frac{2 \vol(S)}{\sqrt{|\Delta_K|}\N\id a} + O_{m,k}\left(\frac{R}{\sqrt{\N\id a}} + 1\right)\text.
  \end{equation*}
\end{lemma}

\begin{proof}
  After translation by $-\beta$, we may assume that $\beta = 0$. The
  ideal $\aaa$ is a lattice in $\CC$ of determinant $\det \aaa =
  2^{-1}\sqrt{|\Delta_K|}\N \id a$. Denote its successive minima (with
  respect to the unit ball) by $\lambda_1 \leq \lambda_2$. Then
  $\lambda_1 \geq \sqrt{\N \id a}$ (see e.g. \cite[Lemma
  5]{MR2247898}). By \cite[Lemma VIII.1, Lemma V.8]{MR1434478}, there
  is a basis $\{u_1, u_2\}$ of $\id a$ with $|u_j| = \lambda_j$. Let
  $\psi : \CC \to \CC$ be the linear automorphism given by $\psi(u_1)
  = 1$, $\psi(u_2) = i$. Then $\psi(\aaa) = \ZZ[i]$ and, with respect to
  the standard basis, $\psi$ is represented by the matrix
  \begin{equation*}
    \frac{1}{\det \aaa}
    \begin{pmatrix}
      \Im u_2 & -\Re u_2\\
      -\Im u_1 & \Re u_1
    \end{pmatrix}\text,
  \end{equation*}
  so its operator norm $|\psi|$ is bounded by $2\lambda_2/\det\aaa$. By
  Minkowski's second theorem and the facts from the beginning of this
  proof, we obtain $|\psi| \ll 1/\sqrt{\N\aaa}$.

  Clearly, $|\aaa \cap S| = |\ZZ[i] \cap \psi(S)|$, and $\psi(S)$ is
  still of class $m$. In particular, it satisfies the conditions
  I. and II. from \cite{MR0043821}, so by \cite[Theorem]{MR0043821},
  \begin{equation*}
    |\ZZ[i] \cap \psi(S)| = \vol\psi(S) + O( m V_1 + m^2)\text, 
  \end{equation*}
  where $V_1$ is the sum of the volumes of the projections of
  $\psi(S)$ to $\RR$ and $i\RR$. Since $\det \psi = 1/\det\aaa$, the
  main term is as claimed in the lemma. Since $\psi(S) \subset
  \bigcup_i \psi(B_{p_i}(R))$, the volume of the projection of $\psi(S)$
  to $\RR$ or $i\RR$ is bounded by
  \begin{equation*}
    \sum_{1 \leq i \leq k}\diam(\psi(B_{p_i}(R))) \leq \sum_{1 \leq i \leq k}|\psi|\diam(B_{p_i}(R)) \ll \frac{k R}{\sqrt{\N\aaa}}\text.\qedhere
  \end{equation*}
\end{proof}

For meaningful applications of Lemma \ref{lem:lattice_points} to a set
$S$ as in Lemma \ref{lem:preimage_of_unit_disc}, we need $R$ to be
sufficiently small. The following two lemmas provide such values of
$R$ for certain sets $S$ and list some consequences analogous to
\cite[Lemma 5.1, (4)--(6)]{MR2520770} and \cite[Lemma 5.1,
(1)--(3)]{MR2520770}. For positive $x$, $y$, we interpret the
expression $\min\{x, y/0\}$ as $x$.
\begin{lemma}\label{lem:5.1_4}
  Let $a \in \CC\smallsetminus\{0\}$, $b \in \CC$, $k > 1$. With $R :=
  \min\{|a|^{-1/2}, 2|b|^{-1}\}$, we have
  \begin{enumerate}
  \item $\{z \in \CC \mid \abs{a z^2 + bz} \leq 1\} \subset B_0(R)
    \cup B_{-b/a}(R)$,
  \item $\vol\{z \in \CC \mid \abs{a z^2 + bz} \leq 1\} \ll R^2 \ll
    \min\{\abs{a}^{-1/2}, \abs{b}^{-1}\}$.
  \end{enumerate}
  If additionally $b \neq 0$, we have
  \begin{enumerate}
  \item[(3)] $\vol\{(z, u) \in \CC^2 \mid \abs{a z^2 +
      bzu^k}\leq 1 \}\ll \abs{a}^{-(k-1)/(2k)}\abs{b}^{-1/k}$,
  \item[(4)] $\vol\{(z, u) \in \CC^2 \mid \abs{a z^2 u +
      bzu^2}\leq 1 \}\ll \abs{ab}^{-1/3}$,
  \item[(5)] $\vol\{(z, t) \in \CC \times \RR_{\ge 0} \mid \abs{a z^2 +
      bzt^{k/2}}\leq 1 \}\ll \abs{a}^{-(k-1)/(2k)}\abs{b}^{-1/k}$,
  \item[(6)] $\vol\{(z, t) \in \CC \times \RR_{\ge 0} \mid \abs{a z^2 t^{1/2} +
      bzt}\leq 1 \}\ll \abs{ab}^{-1/3}$.
  \end{enumerate}
\end{lemma}

\begin{proof}
  For (1), we note that $|z||z+b/a| \leq |a|^{-1}$ implies
  \begin{equation*}
    z \in B_0(|a|^{-1/2}) \cup B_{-b/a}(|a|^{-1/2}).
  \end{equation*}
  Suppose now that $b \neq 0$, $|az^2 + bz| \leq 1$, $|b||z|>2$ and
  $|b||a z + b| > 2 |a|$ hold. Then
  \begin{equation*}
    |b||a z + b| > 2 |a| \geq 2 |a||z||az + b|,
  \end{equation*}
  so $|b| > 2|a z|$ and thus $|az+b| > |az|$. This in turn implies
  that $|az^2| < 1$, so $|az^2 + bz| > 2 - 1 > 1$, a
  contradiction. This proves (1) and (2). The volume in (3) is
  \begin{align*}
    &\ll \int_{u \in \CC}\min\{\abs{a}^{-1/2}, \abs{bu^k}^{-1}\}\dd u\\
    &\ll \int_{\abs{u}\leq (\abs{a}^{1/2}\abs{b}^{-1})^{1/k}}\abs{a}^{-1/2}\dd u + \int_{\abs{u}>(\abs{a}^{1/2}\abs{b}^{-1})^{1/k}}\abs{bu^k}^{-1}\dd u\\
    &\ll \abs{a}^{-(k-1)/(2k)}\abs{b}^{-1/k}\text.
  \end{align*}
  The proof of (4) is another elementary computation similar to the
  proof of (3), and (5), (6) are analogous to (3), (4).
\end{proof}

\begin{lemma}\label{lem:5.1_1}
  Let $a \in \CC\smallsetminus \{0\}$, $b \in \CC$, $k > 1$. With $R
  := \min\{|a|^{-1/2}, |a b|^{-1/2}\}$, we have
  \begin{enumerate}
  \item $\{z \in \CC \mid \abs{a z^2 - b} \leq 1\} \subset
    B_{\sqrt{b/a}}(R) \cup B_{-\sqrt{b/a}}(R)$,
  \item $\vol\{z \in \CC \mid \abs{a z^2 - b} \leq 1\} \ll R^2 \leq
    \min\{\abs{a}^{-1/2}, \abs{ab}^{-1/2}\}$,
 
  \end{enumerate}
If additionally $b \neq 0$, we have
\begin{enumerate}
  \item[(3)] $\vol\{(z, u) \in \CC^2 \mid
    \abs{az^2 - bu^k} \leq 1\} \ll \abs{a}^{-1/2}\abs{b}^{-1/k}$ if $k > 2$,
  \item[(4)] $\vol\{(z, u) \in \CC^2 \mid \abs{az^2u -
      bu^k} \leq 1\} \ll (\abs{a}\abs{b}^{1/k})^{-1/2}$,
  \item[(5)] $\vol\{(z, t) \in \CC \times \RR_{\ge 0} \mid
    \abs{az^2 - bt^{k/2}} \leq 1\} \ll \abs{a}^{-1/2}\abs{b}^{-1/k}$ if $k > 2$,
  \item[(6)] $\vol\{(z, t) \in \CC \times \RR_{\ge 0} \mid \abs{az^2t^{1/2} -
      bt^{k/2}} \leq 1\} \ll (\abs{a}\abs{b}^{1/k})^{-1/2}$.
  \end{enumerate}
\end{lemma}
\begin{proof}
  Using the substitution $t = z - \sqrt{b/a}$, (1) is an immediate
  consequence of Lemma \ref{lem:5.1_4}, (1). Moreover, (2) follows
  from (1), and (3)--(6) follow from (2) similarly to Lemma
  \ref{lem:5.1_4}.
\end{proof}

The following lemma provides an easy way to prove uniform boundedness of
quantities such as $R(V_{\yy})$ in Lemma \ref{lem:3.1}, for families
$V_\yy$ of certain volume functions. This is relevant for applications
of our methods from Sections \ref{sec:second_summation} and
\ref{sec:further_summations}. We use the language of semialgebraic geometry
(see, e.g., \cite{coste2000introduction}). The proof uses o-minimal
structures, as presented in \cite{MR1633348}.

\begin{lemma}\label{lem:omin}
  Let $k, n \in \ZZnn$, let $M \subset \RR^k \times \RR \times \RR^n$ be a semialgebraic set, and let $f : M \to \RR$ be a semialgebraic function. Assume that for all $\yy = (y_1, \ldots, y_k) \in \RR^k$, $t \in \RR$, the function $f(\yy, t, \cdot)$ is integrable on the fiber
  \begin{equation*}
    M_{\yy,t} := \{\xx=(x_1, \ldots, x_n) \in \RR^n \mid (y_1, \ldots, y_k, t, x_1, \ldots, x_n) \in M\}\text.
  \end{equation*}
Then there exists a constant $C \in\ZZp$, such that for all $\yy \in \RR^k$ there is a partition of $\RR$ into at most $C$ intervals on whose interior the function $V_{\yy} : \RR \to \RR$ defined by
  \begin{equation*}
    V_{\yy}(t) := \int_{\xx \in M_{\yy,t}} f \dd \xx
  \end{equation*}
is continuously differentiable and monotonic.
\end{lemma}

\begin{proof}
  The function $V : \RR^k \times \RR \to \RR$, $(\yy,t)\mapsto V_{\yy}(t)$ is definable in an o-minimal structure. Indeed, by \cite{MR1644093}, parametric integrals of global subanalytic functions are definable in the expansion $(\RR_{\text{an}}, \text{exp})$ of the structure of global subanalytic sets $\RR_{\text{an}}$ by the global exponential function, which is o-minimal. (In \cite{kaiserMZ}, a smaller structure is constructed which is sufficient for parametric integrals of semialgebraic functions.) 

Let $\mathcal{D}$ be a decomposition of $\RR^k \times \RR$ into $C^1$-cells such that the restriction of $V$ to each cell $D$ of $\mathcal{D}$ is $C^1$ (\cite[Theorem 7.3.2]{MR1633348}).

For each cell $D$ of $\mathcal{D}$, there is a definable open set $D \subset U_D \subset \RR^k\times\RR$ and a definable $C^1$ function $V_D : U_D \to \RR$ such that $V_D|_D = V|_D$. Let $\mathcal{E}$ be a decomposition of $\RR^k \times \RR$ into $C^1$-cells partitioning the definable sets
\begin{equation*}
  A_D^+:=\{(\yy, t) \in D \mid \partial V_D/\partial t \geq 0\} \text{ and } A_D^-:=\{(\yy,t)\in D \mid \partial V_D/\partial t \leq 0\} \text{ for }D \in \mathcal{D}\text.
\end{equation*}
We note that $\bigcup_{D}(A_D^+ \cup A_D^-) = \RR^k\times\RR$, so each cell $E$ of $\mathcal{E}$ is contained in some $A_D^+$ or $A_D^-$.

Let $\pi : \RR^k \times \RR \to \RR^k$ be the projection on the first $k$ coordinates. Let $\yy \in \RR^k$. For cells $E$ of $\mathcal{E}$ with $\yy \in \pi(E)$, the sets $E_\yy := \{t \in \RR \mid (\yy, t) \in E\}$ are the cells of a decomposition $\mathcal{E}_\yy$ of $\RR$ (\cite[Proposition 3.3.5]{MR1633348}). On cells $E_\yy$ that are open intervals, $V_\yy'(t)$ is defined and coincides with $\partial V_D/\partial t(\yy,t)$ (if $E \subset A_D^+$ or $E\subset A_D^-$). Therefore, $V_\yy$ is continuously differentiable and monotonic on $E_\yy$. The observation that $|\mathcal{E}_\yy| \leq |\mathcal{E}|$ completes our proof.
\end{proof}

\section{Passage to a universal torsor}\label{sec:passage}
In this section, we describe a strategy to parameterize rational
points on a split singular del Pezzo surface by integral points on a
universal torsor. This generalizes \cite[\S 4]{MR2290499} from $\QQ$
to imaginary quadratic fields with arbitrary class number. In \cite[\S
4]{arXiv:1105.2807}, \cite{arXiv:1204.0383}, a similar strategy is
used in the easier case of a toric split singular cubic surface, where
a universal torsor is an open subset of affine space.

Let $K$ be a number field. Let $S$ be a non-toric split singular del
Pezzo surface defined over $K$ whose minimal desingularization $\tS$
has a universal torsor that is an open subset of a hypersurface in
affine space. Up to isomorphism, there are only finitely many del
Pezzo surfaces satisfying these properties. Together with an explicit
description of all their properties used below, their classification
can be found in \cite{math.AG/0604194}. For del Pezzo surfaces with
more complicated universal torsors, we expect that a similar strategy
can be used, but that several complications may appear.

We assume for simplicity that $\deg(S) \in \{3, \dots, 6\}$; the
adaptation to $\deg(S) \in \{1,2\}$ is straightforward. To count
$K$-rational points on $S$, we use the Weil height given by an
anticanonical embedding $S \subset \PP^{\deg(S)}_K$ satisfying the
following assumptions.
\begin{itemize}
\item Let $r:=9-\deg(S)$. By our assumption on a universal torsor of
  $\tS$, its Cox ring $\Cox(\tS)$ has a minimal system of $r+4$
  generators $\te_1, \dots, \te_{r+4}$ that are homogeneous (with
  respect to the natural $\Pic(\tS)$-grading of $\Cox(\tS)$), are
  defined over $K$ (since $S$ is split), correspond to curves $E_1,
  \ldots, E_{r+4}$ on $\tS$, and satisfy one homogeneous relation
  \begin{equation}\label{eq:torsor}
    R(\te_1, \dots, \te_{r+4})=0,
  \end{equation}
  which we call \emph{torsor equation}. Possibly after replacing
  some $\te_i$ by scalar multiples, we may assume that all coefficients
  in $R$ are $\pm 1$.
\item The choice of a basis $s_0, \dots, s_{\deg(S)}$ of $H^0(\tS,
  \OO(-K_\tS))$ defines a map $\pi: \tS \to \PP^{\deg(S)}_K$ whose image
  is an anticanonical embedding $S \subset \PP^{\deg(S)}_K$. Since
  $H^0(\tS, \OO(-K_\tS)) \subset \Cox(\tS)$, we may choose each $s_i$
  as a monic monomial
  \begin{equation}\label{eq:anticanonical_section}
    \Psi_i(\te_1, \dots, \te_{r+4})
  \end{equation}
  in the generators of $\Cox(\tS)$, for $i=0, \dots, \deg(S)$.
\end{itemize}

To describe our expected parameterization of $K$-rational points of
bounded height on $S$ in Claims~\ref{claim:passage} and
\ref{claim:passage_step} below, we introduce the following notation.
\begin{itemize}
\item The split generalized del Pezzo surface $\tS$ is a blow-up
  $\rho: \tS \to \PP^2_K$ in $r$ points \emph{in almost general
    position}, i.e., a composition of $r$ blow-ups
  \begin{equation}\label{eq:blow_up_sequence}
    \tS = \tS_r \xrightarrow{\rho_r} \tS_{r-1} \to \dots \to \tS_1 \xrightarrow{\rho_1} \tS_0 = \PP^2_K,
  \end{equation}
  where each $\rho_i : \tS_i \to \tS_{i-1}$ is the blow-up of a point
  $p_i$ not lying on a $(-2)$-curve on $\tS_{i-1}$. Let $\ell_0$ be
  the class of $\rho^*(\OO_{\PP^2_K}(1))$ and $\ell_i$ the class of the
  total transform of the exceptional divisor of $\rho_i$, for $i=1,
  \dots, r$. Then $\ell_0, \dots, \ell_r$ form a basis of $\Pic(\tS)$,
  so
  \begin{equation}\label{eq:divisor_classes}
    [E_j] = a_{j,0}\ell_0+\dots+a_{j,r}\ell_r \in \Pic(\tS)
  \end{equation}
  for some $a_{j,i} \in \ZZ$, for $j=1, \dots, r+4$.

  For any $\classtuple = (C_0, \dots, C_r) \in \classrep^{r+1}$ (see
  Section~\ref{sec:notation}), we use the integers $a_{j,i}$ to define
  the fractional ideals
  \begin{equation}\label{eq:ideals}
    \OO_j := C_0^{a_{j,0}}\cdots C_r^{a_{j,r}},
  \end{equation}
  and their subsets
  \begin{equation*}
    \OO_{j*}:=
    \begin{cases}
      (\OO_j)^{\ne 0} , &([E_j],[E_j])<0,\\
      \OO_j, &\text{otherwise.}
    \end{cases}
  \end{equation*}
\item For $\e_j \in \OO_j$, consider the ideals
  \begin{equation*}
    \eI_j:=\e_j\OO_j^{-1}.
  \end{equation*}
  Via the configuration of $E_1, \dots, E_{r+4}$, we define
  \emph{coprimality conditions}
  \begin{equation}\label{eq:coprimality}
    \sum_{j \in J} \eI_j=\OO_K\text{ for all minimal }J
    \subset \{1, \dots, r+4\}\text{ with }\bigcap_{j \in J}
    E_j=\emptyset.
  \end{equation}
  We observe from the classification in \cite{math.AG/0604194} that
  these minimal $J$ have the form $J = \{j,j'\}$ for non-intersecting
  $E_j, E_{j'}$ (encoded in the extended Dynkin diagram) or $J =
  \{j,j',j''\}$ for pairwise intersecting $E_j, E_{j'}, E_{j''}$ that do not
  meet in a common point.

\item Assume that $K$ is an imaginary-quadratic field or $K=\QQ$. We
  consider $K$ as a subset of $K_\infty \in \{\RR, \CC\}$, its
  completion at the infinite place, with $\abs{\cdot}$ the usual real
  absolute value resp.\ the square of the usual complex one.

  Let $\R(B)$ be the set of all $(\e_1, \dots, \e_{r+4}) \in
  K_\infty^{r+4}$ satisfying the \emph{height conditions}
  \begin{equation}\label{eq:height_condition}
    \abs{\Psi_i(\e_1, \dots, \e_{r+4})} \le B,
  \end{equation}
  for $i=0, \dots, \deg(S)$, where $\Psi_i$ is the monic monomial from
  \eqref{eq:anticanonical_section}.

  For any $\classtuple \in \classrep^{r+1}$, we define
  $u_\classtuple:=\N(C_0^3C_1^{-1}\cdots C_r^{-1})$,
  corresponding to the anticanonical class
  $[-K_\tS]=3\ell_0-\ell_1-\dots-\ell_r$.  Let
  $M_\classtuple(B)$ be the set of all
  \begin{equation*}
    (\e_1, \dots, \e_{r+4}) \in
    \OO_{1*} \times \dots \times \OO_{r+4*}
  \end{equation*}
  lying in the set $\R(u_\classtuple B)$ defined by the height conditions and
  satisfying the torsor equation \eqref{eq:torsor}, the coprimality conditions
  \eqref{eq:coprimality}.
\end{itemize}

\begin{claim}\label{claim:passage}
  Let $K$ be an imaginary quadratic field or $K=\QQ$. Let $S\subset
  \PP_K^{\deg(S)}$ be a split singular del Pezzo surface of degree $3,
  \dots, 6$ over $K$ whose universal torsors are open subsets of
  hypersurfaces, with an anticanonical embedding satisfying the
  assumptions above. Let $U$ be the complement of its lines. Let
  $N_{U, H}(B)$ be defined as in \eqref{eq:def_NUH}, with the usual
  Weil height $H$ on $\PP_K^{\deg(S)}(K)$. With the notation
  introduced above, for $B > 0$, we have
  \begin{equation*}
    N_{U,H}(B) = \frac{1}{\omega_K^{10-\deg(S)}} \sum_{\classtuple \in
      \classrep^{10-\deg(S)}} |M_\classtuple(B)|.
  \end{equation*}
\end{claim}

Motivated by the geometry of $S$, we propose a strategy to prove
Claim~\ref{claim:passage} by induction, via the closely related
Claim~\ref{claim:passage_step} below, for $i=0, \dots, r$.
The starting point is a parameterization of rational points via the
birational map $\pi \circ \rho^{-1} : \PP^2_K \rto S$. In each step
$i=1, \dots, r$, the rational points are parameterized by variables
$\e_j$ corresponding to curves on $\tS_{i-1}$; if $\rho_i$ is the
blow-up of the intersection point of some of these curves, we
introduce a new variable essentially as the greatest common divisor of
the variables corresponding to those curves to obtain the next step of
the parameterization.

From here on, we work again over an arbitrary number field $K$.
To set up the induction in Claim~\ref{claim:passage_step}, we need
more notation. For $i=0, \dots, r$ and $j=1, \dots, r+4$, let
$E_j^{(i)} := (\rho_{i+1}\circ \dots \circ \rho_r)(E_j)$ be the
projection of $E_j$ on $\tS_i$. If $E_j^{(i)}$ is a curve on $\tS_i$,
then $E_j$ is its strict transform on $\tS$. Possibly after
rearranging the generators of $\Cox(\tS)$, we may assume that
$E_1^{(0)}, E_2^{(0)}, E_3^{(0)}$ are lines in $\PP^2_K$, that
$E_4^{(0)}$ is a curve of some degree $D$ in $\PP^2_K$, and that
$E_{i+4}^{(i)}$ is the exceptional divisor of $\rho_i$, so
\begin{equation}\label{eq:coefficients}
  a_{1,0}=a_{2,0}=a_{3,0}=1,\ a_{4,0}=D,\ a_{i+4,0}=\dots=a_{i+4,i-1}=0,\ a_{i+4,i}=1,
\end{equation}
for $i=1, \dots, r$. By \cite[Lemma~12]{math.AG/0604194}, we may
assume (possibly by a linear change of coordinates $y_0,y_1,y_2$ on
$\PP^2_K$) that
\begin{equation*}
  E_1^{(0)}=\{y_0=0\}, \ E_2^{(0)}=\{y_1=0\}, \ E_3^{(0)}=\{y_2=0\}, \ E_4^{(0)}=\{R'(y_0,y_1,y_2)=0\}
\end{equation*}
in $\PP^2_K$, where $R'$ is a homogeneous polynomial of degree $D$
satisfying
\begin{equation}\label{eq:rprime}
  Y_3-R'(Y_0,Y_1,Y_2)=R(Y_0, \dots, Y_3,1, \dots, 1).
\end{equation}

Via the natural embeddings $\Pic(\PP^2_K) \subset \Pic(\tS_1)
\subset \dots \subset \Pic(\tS_{r-1}) \subset \Pic(\tS)$, we may view $\ell_0,
\dots, \ell_i$ as a basis of $\Pic(\tS_i)$. Then
\begin{equation}\label{eq:divisor_classes_i}
  [E_j^{(i)}] = a_{j,0}\ell_0+\dots+a_{j,i}\ell_i \in \Pic(\tS_i)
\end{equation}
with the integers $a_{j,i}$
from \eqref{eq:divisor_classes}, for any $i=0, \dots, r$ and $j=1, \dots, i+4$.

For $i=0, \dots, r$ and any $(C_0, \dots, C_i) \in \classrep^{i+1}$, we define
analogously to \eqref{eq:ideals}
\begin{equation*}
  \OO_j^{(i)} := C_0^{a_{j,0}}\cdots C_i^{a_{j,i}}, \qquad 
  \OO_{j*}^{(i)}:=
  \begin{cases}
    (\OO_j^{(i)})^{\ne 0} , &([E_j],[E_j])<0,\\
    \OO_j^{(i)}, &\text{otherwise,}
  \end{cases}
\end{equation*}
and, for $\e_j \in \OO_j^{(i)}$,
\begin{equation*}
  \eI_j^{(i)} := \e_j(\OO_j^{(i)})^{-1}
\end{equation*}
for $j=1,\dots, i+4$.

We use the monomials $\Psi_i(\te_1, \dots, \te_{r+4})$ from
\eqref{eq:height_condition} to define the map
\begin{equation*}
  \Psi : K^{r+4} \to K^{\deg(S)+1},\quad (\e_1, \dots, \e_{r+4}) \mapsto (\Psi_i(\e_1, \dots, \e_{r+4}))_{i=0, \dots, \deg(S)}.
\end{equation*}

\begin{claim}\label{claim:passage_step}
  Let $K$ be a number field. Assume that $U \subset S \subset
  \PP^{\deg(S)}_K$ are as in Claim~\ref{claim:passage}. Assume
  that $\te_1,\ldots,\te_{r+4}$ are ordered in such a way that
  $E_{i+4}^{(i)}$ is the exceptional divisor of $\rho_i$, for $i=1,
  \dots, r$. For any $i \in \{0, \dots, r\}$, we have a map $(\e_1,
  \ldots, \e_{i+4})\mapsto\Psi(\e_1, \dots, \e_{i+4}, 1, \dots, 1)$
  from the disjoint union
  \begin{equation*}
    \bigcup_{C_0, \dots, C_i \in \classrep} \left\{
      \begin{aligned}
        &(\e_1, \dots, \e_{i+4}) \in \OO_{1*}^{(i)} \times \dots \times
        \OO_{i+4*}^{(i)} :
        R(\e_1, \dots, \e_{i+4}, 1, \dots, 1)=0,\\
        &\sum_{j \in J} \eI_j^{(i)}=\OO_K\text{ for all minimal }J
        \subset \{1, \dots, i+4\}\text{ with }\bigcap_{j \in J}
        E_j^{(i)}=\emptyset
      \end{aligned}
    \right\}
  \end{equation*}
  to $U(K)$. This induces a bijection between the orbits under the natural free
  action of $(\OO_K^\times)^{i+1}$ on the former set and $U(K)$.
\end{claim}

Here, the natural action of $(\lambda_0, \dots, \lambda_i) \in
(\OO_K^\times)^{i+1}$ on these subsets of $K^{i+4}$ is explicitly given via the $\Pic(\tS_i)$-degrees of $\te_1,\ldots,\te_{i+4}$
\eqref{eq:divisor_classes_i}:
\begin{equation*}
(\lambda_0, \dots, \lambda_i)\cdot (\e_1, \ldots, \e_{i+4}) := (\lambda_0^{a_{1,0}}\cdots\lambda_i^{a_{1,i}}\e_1, \ldots, \lambda_0^{a_{i+4,0}}\cdots\lambda_i^{a_{i+4,i}}\e_{i+4}).
\end{equation*}
Freeness of this action follows immediately from
\eqref{eq:coefficients}, the assumption that $E_j$ is a negative curve
for all $j \in \{5,\ldots,r+4\}$, and the fact that there are at least
$r+1$ negative curves on any generalized del Pezzo surface of degree
$\leq 7$. Also, $\Psi$ induces a well-defined map on the orbits
because all $\Psi_j(\te_1, \dots, \te_{i+4}, 1, \dots, 1)$ have the
same degree $[-K_{\tS_i}]$.

Assume we have established Claim~\ref{claim:passage_step} for $i=r$. To deduce
Claim~\ref{claim:passage} in specific cases over number fields $K$ with finite $\OO_K^\times$, it
remains to lift the height function via $\Psi$. By the definition of the Weil
height as in \eqref{eq:height}, this depends essentially on the norm of
\begin{equation*}
  \Psi_0(\e_1, \dots, \e_{r+4})\OO_K+\dots+\Psi_{\deg(S)}(\e_1, \dots, \e_{r+4})\OO_K.
\end{equation*}
For $(\e_1, \dots, \e_{r+4}) \in \OO_{1*} \times \dots \times
\OO_{r+4*}$, this is a multiple of $u_\classtuple =
\N(C_0^3C_1^{-1}\cdots C_r^{-1})$ since the $\Psi_i(\te_1, \ldots,
\te_{r+4})$ have degree $[-K_\tS]=3\ell_0-\ell_1-\dots-\ell_r$. We
expect that it is indeed equal to $u_\classtuple$ under
\eqref{eq:coprimality}. Then $H(\Psi(\e_1, \dots, \e_{r+4})) \le B$ if
and only if $(\e_1, \dots, \e_{r+4}) \in \R(u_\classtuple B)$, and
Claim~\ref{claim:passage} follows.

The following two lemmas turns out to be sufficient to prove
Claim~\ref{claim:passage_step} for the quartic surface of type
$\Athree$ with five lines defined by \eqref{eq:def_S}. For other
surfaces, some induction steps must be done by hand. In particular, it may be
necessary to use the relation $R$ to deduce the new set of coprimality
conditions. We note that the assumption on $\psi$ in the first lemma holds for
every example in \cite{math.AG/0604194}.

\begin{lemma}\label{lem:passage_start}
  The birational map $\pi \circ \rho^{-1} : \PP^2_K \rto S$ induces an
  isomorphism between an open subset $V \subset \PP^2_K$ and $U
  \subset S$. The homogeneous cubic polynomials
  \begin{equation}\label{eq:psii}
    \psi_i(Y_0,Y_1,Y_2) := \Psi_i(Y_0,Y_1,Y_2,R'(Y_0,Y_1,Y_2),1,\dots, 1),
  \end{equation}
  for $i=0, \dots, \deg(S)$, define a rational map 
  \begin{equation}\label{eq:psi}
    \psi: \PP^2_K \rto S,\quad (y_0:y_1:y_2) \mapsto (\psi_0(y_0,y_1,y_2): \dots: \psi_{\deg(S)}(y_0,y_1,y_2)).
  \end{equation}
  If $\psi$ represents $\pi \circ \rho^{-1}$ on $V$, then
  Claim~\ref{claim:passage_step} holds for $i=0$.
\end{lemma}

\begin{proof}
  Let $V \subset \PP^2_K$ be the complement of all $E_j^{(0)}$ with $j
  \in \{1, \dots, 4\}$ such that $E_j$ is a negative curve on
  $\tS$. Let $W$ be the complement of the negative curves on
  $\tS$. Then $\pi(W)=U$ since $\pi$ maps the $(-1)$-curves to the
  lines and the $(-2)$-curves to the singularities on $S$ (each lying
  on a line for any singular del Pezzo surface except for the
  Hirzebruch surface $F_2$, which is excluded since it is toric), and
  $\rho(W)=V$ since $\rho$ contracts the negative curves $E_5, \dots,
  E_{r+4}$ to points lying on the negative curves among $E_1^{(0)},
  \dots, E_4^{(0)}$ (since the extended Dynkin diagram of negative
  curves on $\tS$ is connected and there are at least $r+1$ negative
  curves). Therefore, the birational map $\pi \circ \rho^{-1}$
  induces an isomorphism between $V$ and $U$.

  For $i=0, \dots, \deg(S)$, we note that $\psi_i$ is a cubic
  polynomial, by considering coefficients $(a_{1,0}, \dots, a_{r+4,0})
  = (1,1,1,D,0,\dots,0)$ of $\ell_0$ from \eqref{eq:coefficients} and
  the degree of $\Psi_i$.  Since $\Psi_i$ are monomials, $\psi$ is
  defined at least on the complement of $E_1^{(0)}, \dots,
  E_4^{(0)}$. Its image lies in $S$ since for any equation $F \in
  K[X_0, \dots, X_{\deg(S)}]$ defining $S \subset \PP^{\deg(S)}_K$, we
  know that $F(\Psi_0(\te_1, \dots, \te_{r+4}), \dots,
  \Psi_{\deg(S)}(\te_1, \dots, \te_{r+4}))$ is a multiple of $R(\te_1,
  \dots, \te_{r+4})$, so that $F(\psi_0(Y_0,Y_1,Y_2), \dots,
  \psi_{\deg(S)}(Y_0,Y_1,Y_2))$ is a multiple of
  $R(Y_0,Y_1,Y_2,R'(Y_0,Y_1,Y_2),1, \dots, 1)$, which is trivial by
  \eqref{eq:rprime}.

  To prove Claim~\ref{claim:passage_step} for $i=0$, we note that $\pi
  \circ \rho^{-1}$ induces a bijection between $V(K)$ and $U(K)$ that
  is explicitly given by $\psi$ by assumption.

  Any element of $\PP^2_K(K)$ is represented uniquely up to
  multiplication by scalars from $\OO_K^\times$ by $(y_0, y_1, y_2)
  \in \OO_K^3 \smallsetminus \{0\}$ with $y_0\OO_K + y_1\OO_K +
  y_2\OO_K \in \classrep$ (and in particular $y_0,y_1,y_2$ in the same
  element of $\classrep$, say $C_0$). Therefore, $\psi$ induces a
  bijection between the orbits of the action of $\OO_K^\times$ by
  scalar multiplication on the disjoint union
  \begin{equation*}
    \bigcup_{C_0 \in \classrep} \left\{(y_0, y_1, y_2) \in C_0^3 \WHERE
      \begin{aligned}
        &y_0 C_0^{-1}+y_1 C_0^{-1}+y_2 C_0^{-1}=\OO_K,\\
        &y_{i-1} \ne 0\text{ if $E_i$ is a negative curve, for $i=1,2,3$,}\\
        &R'(y_0,y_1,y_2) \ne 0\text{ if $E_4$ is a negative curve}
      \end{aligned}
    \right\}
  \end{equation*}
  and $U(K)$.

  We rename $(y_0,y_1,y_2)$ to $(\e_1,\e_2,\e_3)$ and introduce an
  additional variable $\e_4:=R'(\e_1,\e_2,\e_3)$, which is equivalent
  to $R(\e_1, \dots, \e_4, 1, \dots, 1)=0$ by \eqref{eq:rprime}. By
  \eqref{eq:psii}, this substitution turns $\psi$ into $\Psi(\e_1,
  \dots, \e_4, 1, \dots, 1)$.  We note $(\OO_1^{(0)}, \dots,
  \OO_4^{(0)})=(C_0,C_0,C_0,C_0^D)$ by \eqref{eq:coefficients} and
  that the action of $\lambda_0 \in \OO_K^\times$ on $(\e_1,
  \e_2,\e_3)$ by scalar multiplication leads to an action on $\e_4$ by
  multiplication by $\lambda_0^D$.

  It remains to show that the coprimality condition for
  $\e_1,\e_2,\e_3$ is equivalent to the system of coprimality
  conditions described in Claim~\ref{claim:passage_step}. Since any
  two curves in $\PP^2_K$ meet and since $E_1^{(0)}, E_2^{(0)},
  E_3^{(0)}$ do not meet in one point, we must show that adding resp.\
  removing a condition such as $\e_1 C_0^{-1}+\e_2 C_0^{-1}+\e_4
  C_0^{-D}=\OO_K$ for $E_1^{(0)} \cap E_2^{(0)} \cap E_4^{(0)} =
  \emptyset$ makes no difference. The emptiness of this intersection
  is equivalent to $R'(0,0,1) \ne 0$, i.e., the term $Y_2^D$ appears
  in $R'$ with a nonzero coefficient. In fact, this coefficient is
  $\pm 1$ since all coefficients in $R$ are $\pm 1$ by assumption, and
  this could fail after the substitution in \eqref{eq:rprime} only if
  two terms of $R$ would differ only by powers of $\te_5, \dots,
  \te_{r+4}$, which is impossible because of \eqref{eq:coefficients}
  and the homogeneity of $R$.  If there was a prime ideal $\p$ of
  $\OO_K$ dividing $\e_1C_0^{-1},\e_2 C_0^{-1},\e_4C_0^{-D}$, then the
  relation $\e_4=R'(\e_1,\e_2,\e_3)$ would imply that $\p$ divides
  $\e_3 C_0^{-1}$, contradicting the coprimality of
  $\e_1C_0^{-1},\e_2C_0^{-1},\e_3C_0^{-1}$.
\end{proof}

\begin{lemma}\label{lem:passage_step}
  Assume that Claim~\ref{claim:passage_step} holds for some $i-1 \in \{0,
  \dots, r-1\}$. If $\rho_i$ in~\eqref{eq:blow_up_sequence} is the blow-up of
  a point on $\tS_{i-1}$ lying on precisely two of $E_1^{(i-1)}, \dots,
  E_{i+3}^{(i-1)}$, if these two meet transversally in that point and meet
  nowhere else, and if the strict transform on $\tS$ of at least one of these
  two is a negative curve, then Claim~\ref{claim:passage_step} holds for $i$.
\end{lemma}

\begin{remark}\label{rem:passage_step_general}
  For most steps of the proof of Lemma~\ref{lem:passage_step}, we consider the
  following more general situation for $\rho_i$. Let $J_0$ be the set of all
  $j \in \{1, \dots, i+3\}$ such that $E^{(i-1)}_j$ contains the blown-up
  point $p_i \in \tS_{i-1}$. Assume that $p_i$ has multiplicity $1$ on each
  $E^{(i-1)}_j$ with $j \in J_0$, that we have $\bigcap_{j \in J_0} E^{(i)}_j
  = \emptyset$ for their strict transforms on $\tS_i$, and that the strict
  transform $E_j$ on $\tS$ is a negative curve for some $j \in J_0$.
  
  The additional assumption $|J_0|=2$ in Lemma~\ref{lem:passage_step} is used
  only for parts of one direction of the coprimality conditions, see
  (\ref{eq:special_coprimality}) below. Without this assumption, we expect
  that we must use the torsor equation to derive the coprimality conditions
  for $J \subset J_0 \cup \{i+4\}$ of Claim~\ref{claim:passage_step} for $i$.
\end{remark}

\begin{proof}[Proof of Lemma~\ref{lem:passage_step}]
  Except in the paragraph containing (\ref{eq:special_coprimality}), we
  work in the situation of Remark~\ref{rem:passage_step_general}.

  We write $E_j':=E_j^{(i-1)}$ for divisors on $\tS_{i-1}$ and
  $E_j'':=E_j^{(i)}$ for their strict transforms on $\tS_i$. The
  exceptional divisor of $\rho_i$ is $E_{i+4}'' := E_{i+4}^{(i)}$.

  Let $M'$ resp.\ $M''$ be the disjoint union in step $i-1$ resp.\ $i$ of
  Claim~\ref{claim:passage_step}. We construct a bijection between the
  $(\OO_K^\times)^i$-orbits in $M'$ and the
  $(\OO_K^\times)^{i+1}$-orbits in $M''$.  We use $\e_j'$ for
  coordinates of points in $M'$ and $\e_j''$ for coordinates in $M''$,
  and similarly $\OO_j' := \OO_j^{(i-1)}$, $\OO_j'' := \OO_j^{(i)}$ and
  $\eI_j' := \eI_j^{(i-1)}$, $\eI_j'' := \eI_j^{(i)}$ for their
  corresponding (fractional) ideals.

  Given $\ee'=(\e_1', \dots, \e_{i+3}') \in M'$, we have corresponding $C_0,
  \dots, C_{i-1} \in \classrep$ and $\OO_j'$ with $\e_j' \in \OO_{j*}'$, and
  $\eI_j' = \e_j'\OO_j'^{-1}$. Since $E_j$ is a negative curve on $\tS$ for
  some $j \in J_0$, at least one of the $\e_j'$ with $j \in J_0$ is
  nonzero. Therefore, there is a unique $C_i \in \classrep$ such that
  $[\sum_{j \in J_0} \eI_j'] = [C_i^{-1}]$, giving $\OO_j''$ and $\eI_j''$ for
  $j=1, \dots, i+3$. Choose $\e_{i+4}''\in C_i=\OO_{i+4}''$ such that
  $\eI_{i+4}'' = \sum_{j \in J_0} \eI_j'$, which is unique up to
  multiplication by $\OO_K^\times$. Then we define $\e_j'':=\e_j'/\e_{i+4}''$
  for $j \in J_0$ and $\e_j'' := \e_j'$ for all $j \in \{1, \dots, i+3\}
  \smallsetminus J_0$, giving $\ee''=(\e_1'', \dots, \e_{i+4}'') \in
  \OO_{1*}'' \times \dots \times \OO_{i+4*}''$, uniquely up to the action of
  $\lambda_i \in \OO_K^\times$ by $\e_{i+4}''\mapsto \lambda_i\e_{i+4}''$ and
  $\e_j''\mapsto\lambda_i^{-1}\e_j''$ for all $j \in J_0$ and $\e_j''\mapsto
  \e_j''$ for all $j \in \{1, \dots, i+3\}\smallsetminus J_0$.
 
  We check that these $\ee''$ satisfy the coprimality conditions on $M''$. For
  $J \subset \{1, \dots, i+4\}$ with $J \not\subset J_0 \cup \{i+4\}$, assume
  first that $i+4 \notin J$. Since blowing up $p_i$ only separates divisors
  meeting in $p_i$ and since $J \not\subset J_0$, we have $\bigcap_{j \in J}
  E_j''=\emptyset$ only for $\bigcap_{j \in J} E_j' = \emptyset$, hence
  $\sum_{j \in J} \eI_j'=\OO_K$ and hence $\sum_{j \in J} \eI_j''=\OO_K$, as
  desired, because each $\eI_j''$ divides $\eI_j'$. Assume next that $i+4 \in
  J$. Then only the case $J = \{k,i+4\}$ with $k \notin J_0$ is relevant
  because of the minimality assumption on $J$, so $E_k'' \cap E_{i+4}''
  =\emptyset$; by the assumption $\bigcap_{j \in J_0} E_j'' = \emptyset$, we
  have $\bigcap_{j \in J_0} E_j' = \{p_i\}$, hence $E_k' \cap (\bigcap_{j \in
    J_0} E_j') = \emptyset$; hence $\eI_k'+\sum_{j \in J_0} \eI_j'=\OO_K$ and
  since $\eI_{i+4}''$ divides all $\eI_j'$ with $j \in J_0$, we conclude
  $\eI_k''+\eI_{i+4}''=\OO_K$.

  It remains to check the coprimality conditions for
  \begin{equation}\label{eq:special_coprimality}
    J \subset J_0 \cup \{i+4\}.
  \end{equation}
  Here we use the additional assumption $|J_0|=2$, say $J_0=\{a,b\}$. Then our
  other assumptions imply $([E_a'],[E_b'])=1$, hence $([E_a''],[E_b''])=0$ and
  $([E_a''],[E_{i+4}''])=([E_b''],[E_{i+4}''])=1$. Therefore, the only
  remaining coprimality condition is $\eI_a'' + \eI_b'' =\OO_K$, and this is
  clearly fulfilled using $\eI_a''=\eI_a'/\eI_{i+4}''$ and
  $\eI_b''=\eI_b'/\eI_{i+4}''$ with $\eI_{i+4}''=\eI_a' + \eI_b'$.

  To check that the $\ee''$ constructed above satisfy the torsor
  equation on $M''$, we first discuss how the polynomial $R$ behaves
  under analogous substitutions. Let $c_0\ell_0+\dots+c_r\ell_r$ be
  the degree of the homogeneous relation $R$ of the Cox ring. Then
  $R(T_1', \dots, T_{i+3}', 1, \dots, 1)$ is homogeneous of degree
  $c_0\ell_0+\dots+c_{i-1}\ell_{i-1}$ if we give each $T_j'$ the
  degree $[E_j']=a_{j,0}\ell_0+\dots+a_{j,i-1}\ell_{i-1}$ for the
  moment. Similarly, $R(T_1'', \dots, T_{i+4}'', 1, \dots, 1)$ is
  homogeneous of degree $c_0\ell_0+\dots+c_i\ell_i$ if we give each
  $T_j''$ the degree $[E_j'']=a_{j,0}\ell_0+\dots+a_{j,i}\ell_i$.  If
  we substitute $T_j'$ in $R(T_1', \dots,
  T_{i+3}', 1, \dots, 1)$ by $T_j''T_{i+4}''$ for $j \in J_0$ and by 
  $T_j''$ for $j \in \{1, \dots, i+3\} \smallsetminus J_0$, then we 
  obtain an expression in $T_1'', \dots, T_{i+4}''$
  that is homogeneous of the same degree
  $c_0\ell_0+\dots+c_{i-1}\ell_{i-1}$. Indeed, $T_j''T_{i+4}''$ has
  the same degree as $T_j'$ for $j \in J_0$ since $[E_j'']=[E_j']-\ell_i$ and
  $[E_{i+4}'']=\ell_i$, and similarly for $j \in \{1, \dots, i+3\} 
  \smallsetminus J_0$. Furthermore, the
  result of the substitution clearly agrees with $R(T_1'', \dots,
  T_{i+4}'', 1, \dots, 1)$ up to powers of $T_{i+4}''$ in each
  term. But both are homogeneous of degrees differing by $c_i\ell_i$,
  so the result of the substitution is $T_{i+4}''^{-c_i}R(T_1'',
  \dots, T_{i+4}'', 1, \dots, 1)$.

  Since $\e_j'=\e_j''\e_{i+4}''$ for $j \in J_0$ and 
  $\e_j'=\e_j''$ for $j \in \{1, \dots, i+3\} \smallsetminus J_0$, 
  this implies that
  \begin{equation*}
    \e_{i+4}''^{-c_i}R(\e_1'',
    \dots, \e_{i+4}'', 1, \dots, 1) = R(\e_1', \dots, \e_{i+3}', 1, \dots, 1).
  \end{equation*}
  Since $R(\e_1', \dots, \e_{i+3}',1, \dots, 1)=0$ and $\e_{i+4}'' \ne
  0$, this implies that $\ee''$ satisfies the torsor equation on
  $M''$. In total, we have constructed for $\ee' \in M'$ an
  $\OO_K^\times$-orbit of $\ee'' \in M''$.

  In the other direction, given $\ee'' \in M''$ with corresponding $C_0,
  \dots, C_i \in \classrep$, we define $\e_j':=\e_j''\e_{i+4}''$ for $j \in
  J_0$ and $\e_j':=\e_j''$ for $j \in \{1, \dots, i+3\} \smallsetminus J_0$,
  giving $\ee'=(\e_1', \dots, \e_{i+3}') \in \OO_{1*}' \times \dots \times
  \OO_{i+3*}'$.

  If $\ee'' \in M''$ satisfies the coprimality conditions, the same holds for
  $\ee'$ that we just defined. Indeed, if $\bigcap_{j \in J} E_j' = \emptyset$,
  then $\bigcap_{j \in J} E_j'' = \emptyset$ since blowing up only decreases
  intersection numbers, so $\sum_{j \in J} I_j''=\OO_K$. Since $\bigcap_{j \in
    J} E_j' = \emptyset$ does not contain $p_i$, there is at least one $k \in
  J$ with $k \notin J_0$, so $([E_k''],[E_{i+4}''])=0$, hence
  $I_k''+I_{i+4}''=\OO_K$. In particular, the factors $\e_{i+4}''$ in
  $\e_j'=\e_j''\e_{i+4}''$ for all $j \in J \cap J_0$ do not contribute to the
  greatest common divisor, so we have $\sum_{j \in J} I_j'=\OO_K$. Therefore,
  $\ee'$ satisfies the coprimality conditions on $M'$. Similarly as above,
  $\ee'$ satisfies the torsor equation. Clearly all $\ee''$ in the same
  $\OO_K^\times$-orbit give the same $\ee'$.

  Obviously, $\ee' \mapsto \ee'' \mapsto \ee'$ is the identity on $M'$
  (for any choice of $\e''$ in the corresponding
  $\OO_K^\times$-orbit). The assumption $\bigcap_{j \in J_0} E_j''=\emptyset$
  gives the coprimality condition $\sum_{j \in J_0} \eI_j'' = \OO_K$ on $M''$,
  and this ensures that $\ee'' \mapsto \ee' \mapsto \ee''$
  yields an element of the same $\OO_K^\times$-orbit as the original
  $\ee''$. We have thus constructed a bijection between $M'$ and
  $\OO_K^\times$-orbits in $M''$.

  Moreover, it is clear that the $\OO_K^\times$-orbits in $M''$ are
  contained in the $(\OO_K^\times)^{i+1}$-orbits from Claim
  \ref{claim:passage_step}, and that $\ee_1', \ee_2' \in M'$ are in
  the same $(\OO_K^\times)^{i}$-orbit if and only if $\ee_1''$ and
  $\ee_2''$ are in the same $(\OO_K^\times)^{i+1}$-orbit. Hence, our
  bijection induces the claimed bijection between orbits on $M'$ and
  $M''$.

  Using the coprimality condition $\sum_{j \in J_0} \eI_j'' = \OO_K$, we see
  that the union defining $M''$ is disjoint if the union defining
  $M''$ is disjoint.
 
  To conclude our proof, it is enough to show that the map $M'' \to
  \PP_K^{\deg(S)}(K)$ defined in Claim \ref{claim:passage_step}, step
  $i$, coincides with the composition $M'' \to U(K)$ of the map $M''
  \to M'$ constructed above and the map $M'\to U(K)$ from step
  $i-1$. Using the same gradings and substitutions as in the
  discussion of $R$, we note that $\Psi_i(T_1', \dots, T_{i+3}',1,
  \dots, 1)$ is homogeneous of degree
  $3\ell_0-\ell_1-\dots-\ell_{i-1}$. Our substitution turns this into
  a monic monomial of the same degree that coincides up to powers of
  $T_{i+4}''$ with the monic monomial $\Psi_i(T_1'', \dots, T_{i+4}'',
  1, \dots, 1)$, which is homogeneous of degree
  $3\ell_0-\ell_1-\dots-\ell_i$. Since $T_{i+4}''$ has degree
  $\ell_i$, the substitution gives $T_{i+4}''\Psi_i(T_1'', \dots,
  T_{i+4}'', 1, \dots, 1)$. Thus, both maps send $\ee'' \in M''$ to
  $K$-rational points in projective space that differ by a factor of
  $\e_{i+4}'' \ne 0$ in each coordinate, hence are the same.
\end{proof}

\begin{remark}\label{rem:passage_degrees}
  By our assumption, in the Cox ring relation $R(\te_1, \dots,
  \te_{r+4})=\sum_{k=1}^t \lambda_k\te_1^{b_{1,k}}\cdots \te_{r+4}^{b_{r+4,k}}$
  with exponents $b_{j,k} \in \ZZnn$, all coefficients $\lambda_k$ are $\pm
  1$. For $j=1, \dots, r+4$, write $\OO_j := \OO_j^{(r)}$ for simplicity. Then
  the fractional ideals $\lambda_k \OO_1^{b_{1,k}}\cdots \OO_{r+4}^{b_{r+4,k}}$
  coincide for all $k=1, \dots, t$. Indeed, since $R$ is homogeneous of some
  degree $c_0\ell_0+\dots+c_r\ell_r \in \Pic(\tS)$, each of them is
  $C_0^{c_0}\cdots C_r^{c_r}$.
\end{remark}

\section{The first summation}\label{sec:first_summation}

Let $K$ be an imaginary quadratic field, which we regard as a subfield
of $\CC$.  Given a parameterization as in Claim~\ref{claim:passage} of
rational points on a del Pezzo surface $S$, we must estimate the
cardinality of each $M_\classtuple(B)$. As indicated in
Section~\ref{sec:plan}, we start by estimating the number of
$\e_{B_0}, \e_{C_0}$ in the fractional ideals $\OO_{B_0}, \OO_{C_0}$,
say, satisfying the torsor equation, with the remaining variables
fixed. The details depend on the precise shape of the torsor equation
and coprimality conditions, via the configuration of curves on $\tS$
encoded in an extended Dynkin diagram. In this section, we assume that
they are as in \eqref{eq:gen_torsor} and
Figure~\ref{fig:general_coprim_graph}. As discussed in
\cite[Remark~2.1]{MR2520770}, this is true for the majority of
singular del Pezzo surfaces described in \cite{math.AG/0604194}, and
the additional assumptions for Proposition~\ref{prop:first_summation}
are expected to follow from Claim~\ref{claim:passage}.

We use the following notation, similar to \cite[Section
2]{MR2520770}. Let $r$, $s$, $t \in \ZZnn$, $(a_0, \ldots, a_r) \in
\ZZp^{r+1}$, $(b_0, \ldots, b_s) \in \ZZp^{s+1}$, $(c_1, \ldots,
c_t)\in \ZZp^{t}$. Let $G = (V, E)$ be the graph given in Figure
\ref{fig:general_coprim_graph}, and let $G' = (V', E')$ be the graph
obtained from $G$ by deleting the vertices $B_0$, $C_0$ (see Figure
\ref{fig:general_coprim_graph_vertices_deleted}).
\begin{figure}[ht]
  \begin{minipage}[b]{0.50\linewidth}
    \centering
    \[
    \xymatrix@C-=12pt@R-=7pt{
      A_0 \ar@{-}[r] \ar@{-}[dd] \ar@{-}[rd] & A_r \ar@{-}[r] & A_{r-1} \ar@{-}[r] & \cdots \ar@{-}[r] & A_1 \ar@{-}[rd]\\
      &B_0 \ar@{-}[r] & B_s \ar@{-}[r] & \cdots \ar@{-}[r] & B_1 \ar@{-}[r] & D\\
      C_0 \ar@{-}[r] \ar@{-}[ru] & C_t \ar@{-}[r] & C_{t-1} \ar@{-}[r]
      & \cdots \ar@{-}[r] & C_1 \ar@{-}[ru] }
    \]
    \caption{$G = (V, E)$.}
    \label{fig:general_coprim_graph}
  \end{minipage}
  \hspace{0.5cm}
  \begin{minipage}[b]{0.44\linewidth}
    \centering
    \[
    \xymatrix@C-=12pt@R-=7pt{
      A_0 \ar@{-}[r] & A_r \ar@{-}[r] & \cdots \ar@{-}[r] & A_1 \ar@{-}[rd]\\
      B_s \ar@{-}[rr]& & \cdots \ar@{-}[r] & B_1 \ar@{-}[r] & D\\
      C_t \ar@{-}[rr]& & \cdots \ar@{-}[r] & C_1 \ar@{-}[ru] }
    \]
    \caption{$G' = (V', E')$.}
    \label{fig:general_coprim_graph_vertices_deleted}
  \end{minipage}
\end{figure}

\noindent For $v \in V$, let $\OO_v$ be a nonzero fractional ideal of
$K$ such that
\begin{equation*}
  \OO_{A_0}^{a_0}\cdots \OO_{A_r}^{a_r} = \OO_{B_0}^{b_0}\cdots \OO_{B_s}^{b_s} 
  = \OO_{C_0}\OO_{C_1}^{c_1}\cdots \OO_{C_t}^{c_t} =:\OO
\end{equation*}
(see Remark~\ref{rem:passage_degrees}). We define
\begin{equation*}
  \OO_{v*} :=
  \begin{cases}
    \OO_{A_0} \text{ or } \OO_{A_0}^{\neq 0}&\text{ if }v = A_0\\
    \OO_{v}&\text{ if } v \in \{B_0, C_0\}\\
    \OO_{v}^{\neq 0}&\text{ if } v \in V\smallsetminus\{A_0, B_0,
    C_0\}\text.
  \end{cases}
\end{equation*}
For $B > 0$, let $M(B)$ be the set of all $(\e_v)_{v \in V} \in
\prod_{v \in V}\OO_{v*}$ with the following properties:
\begin{itemize}
\item $(\e_v)_{v \in V\smallsetminus \{D\}}$ satisfies the
  \emph{torsor equation}
  \begin{equation}\label{eq:gen_torsor}
    \e_{A_0}^{a_0}\cdots \e_{A_r}^{a_r} + \e_{B_0}^{b_0}\cdots \e_{B_s}^{b_s} + \e_{C_0}\e_{C_1}^{c_1}\cdots \e_{C_t}^{c_t} = 0\text.  
  \end{equation}
\item $(\e_v)_{v \in V' \cup \{B_0\}}$ satisfies \emph{height
    conditions} written as
  \begin{equation}\label{eq:gen_height}
    ((\e_v)_{v \in V'}, \e_{B_0}) \in \mathcal{R}(B)\text,
  \end{equation}
  for a subset $\mathcal{R}(B) \subset \CC^{V'}\times\CC$. Moreover, we
  assume that for all $(\e_v)_{v\in V'}$ and $B$, the set
  $\mathcal{R}((\e_v)_{v\in V'}; B)$ of all $z \in \CC$ with $((\e_v)_{v\in V'},
  z) \in \mathcal{R}(B)$ is of class $m$ (see Definition \ref{def:class_m}) and
  contained in the union of $k$ closed balls of radius $R((\e_v)_{v
    \in V'};B)$. Here, $k$, $m$ are fixed constants.
\item The ideals
  \begin{equation*}
    \eI_v := \e_v \OO_v^{-1}\text{, }\ v \in V\text,
  \end{equation*}
  of $\OO_K$ satisfy the \emph{coprimality conditions encoded by the
    graph $G$}, in the following sense: For any two non-adjacent
  vertices $v$ and $w$ in $G$, the corresponding ideals $\eI_v$ and
  $\eI_w$ are relatively prime. We impose the additional coprimality
  condition
  \begin{equation*}
    \text{Each prime ideal $\p$ dividing $\eI_{D}$ may divide at most one of $\eI_{A_0}$, $\eI_{B_0}$, $\eI_{C_0}$,}
  \end{equation*}
  which is only relevant if at least two of $r$, $s$, $t$ are
  $0$. Thus, $(\eI_{A_0}, \eI_{B_0}, \eI_{C_0})$ is the only triplet
  of ideals $\eI_v$ allowed to have a nontrivial common divisor.
\end{itemize}

In this section, we count, for fixed $(\e_v)_{v \in V'}$, the number
of all $(\e_{B_0}, \e_{C_0})$ such that $(\e_v)_{v \in V}$ satisfies
the above conditions. This is analogous to \cite[Section
2]{MR2520770}, except that non-uniqueness of factorization in our case
(if $h_K > 1$) leads to technical difficulties. For ease of notation,
we write $\ee' := (\e_v)_{v\in V'}$, $\eII' := (\eI_v)_{v \in V'}$,
\begin{align*}
  \eeA &:= (\e_{A_1}, \ldots, \e_{A_r})\text{, }& \eeB &:= (\e_{B_1}, \ldots, \e_{B_s})\text{, } & \eeC &:= (\e_{C_1}, \ldots, \e_{C_t})\text,\\
  \eIIA &:= (\eI_{A_1}, \ldots,
  \eI_{A_r})\text{, }& \eIIB &:= (\eI_{B_1}, \ldots, \eI_{B_s})\text{,
  }&\eIIC &:= (\eI_{C_1}, \ldots, \eI_{C_t})\text.
\end{align*}
Let
\begin{equation*}
  \Pi(\eeA) := \e_{A_1}^{a_1}\cdots \e_{A_r}^{a_r}\text{, }\quad \Pi(\eIIA) := \eI_{A_1}^{a_1}\cdots \eI_{A_r}^{a_r}\text,
\end{equation*}
and
\begin{equation*}
  \Pi'(\eI_{D}, \eIIA) :=
  \begin{cases}
    \eI_{D} \eI_{A_1} \cdots \eI_{A_{r-1}}\text{, }&\text{if }r \geq 1\\
    \OO_K\text{, }&\text{if }r=0\text.
  \end{cases}
\end{equation*}
Analogously, we define $\Pi(\eeB)$, $\Pi(\eIIB)$, $\Pi'(\eI_{D},
\eIIB)$ and $\Pi(\eeC)$, $\Pi(\eIIC)$, $\Pi'(\eI_{D}, \eIIC)$.

The following notation encoding coprimality conditions is similar to
the one in Definition~\ref{def:thetarprime}.  For any prime ideal $\id
p$ of $\OO_K$, let
\begin{equation}\label{eq:general_def_Ip}
  J_{\id p}(\eII') := \{ v \in V' \ :\ \p \mid \eI_v\}\text.
\end{equation}
We define $\theta_0(\eII') := \prod_{\id p}\theta_{0,\p}(J_{\id
  p}(\eII'))$, where
\begin{equation*}
  \theta_{0, \id p}(J) :=
  \begin{cases}
    1 &\text{ if }J = \emptyset\text{, }J = \{v\}\text{ with }v \in V'\text{, or }J = \{v, w\} \in E',\\
    0 &\text{ otherwise.}
  \end{cases}
\end{equation*}

\begin{lemma}\label{lem:gen_simpler_coprimality_conditions}
  If $(\e_v)_{v \in V\smallsetminus \{D\}}$ satisfy the torsor
  equation \eqref{eq:gen_torsor}, then the coprimality conditions
  encoded by $G$ are equivalent to
  \begin{align}
    &\eI_{B_0} + \Pi'(\eI_{D}, \eIIB)\Pi(\eIIA) = \OO_K\label{eq:gen_cop_b}\\
    &\eI_{C_0} + \Pi'(\eI_{D}, \eIIC) = \OO_K\label{eq:gen_cop_c}\\
    &\theta_0(\eII') = 1\text.\label{eq:gen_cop_rest}
  \end{align}
\end{lemma}

\begin{proof}
  This is analogous to \cite[Lemma 2.3]{MR2520770}. Condition
  \eqref{eq:gen_cop_rest} is equivalent to the coprimality conditions
  encoded by $G$ for all $\eI_v$, $v \in V'$. Conditions
  \eqref{eq:gen_cop_b}, \eqref{eq:gen_cop_c} are clearly implied by
  the coprimality conditions for $\eI_{B_0}$, $\eI_{C_0}$,
  respectively. Using the torsor equation \eqref{eq:gen_torsor}, one
  can easily check that \eqref{eq:gen_cop_b} and
  \eqref{eq:gen_cop_rest} imply $\eI_{B_0}+\Pi'(\eI_{D},
  \eIIB)\Pi(\eIIA)\Pi(\eIIC) = \OO_K$, and that \eqref{eq:gen_cop_b},
  \eqref{eq:gen_cop_c}, \eqref{eq:gen_cop_rest} imply $\eI_{C_0}+
  \Pi'(\eI_{D}, \eIIC)\Pi(\eIIA)\Pi(\eIIB) = \OO_K$.
\end{proof}

For given $\ee'$, let $\id A = \id A(\ee')$ be a nonzero ideal of
$\OO_K$ that is relatively prime to $\Pi'(\eI_{D},
\eIIC)\Pi(\eIIC)$, such that we can write
\begin{equation*}
  \e_{A_0}^{a_0}\Pi(\eeA) = \Ao\At^{b_0}\text,
\end{equation*}
with $\At = \At(\ee') \in \id A\OO_{B_0}$ and $\Ao = \Ao(\ee') \in
\OO(\id A\OO_{B_0})^{-b_0}$.
\begin{remark}\label{rem:first_summation_choice_of_pi}
  For example, we can choose $\id A := \p$ to be a suitable prime
  ideal $\p$ not dividing $\Pi'(\eI_{D}, \eIIC)\Pi(\eIIC)$, such that
  $\p\OO_{B_0}$ is a principal fractional ideal $(t)$, and let $\At
  := t$, $\Ao := \e_{A_0}^{a_0}\Pi(\eeA)/t^{b_0}$. However, in some
  applications it is desirable use $\At^{b_0}$ to collect $b_0$-th
  powers of the variables $\e_{A_i}$ appearing in
  $\e_{A_0}^{a_0}\Pi(\eeA)$.
\end{remark}

\begin{prop}\label{prop:first_summation}
  With all the above definitions, we have
  \begin{align*}
    & |M(B)| = \frac{2}{\sqrt{|\Delta_K|}}\sum_{\substack{\ee' \in \prod_{v \in V'}\OO_{v*}}}\theta_1(\ee') V_1(\ee'; B)\\
    &+ O\left(\sum_{\ee'\text{, }\eqref{eq:first_summation_error_sum}}2^{\omega_K(\Pi'(\eI_{D},
        \eIIC)) + \omega_K(\Pi'(\eI_{D},
        \eIIB)\Pi(\eIIA))}b_0^{\omega_K(\eI_{D}\Pi(\eIIC))}\left(\frac{R(\ee';B)}{\N\Pi(\eIIC)^{1/2}}
        + 1\right)\right)\text,
  \end{align*}
  where the sum in the error term runs over all $\ee' \in \prod_{v \in
    V'}\OO_{v*}$ such that
  \begin{equation}\label{eq:first_summation_error_sum}
    \mathcal{R}(\ee'; B) \neq \emptyset,
  \end{equation}
  and the implicit constant may depend on $K$, $k$, $m$, and
  $\OO_{B_0}$. In the main term,
  \begin{equation*}
    V_1(\ee'; B):=\int_{z \in \mathcal{R}(\ee'; B)}\frac{1}{\N(\Pi(\eIIC)\OO_{B_0})}\dd z,
  \end{equation*}
  and 
  \begin{equation*}
    \theta_1(\ee') := \sum_{\substack{\kc \mid \Pi'(\eI_{D}, \eIIC)\\\kc+\eI_{A_0}\Pi(\eIIA)\Pi(\eIIB) =
        \OO_K}}\frac{\mu_K(\kc)}{\N\kc}\tilde{\theta}_1(\eII', \kc)\sum_{\substack{\rho \mod \kc\Pi(\eIIC)\\\rho\OO_K+\kc\Pi(\eIIC)=\OO_K\\
        \rho^{b_0} \equiv_{\kc\Pi(\eIIC)} -\Ao/\Pi(\eeB)}}1\text.
  \end{equation*}
  Here,
  \begin{equation*}
    \tilde{\theta}_1(\eII', \kc):=\theta_0(\eII')
    \frac{\phi_K^*(\Pi'(\eI_{D}, \eIIB)\Pi(\eIIA))}{\phi_K^*(\Pi'(\eI_{D}, \eIIB)+ \kc\Pi(\eIIC))},
  \end{equation*}
  and $\Ao/\Pi(\eeB)$ is invertible modulo $\kc\Pi(\eIIC)$ whenever
  $\theta_0(\eII')\neq 0$. In the inner sum, $\rho$ runs through a system of
  representatives for the invertible residue classes modulo $\kc\Pi(\eIIC)$
  whose $b_0$-th power is the class of $-\Ao/\Pi(\eeB)$.

\end{prop}

If $b_0=1$, then the sum over $\rho$ in the definition of $\theta_1$ is
just $1$ whenever $\theta_0(\eII')\neq 0$, so $\theta_1(\ee') = \theta_1'(\eII')$, where
\begin{equation}\label{eq:first_summation_def_theta_1_strich}
  \theta_1'(\eII') := \sum_{\substack{\kc \mid \Pi'(\eI_{D}, \eIIC)\\\kc+\eI_{A_0}\Pi(\eIIA)\Pi(\eIIB) =
      \OO_K}}\frac{\mu_K(\kc)}{\N\kc}\tilde\theta_1(\eII', \kc)\text.
\end{equation}
In our applications, the function $\theta_1'(\eII')$ plays an
important role in the computation of the main term in the second
summation, regardless of whether $b_0 = 1$ or not. Thus, let us
investigate $\theta_1'$, at least in the case where $s$, $t \geq
1$. Recall that the $\eI_v$, $v \in V' \smallsetminus\{A_0\}$, are
always nonzero ideals of $\OO_K$. In the following, we will assume
that $\eI_{A_0} \neq \{0\}$ holds as well.
\begin{lemma}\label{lem:general_comp_theta_1}
  Let $s$, $t \geq 1$. Then we have
  \begin{equation}\label{eq:general_comp_theta_1}
    \theta_1'(\eII') = \prod_{\id p}\theta'_{1, \id p}(J_{\id p}(\eII'))\text,
  \end{equation}
  where $J_\p$ is defined in \eqref{eq:general_def_Ip}, and for any $J
  \subset V'$,
  \begin{equation*}
    \theta'_{1, \id p}(J) :=
    \begin{cases}
      1 &\text{ if }   J = \emptyset, \{B_s\}, \{C_t\}, \{A_0\}\text,\\
      1-\frac{2}{\N\id p} &\text{ if } J = \{D\}\text,\\
      1-\frac{1}{\N\id p} &\text{ if $J= \{v\}$, with $v \in V' \smallsetminus \{B_s, C_t, A_0, D\}$,} \\
      \ &\text{ or $J = \{v, w\} \in E'$,} \\
      0 &\text{ otherwise.}
    \end{cases}
  \end{equation*}
  In particular, $\theta'_{1} \in \Theta_{r+s+t+2}'(2)$ and, with
  $\rk := r+s+t+1$, we have
  \begin{equation}\label{eq:general_average_theta_1}
    \mathcal{A}(\theta'(\eII'), \eII') = \prod_{\p}\left(1-\frac{1}{\N\p}\right)^\rk\left(1+\frac{\rk}{\N\p} + \frac{1}{\N\p^2}\right)\text.
  \end{equation}
  Moreover, let $v \in V' \smallsetminus\{A_1, B_1, C_1, D\}$ and let
  $\bbb$ be the product of all prime ideals of $\OO_K$ dividing at least
  one $\eI_w$ with $w \in V'\smallsetminus\{v\}$ not adjacent to
  $v$. Then, considered as a function of $\eI_v$, we have $\theta_1'(\eII')
  \in \Theta(\bbb, 1, 1, 1)$.
\end{lemma}

\begin{proof}
  We write $\theta_1'(\eII')$ as
  \begin{equation*}
    \theta_0(\eII')\frac{\phi_K^*(\Pi'(\eI_{D}, \eIIB)\Pi(\eIIA))}{\phi_K^*(\Pi'(\eI_{D}, \eIIB)+\Pi(\eIIC))}\hspace{-0.4cm}\sum_{\substack{\kc
        \mid \Pi'(\eI_{D}, \eIIC)\\\kc+\eI_{A_0}\Pi(\eIIA)\Pi(\eIIB) =
        \OO_K}}\hspace{-0.5cm}\frac{\mu_K(\kc)}{\N\kc}\prod_{\substack{\p \mid (\kc+\Pi'(\eI_{D}, \eIIB))\\\id p \nmid \Pi(\eIIC)}}\left(1-\frac{1}{\N\p}\right)^{-1}\text.
  \end{equation*}
  The first factor is defined as a product of local factors which
  depend only on the set $J_\p(\eII')$. It is obvious how to write the
  second factor as such a product. Recall that we assumed $s$, $t \geq
  1$. Whenever $\theta_0(\eII') \neq 0$, we can write the third factor
  as
  \begin{equation*}
    \prod_{\substack{\p \mid \eI_{D}\\\p \nmid \eI_{A_0}\Pi(\eIIA)\Pi(\eIIB)\Pi(\eIIC)}}\frac{\N\p - 2}{\N\p - 1}\prod_{\substack{\p \mid (\Pi'(\eI_{D}, \eIIC)+\Pi(\eIIC))\\\id p \nmid \eI_{A_0}\Pi(\eIIA)\Pi(\eIIB)}}\left(1 - \frac{1}{\N\p}\right)\text.
  \end{equation*}
  Now \eqref{eq:general_comp_theta_1} can be proved by a
  straightforward inspection of the local factors. To prove
  \eqref{eq:general_average_theta_1}, we use
  \eqref{eq:arith_full_average} in Lemma
  \ref{lem:repeated_average}. Then \eqref{eq:general_comp_theta_1} and
  counting the vertices and edges in $G'$ show that the local factor
  at each prime ideal $\p$ is indeed as in
  \eqref{eq:general_average_theta_1}.

  The last assertion in the lemma is again an immediate consequence of
  \eqref{eq:general_comp_theta_1}.
\end{proof}
An analogous version of the last assertion in Lemma
\ref{lem:general_comp_theta_1} holds for $\tilde\theta_1$.
\begin{lemma}\label{lem:theta_1_tilde_in_theta_4}
  Let $v \in V'$ and let $\bbb$ be the product of all prime ideals of
  $\OO_K$ dividing at least one $\eI_w$ with $w \in
  V'\smallsetminus\{v\}$ not adjacent to $v$. Then, considered as a
  function of $\eI_v$, we have $\tilde\theta_1(\eII', \kc) \in \Theta(\bbb, 1, 1,
  1)$.
\end{lemma}

\begin{proof}
  This follows immediately from the definition of $\tilde\theta_1$.
\end{proof}

\subsection{Proof of Proposition \ref{prop:first_summation}}
The proof is mostly analogous to \cite[Proposition 2.4]{MR2520770},
but the lack of unique factorization in $\OO_K$ leads to some technical
difficulties. We use two simple lemmas.

\begin{lemma}\label{lem:aux_congruence_1}
  Let $\ideal$ be an ideal and $\id f$ a nonzero
  fractional ideal of $\OO_K$. Let $y_1$, $y_2 \in \id f$ such that
  $(y_1\id f^{-1}, \ideal) = (y_2\id f^{-1}, \ideal) = \OO_K$. Then
  $y_2/y_1$ is invertible modulo $\ideal$ and, for $x \in \OO_K$, we
  have
  \begin{equation*}
    x y_1 - y_2 \in \ideal\id f \quad\text{ if and only if }\quad x\equiv_\ideal y_2/y_1\text.
  \end{equation*}
\end{lemma}

\begin{proof}
  For every prime ideal $\p \mid \ideal$, we have $v_\p(y_1) =
  v_\p(\id f) = v_\p(y_2)$, so $y_2/y_1$ is invertible modulo
  $\ideal$. Moreover, $x y_1 - y_2 \in \ideal\id f$ holds if and only if
  $v_\p(x-y_2/y_1)\geq v_\p(\ideal)-v_\p(y_1\id f^{-1})$ for all prime ideals $\p$. Given our
  assumptions, this is equivalent to $x \equiv_\ideal y_2/y_1$.
\end{proof}

\begin{lemma}\label{lem:aux_congruence_2}
  Let $\ideal_1$, $\ideal_2$ be fractional ideals of $\OO_K$ and let
  $x$, $y \in \ideal_2$ such that $x - y \in \ideal_1\ideal_2$. Then,
  for any positive integer $n$, we have $x^n - y^n \in
  \ideal_1\ideal_2^n$.
\end{lemma}

\begin{proof}
  Clearly, $x^n-y^n = (x-y)(x^{n-1}+x^{n-2}y + \cdots + y^{n-1}) \in
  \ideal_1\ideal_2 \cdot \ideal_2^{n-1}$.
\end{proof}

For fixed $B>0$ and $\ee' \in \prod_{v \in V'}\OO_{v*}$
subject to \eqref{eq:gen_cop_rest}, let $N_1 = N_1(\ee'; B)$ be the
number of all $(\e_{B_0}, \e_{C_0})\in \OO_{B_0} \times \OO_{C_0}$ such
that the torsor equation \eqref{eq:gen_torsor}, the coprimality
conditions \eqref{eq:gen_cop_b}, \eqref{eq:gen_cop_c}, and the height
conditions \eqref{eq:gen_height} are satisfied. Then
\begin{equation*}
  |M(B)| = \sum_{\substack{\ee' \in  \prod_{v \in V'}\OO_{v*}\\\eqref{eq:gen_cop_rest}}}N_1(\ee'; B)\text.
\end{equation*}
By M\"obius inversion for \eqref{eq:gen_cop_c}, we obtain
\begin{equation*}
  N_1 = \sum_{\kc \mid \Pi'(\eI_{D}, \eIIC)}\mu(\kc)\left|\left\{(\e_{B_0}, \e_{C_0}) \in \OO_{B_0} \times \kc\OO_{C_0} \mid \eqref{eq:gen_torsor}\text{, }\eqref{eq:gen_height}\text{, }\eqref{eq:gen_cop_b}\right\}\right|\text.
\end{equation*} 
We notice that, given $\e_{B_0} \in \OO_{B_0}$, there is a (unique)
$\e_{C_0} \in \kc\OO_{C_0}$ with \eqref{eq:gen_torsor} if and only if
\begin{equation}\label{eq:gen_congruence}
  \e_{A_0}^{a_0}\Pi(\eeA) + \e_{B_0}^{b_0}\Pi(\eeB) \in \Pi(\eeC)\kc\OO_{C_0} = \Pi(\eIIC)\kc\OO\text.
\end{equation}
Similarly as in the proof of \cite[Proposition 2.4]{MR2520770}, we see
that this, \eqref{eq:gen_cop_b}, and \eqref{eq:gen_cop_rest} can only hold if $\kc+\eI_{A_0}\Pi(\eIIA)\Pi(\eIIB) =
\OO_K$, so
\begin{equation*}
  N_1 = \sum_{\substack{\kc \mid \Pi'(\eI_{D}, \eIIC)\\\kc+\eI_{A_0}\Pi(\eIIA)\Pi(\eIIB)
      = \OO_K}}\mu(\kc)\left|\left\{\e_{B_0} \in \OO_{B_0} \mid
      \eqref{eq:gen_height}\text{, }\eqref{eq:gen_cop_b}, \eqref{eq:gen_congruence}\right\}\right|\text.
\end{equation*} 
Let us consider condition \eqref{eq:gen_congruence}. Recall the
definition of $\Ao$ and $\At $ before Proposition
\ref{prop:first_summation}. We note that
\begin{equation*}
  \Ao(\OO(\id A\OO_{B_0})^{-b_0})^{-1}\cdot (\At(\id A\OO_{B_0})^{-1})^{b_0} = \e_{A_0}^{a_0}\Pi(\eeA)\OO^{-1} = \eI_{A_0}^{a_0}\Pi(\eIIA), 
\end{equation*}
so $\Ao(\OO(\id A\OO_{B_0})^{-b_0})^{-1}+\kc\Pi(\eIIC)$ and $\At(\id
A\OO_{B_0})^{-1}+\kc\Pi(\eIIC)$ are $\OO_K$.

\begin{lemma}\label{lem:proof_first_sum_remove_power}
  For all $\e_{B_0} \in \OO_{B_0}$ satisfying \eqref{eq:gen_congruence}
  there exists $\rho$ in $\OO_K$, unique modulo $\kc\Pi(\eIIC)$, such
  that
  \begin{equation}\label{eq:rho_beta_0}
    \e_{B_0} - \rho \At \in \kc \Pi(\eIIC)\OO_{B_0}\text.
  \end{equation}
  This $\rho$ satisfies
  \begin{equation}\label{eq:rho_congruence}
    \rho^{b_0} \equiv_{\kc\Pi(\eIIC)} -\Ao/\Pi(\eeB)\text.
  \end{equation}
  Here, $\Ao/\Pi(\eeB) $ is invertible modulo $\kc\Pi(\eIIC)$, so
  $\rho$ is invertible modulo $\kc\Pi(\eIIC)$ as well.

  Conversely, if $\e_{B_0} \in \OO_{B_0}$ satisfies
  \eqref{eq:rho_beta_0} for some $\rho$ with \eqref{eq:rho_congruence}
  then it satisfies \eqref{eq:gen_congruence}.
\end{lemma}

\begin{proof}
  We write (\ref{eq:gen_congruence}) as
  \begin{equation}\label{eq:proof_first_sum_AB_cong}
    \e_{B_0}^{b_0}\Pi(\eeB) + \Ao \At ^{b_0} \in \kc\Pi(\eIIC)\OO\text.
  \end{equation}
  Since $\Ao\At ^{b_0}\OO^{-1} = \eI_{A_0}^{a_0}\Pi(\eIIA)$ is coprime
  to $\kc\Pi(\eIIC)$, we see that $\e_{B_0}^{b_0}\Pi(\eeB)\OO^{-1} =
  \eI_{B_0}^{b_0}\Pi(\eIIB)$ is coprime to $\kc\Pi(\eIIC)$ as well.

  Therefore, $\e_{B_0}\OO_{B_0}^{-1} = \eI_{B_0}$ is relatively prime
  with $\kc\Pi(\eIIC)$. Moreover, $\At \in \OO_{B_0}$ and, by our
  choice of $\id A$, we have $\At \OO_{B_0}^{-1}+\kc\Pi(\eIIC) =
  \OO_K$. Therefore, we can apply Lemma \ref{lem:aux_congruence_1} with
  $x:=\rho$, $y_1 := \At$, $y_2 := \e_{B_0}$, $\ideal :=
  \kc\Pi(\eIIC)$, and $\id f := \OO_{B_0}$ to see that there is a
  unique $\rho$ modulo $\kc\Pi(\eIIC)$ with (\ref{eq:rho_beta_0}).

  By Lemma \ref{lem:aux_congruence_2}, this $\rho$ satisfies
  \begin{equation}\label{eq:proof_first_sum_AB2_cong}
    \e_{B_0}^{b_0} - (\rho \At )^{b_0} \in \kc\Pi(\eIIC)\OO_{B_0}^{b_0}\text.
  \end{equation}
  Clearly, $\Pi(\eeB)\OO_K = \Pi(\eIIB)\OO\OO_{B_0}^{-b_0} \subset
  \OO\OO_{B_0}^{-b_0}$, so \eqref{eq:proof_first_sum_AB_cong} and
  \eqref{eq:proof_first_sum_AB2_cong} imply
  \begin{equation}\label{eq:proof_first_sum_AB2_cong_2}
    \rho^{b_0}\Pi(\eeB)  \At ^{b_0} + \Ao \At ^{b_0} \in \kc\Pi(\eIIC)\OO\text.
  \end{equation}
  Now $\Pi(\eeB) \At ^{b_0}\OO^{-1} = \Pi(\eIIB) \At
  ^{b_0}\OO_{B_0}^{-b_0}$, so $\Pi(\eeB)\At ^{b_0} \in \OO$ and
  $\Pi(\eeB)\At ^{b_0}\OO^{-1}+\kc \Pi(\eIIC) = \OO_K$. We have
  already seen that $\Ao\At ^{b_0} \in \OO$ and $\Ao\At ^{b_0}\OO^{-1}+
  \kc\Pi(\eIIC) = \OO_K$. By Lemma \ref{lem:aux_congruence_1},
  $\Ao/\Pi(\eeB)$ is invertible modulo $\kc\Pi(\eIIC)$ and
  \eqref{eq:rho_congruence} holds.

  Now assume that we are given $\e_{B_0} \in \OO_{B_0}$ such that
  \eqref{eq:rho_beta_0} and \eqref{eq:rho_congruence} hold for some
  $\rho$. By the same argument as in the above paragraph, using the
  reverse implication in Lemma \ref{lem:aux_congruence_1},
  \eqref{eq:rho_congruence} implies
  \eqref{eq:proof_first_sum_AB2_cong_2}. By Lemma
  \ref{lem:aux_congruence_2}, \eqref{eq:rho_beta_0} implies that
  $\e_{B_0}^{b_0} - (\rho \At)^{b_0} \in \kc\Pi(\eIIC)\OO_{B_0}^{b_0}$,
  which, together with \eqref{eq:proof_first_sum_AB2_cong_2}, yields
  \eqref{eq:gen_congruence}.
\end{proof}
By the lemma,
\begin{equation*}
  N_1 =\sum_{\substack{\kc \mid \Pi'(\eI_{D}, \eIIC)\\\kc+\eI_{A_0}\Pi(\eIIA)\Pi(\eIIB)
      = \OO_K}}\mu(\kc)\sum_{\substack{\rho \mod \kc\Pi(\eIIC)\\\rho\OO_K+\kc\Pi(\eIIC)=\OO_K\\\eqref{eq:rho_congruence}}}|\{\e_{B_0} \in \OO_{B_0}\mid
  (\ref{eq:gen_height}), (\ref{eq:gen_cop_b}), (\ref{eq:rho_beta_0})\}|\text.
\end{equation*}
After M\"obius inversion for the coprimality condition
(\ref{eq:gen_cop_b}), we have
\begin{equation*}
  N_1 =\sum_{\substack{\kc \mid \Pi'(\eI_{D}, \eIIC)\\\kc+\eI_{A_0}\Pi(\eIIA)\Pi(\eIIB)
      = \OO_K}}\mu(\kc)\sum_{\substack{\rho \mod \kc\Pi(\eIIC)\\\rho\OO_K+\kc\Pi(\eIIC)=\OO_K\\\eqref{eq:rho_congruence}}}\sum_{\kb \mid \Pi'(\eI_{D},
    \eIIB)\Pi(\eIIA)}N_2(\kc, \kb, \rho)\text,
\end{equation*}
where
\begin{equation*}
  N_2(\kc, \kb, \rho) := |\{\e_{B_0} \in \kb \OO_{B_0} \mid
  \eqref{eq:gen_height}, \eqref{eq:rho_beta_0}\}|\text.
\end{equation*}
Since $\rho \At \OO_{B_0}^{-1}+\kc\Pi(\eIIC)=\OO_K$, congruence
(\ref{eq:rho_beta_0}) implies that $\e_{B_0} \OO_{B_0}^{-1}+
\kc\Pi(\eIIC) = \OO_K$. Therefore, we can add the condition $\kb+
\kc\Pi(\eIIC) = \OO_K$ to the sum over $\kb$.

Let $\delta \in \OO_K\smallsetminus \{0\}$ such that $\delta \OO_{B_0}$ is an integral
ideal of $\OO_K$. The conditions $\e_{B_0} \in \kb\OO_{B_0}$ and
(\ref{eq:rho_beta_0}) can be written as a system of congruences
\begin{align*}
  \delta\e_{B_0} &\equiv 0 \mod \kb(\delta\OO_{B_0})\\
  \delta\e_{B_0} &\equiv \delta\rho \At \mod \kc\Pi(\eIIC)(\delta\OO_{B_0})\text.
\end{align*}
Since $\kb \delta\OO_{B_0}+\kc\Pi(\eIIC)\delta\OO_{B_0} = \delta\OO_{B_0}$ and $\delta\rho
\At \equiv 0 \mod \delta\OO_{B_0}$, we can apply the Chinese remainder
theorem. Thus, there is an element $x \in \OO_K$ such that these
congruences are equivalent to
\begin{equation*}
  \delta \e_{B_0} \equiv x \mod \kb\kc\Pi(\eIIC)(\delta\OO_{B_0})\text.
\end{equation*}
Hence,
\begin{equation*}
  N_2(\kc, \kb, \rho) = |(x/\delta + \kb\kc\Pi(\eIIC)\OO_{B_0}) \cap \mathcal{R}(\ee';B)|\text.
\end{equation*}
With our assumptions on $\mathcal{R}(\ee'; B)$, Lemma \ref{lem:lattice_points} yields
\begin{equation*}
  N_2(\kc, \kb, \rho) = \frac{2}{\sqrt{|\Delta_K|}}\frac{V_1(\ee'; B)}{\N(\kb\kc)}  + O\left(\frac{R(\ee';B)}{\N\Pi(\eIIC)^{1/2}} + 1\right)\text.
\end{equation*}
Now a simple computation shows that the main term in the proposition
is the correct one. For the error term, we notice that the number of
$\rho$ modulo $\kc\Pi(\eIIC)$ with (\ref{eq:rho_congruence}) is $\ll
b_0^{\omega_K(\eI_{D} \Pi(\eIIC))}$ by Hensel's lemma.

\section{The second summation}\label{sec:second_summation}
As in the previous section, $K$ denotes an imaginary quadratic
field. We provide tools to sum the main term resulting from
Proposition \ref{prop:first_summation} over a further variable.

First, we fix some notation: Let $\OO$ be a nonzero fractional ideal of
$K$, let $\id q \in \I_K$, and $n \in \ZZp$. Let $A \in K$ such that
$v_\p(A \OO) = 0$ for all prime ideals $\p$ of $\OO_K$ dividing $\id
q$. In particular, $A z$ is defined modulo $\id q$ for all $z \in \OO$.

We consider a function $\vartheta : \I_K \to \RR$ such that, with
constants $c_\vartheta > 0$ and $C \geq 0$,
\begin{equation}\label{eq:second_sum_thetabound}
  \sum_{\substack{\aaa \in \I_K\\\N\aaa \le t}} |(\vartheta*\mu_K)(\aaa)|\cdot \N\aaa \ll c_\vartheta t(\log(t+2))^{C}
\end{equation}
holds for all $t > 0$. We write
\begin{equation*}
  \mathcal{A}(\vartheta(\aaa), \aaa, \id q) := \sum_{\substack{\aaa \in \I_K\\\aaa+\id q = \OO_K}}\frac{(\vartheta*\mu_K)(\aaa)}{\N\aaa}\text.
\end{equation*}
(For $\vartheta \in \Theta(\bbb, C_1, C_2, C_3)$, this is
consistent with the definition given in Section
\ref{sec:arithmetic_functions}.)

For $1 \leq t_1 \leq t_2$, let $g : [t_1, t_2] \to \RR$ be a function
such that there exists a partition of $[t_1,t_2]$ into at most $R(g)$
intervals on whose interior $g$ is continuously differentiable and
monotonic.  Moreover, with constants $c_g > 0$ and $a \leq 0$, we
assume that
\begin{equation}\label{eq:second_sum_gbound}
  |g(t)| \ll c_g t^a \text{ on }[t_1, t_2].
\end{equation}
We find an asymptotic formula for the sum
\begin{equation*}
  S(t_1, t_2) := \sum_{\substack{z \in \OO^{\neq 0}\\t_1 < \N(z\OO^{-1}) \leq t_2}}\vartheta(z\OO^{-1})\sum_{\substack{\rho \mod \id q\\\rho\OO_K+\id q = \OO_K\\\rho^n \equiv_{\id q} A z}}g(\N(z\OO^{-1}))\text.
\end{equation*}

\begin{prop}\label{prop:second_summation}
  With the above definitions, we have
  \begin{equation*}
      S(t_1, t_2) = \frac{2\pi}{\sqrt{|\Delta_K|}}\phi_K^*(\id
      q)\mathcal{A}(\vartheta(\aaa), \aaa, \id q) \int_{t_1}^{t_2}g(t)\dd
      t+ O\left(c_\vartheta c_g(\sqrt{\N\id q}\mathcal{E}_1 + \N\id q
        \mathcal{E}_2)\right)\text,
  \end{equation*}
  where
  \begin{equation}\label{eq:lem_3_1_errorbound_1}
    \mathcal{E}_1 \ll_{a,C} R(g)
    \begin{cases}
      \sup_{t_1 \leq t \leq t_2}(t^{a+1/2}) &\text{ if }a \ne -1/2,\\
      \log(t_2 + 2) &\text{ if }a = -1/2,
    \end{cases}
  \end{equation}
  and
  \begin{equation}\label{eq:lem_3_1_errorbound_2}
    \mathcal{E}_2 \ll_{a,C} R(g)
    \begin{cases}
      t_1^{a}\log(t_1+2)^{C+1} &\text{ if }a < 0, \\
      \log(t_2 + 2)^{C+1} &\text{ if }a = 0\text.
    \end{cases}
  \end{equation}
  Moreover, the same formula holds if, in the definition of
  $S(t_1, t_2)$, the range $t_1 < \N(z\OO^{-1}) \leq t_2$ is replaced
  by $t_1 \leq \N(z\OO^{-1}) \leq t_2$.
\end{prop}

\begin{remark}\label{rem:second_summation_no_rho}
  In particular, we can apply Proposition \ref{prop:second_summation}
  with $\id q = \OO_K$, $n=A=1$ to handle sums of the form
  \begin{equation*}
    S(t_1, t_2) = \sum_{\substack{z \in \OO^{\neq 0}\\t_1 < \N(z\OO^{-1}) \leq t_2}}\vartheta(z\OO^{-1})g(\N(z\OO^{-1}))\text.
  \end{equation*}
  In this case, the error term is $\ll_{a,C}R(g)c_\vartheta c_g
  \sup_{t_1 \leq t \leq t_2}(t^{a+1/2})$ if $a \neq -1/2$ and
  $\ll_{C}R(g)c_\vartheta c_g\log(t_2+2)$ if $a = -1/2$. (Note
  that $t_2 \geq t_1 \geq 1$).
\end{remark}

Recall the notation of Section \ref{sec:first_summation}, in
particular Proposition \ref{prop:first_summation}. In a typical
application, we have $r$, $s$, $t \geq 1$, $b_0 \in \{1, 2\}$, and
\begin{equation*}
  V_1(\ee'; B) =: \tilde{V}_1((\N\eI_v)_{v \in V'}; B)
\end{equation*}
depends only on $B$ and the absolute norms of the ideals $\eI_v$, and
not on the $\e_v$. Let us describe how we apply Proposition
\ref{prop:second_summation} to sum the main term in the result of
Proposition \ref{prop:first_summation} over a further variable, say,
$\e_w$. We write $V'' := V'\smallsetminus \{w\}$, $\ee'' := (\e_v)_{v \in
  V''}$ and assume that $g(t) := \tilde{V}_1((\N\eI_v)_{v \in V''}, t;
B)$ satisfies the hypotheses from the beginning of this section. We
define
\begin{equation*}
  V_2((\N\eI_v)_{v \in V''}; B) := \pi\int\limits_{t \geq 1}g(t) \dd t
\end{equation*}
and distinguish between two cases.

In the first case, let $b_0 = 1$.  As mentioned after Proposition
\ref{prop:first_summation}, $\theta_1(\ee') = \theta_1'(\eII')$. Let
$\vartheta(\eI_w) := \theta_1'(\eII')$, considered as a function of
$\eI_w$. By the last assertion of Lemma \ref{lem:general_comp_theta_1}
and Lemma \ref{lem:6.8}, (2), $\vartheta$ satisfies
\eqref{eq:second_sum_thetabound} with $c_\theta = 2^{\omega(\bbb)}$
and $C=0$. Up to a possible contribution of $\e_{w} = 0$ (if $w =
A_0$), we can use Remark \ref{rem:second_summation_no_rho} to estimate
\begin{equation*}
  \frac{2}{\sqrt{|\Delta_K|}} \sum_{\substack{\ee''\in\prod_{v \in V''}\OO_{v*}}}\sum_{\e_w \in \OO_{w*}}\vartheta(\e_w\OO_{w}^{-1})g(\N(\e_w\OO_w^{-1}))\text.
\end{equation*}
We obtain a main term
\begin{equation}\label{eq:second_main_term}
  \left(\frac{2}{\sqrt{|\Delta_K|}}\right)^2 \sum_{\substack{\ee''\in\prod_{v \in V''}\OO_{v*}}}\mathcal{A}(\theta_1'(\eII'), \eI_v)V_2((\N\eI_v)_{v \in V''}; B)\text.
\end{equation}
It remains to bound the sum over $\ee''$ of the error term from Remark
\ref{rem:second_summation_no_rho}.

In the second case $b_0 \geq 2$, the sum over $\rho$ in the definition of $\theta_1$ is
not just $1$. However, we notice that, if $r$, $s \geq 1$, the
condition $\kc+\eI_{A_0}\Pi(\eIIA)\Pi(\eIIB) = \OO_K$ can be replaced
by $\kc+\eI_{A_1}\eI_{B_1} = \OO_K$, since the remaining coprimality
conditions follow from $\theta_0(\eII') = 1$. We additionally
assume that
\begin{equation*}
  w \in \{A_0, A_2, \ldots, A_r\}
\end{equation*}
and that $-\Ao/\Pi(\eeB)$ has the form $A\e_w$, where $A$ does not
depend on $\e_w$. Then $v_\p(A\OO_w) = 0$ for all $\p \mid
\kc\Pi(\eIIC)$. We apply Proposition \ref{prop:second_summation} once
for every summand in the sum over $\kc$, to sum the expression
\begin{equation*}
  \tilde\theta_1(\eII', \kc)\sum_{\substack{\rho \mod \kc\Pi(\eIIC)\\\rho\OO_K+\kc\Pi(\eIIC) = \OO_K\\\rho^{b_0} \equiv_{\kc\Pi(\eIIC)} A\e_{w}}}V_1(\ee'; B)
\end{equation*}
over $\e_w\in \OO_{w*}$. Let $\vartheta(\eI_w) := \tilde\theta_1(\eII',
\kc)$, considered as a function of $\eI_w$. By Lemma
\ref{lem:theta_1_tilde_in_theta_4} and Lemma \ref{lem:6.8}, (2),
$\vartheta$ satisfies \eqref{eq:second_sum_thetabound} with $c_\vartheta
= 2^{\omega_K(\bbb)}$ and $C=0$. After applying Proposition
\ref{prop:second_summation} and summing the result over $\kc$, we
obtain a main term
\begin{equation*}
  \left(\frac{2}{\sqrt{\Delta_K}}\right)^2\sum_{\substack{\ee''\in\prod_{v \in V''}\OO_{v*}}}\theta_2(\eII'')V_2((\N\eI_v)_{v \in V''}; B)\text,
\end{equation*}
where
\begin{equation}\label{eq:sec_sum_def_theta_2}
  \theta_2(\eII'') := \sum_{\substack{\kc \mid \Pi'(\eI_{D}, \eIIC)\\\kc+\eI_{A_1}\eI_{B_1} =
      \OO_K}}\frac{\mu_K(\kc)}{\N\kc}\phi_K^*(\kc\Pi(\eIIC))\mathcal{A}(\vartheta(\eI_w), \eI_w, \kc\Pi(\eIIC))\text.
\end{equation}
 
The following lemma shows that the main term is the same as in the
case $b_0 = 0$. It remains to bound the sum over $\e''$ and $\kc$ of
the error term multiplied by $\mu_K(\kc)/\N\kc$.

\begin{lemma}
  Assume that $r$, $s \geq 1$, choose $w \in \{A_0, A_2, \ldots, A_r,
  B_2, \ldots, B_s\}$, and let $\vartheta(\eI_w) :=
  \tilde\theta_1(\eII', \kc)$, considered as a function of
  $\eI_w$. Then $\vartheta(\eI_w) \in \Theta(\bbb, 1,1,1)$, where $\bbb$
  is given in Lemma \ref{lem:theta_1_tilde_in_theta_4}. Define
  $\theta_2(\eII'')$ as in \eqref{eq:sec_sum_def_theta_2} and
  $\theta_1'(\eII')$ as in
  \eqref{eq:first_summation_def_theta_1_strich}. Then we have
  \begin{equation*}
    \theta_2(\eII'') = \mathcal{A}(\theta_1'(\eII'), \eI_w)\text.
  \end{equation*}
\end{lemma}

\begin{proof}
  It is enough to prove that
  $\phi_K^*(\kc\Pi(\eIIC))\mathcal{A}(\vartheta(\eI_w), \eI_w,
  \kc\Pi(\eIIC)) = \mathcal{A}(\vartheta(\eI_w), \eI_w)$ holds
  whenever $\kc$ satisfies the conditions under the sum. This is
  clearly true if $\vartheta$ is the zero function. If not, write
  $\vartheta(\eI_w) = \prod_\p A_\p(v_\p(\eI_w))$ with $A_\p(n) =
  A_\p(1)$ for all prime ideals $\p$ and all $n \geq 1$. By Lemma
  \ref{lem:6.8}, \emph{(3)},
  $\phi_K^*(\kc\Pi(\eIIC))\mathcal{A}(\vartheta(\eI_w), \eI_w,
  \kc\Pi(\eIIC))$ is given by
  \begin{equation*}
    \prod_{\p \nmid \kc\Pi(\eIIC)}\left(\left(1-\frac{1}{\N\p}\right)A_\p(0) + \frac{1}{\N\p}A_\p(1)\right)\prod_{\p\mid\kc\Pi(\eIIC)}\left(1-\frac{1}{\N\p}\right)A_\p(0)\text,
  \end{equation*}
  and
  \begin{equation*}
    \mathcal{A}(\vartheta(\eI_w), \eI_w) = \prod_{\p}\left(\left(1-\frac{1}{\N\p}\right)A_\p(0) + \frac{1}{\N\p}A_\p(1) \right)\text.
  \end{equation*}
  By our choice of $w$, we have $\vartheta(\eI_w) =
  \tilde\theta_1(\eII', \kc) = 0$ whenever $p \mid (\eI_w+
  \kc\Pi(\eIIC))$. Since $\vartheta$ is not identically zero, this
  implies $A_\p(1) = 0$ for all $\p \mid \kc\Pi(\eIIC)$.
\end{proof}

\subsection{Proof of Proposition \ref{prop:second_summation}}
First, we prove a version of Lemma \ref{lem:6.2} that counts elements
in a given residue class instead of ideals.
\begin{lemma}\label{lem:6.2_congr}
  Let $\aaa$ be an ideal of $K$ and let $\beta \in \OO_K$ such that $\aaa+
  \id q = \beta\OO_K + \id q = \OO_K$. Moreover, let $\vartheta: \I_K \to
  \RR$ satisfy \eqref{eq:second_sum_thetabound}. Then, for $t\geq 0$,
  \begin{equation*}
    \sum_{\substack{z \in \aaa\smallsetminus\{0\}\\z \equiv \beta \mod \id q\\\N(z\aaa^{-1}) \le t}} \!\!\!\vartheta(z\aaa^{-1}) = \frac{2\pi}{\sqrt{|\Delta_K|}\N\id q}\mathcal{A}(\vartheta(\bbb), \bbb, \id q) t + O_C\left(c_\vartheta\left(\sqrt{\frac{t}{\N\id q}} + \log(t+2)^{C+1}\right)\right).
  \end{equation*}
\end{lemma}

\begin{proof}
  The case $t<1$ can be handled as in Lemma \ref{lem:6.2}, so let $t
  \geq 1$. Using $\vartheta = (\vartheta * \mu_K) * 1$, we see that
  \begin{equation*}
    \sum_{\substack{z \in \aaa\smallsetminus\{0\}\\z \equiv \beta \mod \id q\\\N(z\aaa^{-1}) \le t}} \!\vartheta(z\aaa^{-1}) = \sum_{\N\bbb \leq t}(\vartheta*\mu_K)(\bbb)\sum_{\substack{z \in \aaa\bbb\smallsetminus\{0\}\\z \equiv \beta \mod \id q\\\N(z\aaa^{-1}) \le t}}\!1 = \sum_{\substack{\N\bbb \leq t\\\bbb+ \id q=\OO_K}}(\vartheta*\mu_K)(\bbb)\sum_{\substack{z \in \aaa\bbb\smallsetminus\{0\}\\z \equiv \beta \mod \id q\\\abs{z} \le t\N\aaa}}1\text.
  \end{equation*}
  For the second equality, note that the inner sum is $0$ whenever
  $\id b+\id q\neq \OO_K$. Now $\aaa\bbb+\id q = \OO_K$, so the
  Chinese remainder theorem yields an $x \in \OO_K$ such that
  \begin{equation*}
    \sum_{\substack{z \in \aaa\smallsetminus\{0\}\\z \equiv \beta \mod \id q\\\N(z\aaa^{-1}) \le t}} \vartheta(z\aaa^{-1}) = \sum_{\substack{\N\bbb \leq t\\\bbb+ \id q=\OO_K}}(\vartheta*\mu_K)(\bbb)\sum_{\substack{z \in\OO_K^{\neq 0}\\z \equiv x \mod \aaa\bbb\id q\\\abs{z} \le t\N\aaa}}1\text.
  \end{equation*}
  We use Lemma \ref{lem:lattice_points} to estimate the inner sum and
  obtain
  \begin{equation*}
    \sum_{\substack{z \in \aaa\smallsetminus\{0\}\\z \equiv \beta \mod \id q\\\N(z\aaa^{-1}) \le t}} \vartheta(z\aaa^{-1}) = \sum_{\substack{\N\bbb \leq t\\\id b+\id q = \OO_K}}(\vartheta * \mu_K)(\bbb)\left(\frac{2 \pi t}{\sqrt{|\Delta_K|}\N\bbb\N\id q} + O\left(\sqrt{\frac{t}{\N\id b\N\id q}} + 1\right)\right),
  \end{equation*}
  which we expand to the main term in the lemma plus an error term
  \begin{equation*}
    \ll \frac{t}{\N\id q}\sum_{\N\bbb > t}\frac{|(\vartheta * \mu_K)(\bbb)|}{\N\bbb} + \sqrt\frac{t}{\N\id q}\sum_{\N\bbb \leq t}\frac{|(\vartheta * \mu_K)(\bbb)|}{\sqrt{\N\bbb}} + \sum_{\N\bbb \leq t}|(\vartheta * \mu_K)(\bbb)|\text.
  \end{equation*}
  By \eqref{eq:second_sum_thetabound} and Lemma \ref{lem:3.4}, the
  first part of the error term is $\ll_C c_\vartheta\N\id
  q^{-1}\log(t+2)^C$, the second part is $\ll_C
  c_\vartheta\sqrt{t/\N\id q}$, and the third part is
  $\ll_Cc_\vartheta\log(t+2)^{C+1}$.
\end{proof}

\begin{lemma}\label{lem:rho}
  Using the notation from the beginning of this section, we have
  \begin{align}\label{eq:lem_rho}
    \sum_{\substack{z \in \OO^{\neq 0}\\\N(z\OO^{-1})\leq t}}\vartheta(z\OO^{-1})\sum_{\substack{\rho \mod \id q\\\rho\OO_K+\id q = \OO_K\\\rho^n \equiv_{\id q} A z}}1 &= \frac{2\pi}{\sqrt{|\Delta_K|}}\phi_K^*(\id q)\mathcal{A}(\vartheta(\aaa), \aaa, \id q)t\\
    &+ O(c_\vartheta(\sqrt{\N\id q t} + \N\id
    q\log(t+2)^{C+1}))\text.\nonumber
  \end{align}
\end{lemma}

\begin{proof}
  Denote the expression on the left-hand side of \eqref{eq:lem_rho} by
  $L$. Since $v_\p(A\OO) = 0$ for all $\p \mid \id q$, we can, by weak
  approximation, find $A_1 \in \OO^{-1}$, $A_2 \in \OO_K$ such that $A =
  A_1/A_2$ and $A_1\OO + \id q = A_2\OO_K+\id q = \OO_K$. Changing the
  order of summation, we obtain
  \begin{equation*}
    L= \sum_{\substack{\rho \mod \id q\\\rho\OO_K+\id q=\OO_K}}\sum_{\substack{z \in \OO^{\neq 0}\\A_1 z \equiv A_2 \rho^n \mod \id q\\\N(z\OO^{-1})\leq t}}\hspace{-0.5cm}\vartheta(z\OO^{-1}) = \sum_{\substack{\rho \mod \id q\\\rho\OO_K+\id q=\OO_K}} \sum_{\substack{A_1 z \in A_1\OO^{\neq 0}\\A_1 z \equiv A_2 \rho^n \mod \id q\\\N(A_1 z(A_1\OO)^{-1})\leq t}}\hspace{-0.3cm}\vartheta(A_1z(A_1\OO)^{-1})\text.
  \end{equation*}
  The lemma now follows from Lemma \ref{lem:6.2_congr} and the trivial
  estimate $\phi_K(\id q) \leq \N\id q$.
\end{proof}

Define $\tilde{\vartheta} : \I_K \to \RR$ by
\begin{equation*}
  \tilde{\vartheta}(\aaa) := \vartheta(\aaa)\sum_{\substack{z \in \OO^{\neq 0}\\z\OO^{-1} = \aaa}}\sum_{\substack{\rho \mod \id q\\\rho\OO_K+\id q = \OO_K\\\rho^n \equiv_{\id q} A z}} 1\text.
\end{equation*}
The first sum is finite, since $|\OO_K^\times|< \infty$. Then
\begin{equation*}
  S(t_1, t_2) = \sum_{\substack{\aaa \in [\OO^{-1}] \cap \I_K\\t_1 < \N\aaa \leq t_2}}\tilde{\vartheta}(\aaa)g(\N\aaa)\text,
\end{equation*}
and by Lemma \ref{lem:rho} we have
\begin{equation*}
  \sum_{\substack{a \in [\OO^{-1}] \cap \I_K\\\N\aaa \leq t}} \tilde{\vartheta}(\aaa) = \frac{2\pi}{\sqrt{|\Delta_K|}}\phi_K^*(\id q)\mathcal{A}(\vartheta(\aaa), \aaa, \id q)t
  + O(c_\vartheta(\sqrt{\N\id q t} + \N\id q\log(t+2)^{C+1}))\text.
\end{equation*}
With (\ref{eq:second_sum_gbound}) and simple calculations, the
proposition now follows from Lemma \ref{lem:3.1}.

\section{Further summations}\label{sec:further_summations}

Here, we show how to evaluate the main term of the second summation as
in \eqref{eq:second_main_term}, once the sums over $\classtuple \in
\classrep^{r+1}$ from Claim~\ref{claim:passage} and over elements
$\ee'' \in \OO_{1*} \times \dots \times \OO_{r+1*}$ have been
transformed into a sum over ideals $(\aaa_1, \ldots,
\aaa_{r+1})\in\I_K^{r+1}$ (see
Lemma~\ref{lem:A3_second_summation_ideals_M11}, for example).

In this section, $K$ can be an arbitrary number field of degree $d
\geq 2$. For $K = \QQ$, we refer to \cite{MR2520770}.
Let $r \in \ZZp$, $s \in \{0, 1\}$. We consider functions $V :
\RR^{r+s}_{\ge 1}\times \RR_{\ge 3} \to \RR_{\ge 0}$ similar to the
ones in \cite[Proposition~3.9, Proposition~3.10]{MR2520770}. That is,
we consider three cases:
\begin{enumerate}[(a)]
\item We have $s = 0$ and
  \begin{equation*}
    V(t_1, \ldots, t_r; B) \ll \frac{B}{t_1 \cdots t_r}.
  \end{equation*}
\item We have $s = 1$ and there exist $k_2$, $\ldots$, $k_{r} \in
  \RR$, $k_{1}$, $k_{r+1} \in \RR_{\neq 0}$, $a \in \RR_{> 0}$ with
  \begin{equation*}
    V(t_1, \ldots, t_{r+1}; B) \ll \frac{B}{t_1 \cdots t_{r+1}}\cdot\left(\frac{B}{t_1^{k_1}\cdots t_{r+1}^{k_{r+1}}}\right)^{-a}\text.
  \end{equation*}
  Moreover, $V(t_1, \ldots, t_{r+1}; B) = 0$ unless $t_1^{k_1}\cdots
  t_{r+1}^{k_{r+1}} \leq B$.
\item We have $s = 1$ and there exist $k_2$, $\ldots$, $k_{r} \in
  \RR$, $k_1$, $k_{r+1} \in \RR_{\neq 0}$, $a$, $b\in \RR_{>0}$ with
  \begin{equation*}
    V(t_1, \ldots, t_{r+1}; B) \ll \frac{B}{t_1 \cdots t_{r+1}}\cdot\min\left\{\left(\frac{B}{t_1^{k_1}\cdots t_{r+1}^{k_{r+1}}}\right)^{-a}, \left(\frac{B}{t_1^{k_1}\cdots t_{r+1}^{k_{r+1}}}\right)^b\right\}\text.
  \end{equation*}
\end{enumerate}
Additionally, we assume that $V(t_1, \ldots, t_{r+s}) = 0$ unless
$t_1, \ldots, t_{r+s} \leq B$, and that there is a constant $R(V)$
such that for all fixed $t_1$, $\ldots$, $t_{r+s-1}$, $B$, there is a
partition of $[1, B]$ into at most $R(V)$ intervals on whose interior
$V(t_1, \ldots, t_{r+s}; B)$, considered as a function of $t_{r+s}$,
is continuously differentiable and monotonic. We note that (b) implies
(c) for any $b > 0$.
\begin{lemma}\label{lem:3.6}
  Let $V(t_1, \ldots, t_{r+s}; B)$ be as above, $t_{r+s} \geq 1$, $B
  \geq 3$. Then
  \begin{equation*}
    \sum_{\aaa_1, \ldots, \aaa_{r+s-1}\in\I_K}V(\N\aaa_1, \ldots, \N\aaa_{r+s-1}, t_{r+s};B) \ll \frac{B(\log B)^{r-1}}{t_{r+s}}\text.
  \end{equation*}
\end{lemma}

\begin{proof}
  In case (a) this follows immediately from Lemma \ref{lem:3.4} with
  $C=0$, $\kappa=1$ applied $r$ times. In case (b), we apply Lemma
  \ref{lem:3.4} with $C=0$, $\kappa = 1-ak_{1}$ to the sum over
  $\aaa_{1}$ and then proceed as in case (a).

  In case (c), we split the sum over $\aaa_{1}$ into two sums: One
  over all $\aaa_{1}$ with $\N\aaa_1^{k_1}\cdots
  \N\aaa_{r+1}^{k_{r+1}} \leq B$ and one where the opposite inequality
  holds. For the first, we use $V(\N\aaa_1, \ldots, \N\aaa_{r},
  t_{r+1};B) \ll B/(\N\aaa_1\cdots
  \N\aaa_{r}t_{r+1})(B/(\N\aaa_1^{k_1}\cdots
  \N\aaa_{r}^{k_r}t_{r+1}^{k_{r+1}}))^{-a}$ and proceed as in case
  (b). For the second sum, we use $V(\N\aaa_1, \ldots, \N\aaa_{r},
  t_{r+1};B) \ll B/(\N\aaa_1\cdots \N\aaa_{r}
  t_{r+1})(B/(\N\aaa_1^{k_1}\cdots \N\aaa_{r}^{k_r}
  t_{r+1}^{k_{r+1}}))^{b}$ and apply Lemma \ref{lem:3.4} with $C=0$,
  $\kappa=1+bk_{1}$. The remaining summations over $\aaa_2$, $\ldots$,
  $\aaa_{r}$ are again handled as in case (a).
\end{proof}

\begin{prop}\label{prop:3.9}
  Let $V$ be as above and $\theta \in \Theta_{r+s}'(C)$ for some $C
  \in \ZZp$. Then
  \begin{align*}
    &\sum_{\aaa_1, \ldots, \aaa_{r+s}} \theta(\aaa_1, \dots,
    \aaa_{r+s})V(\N\aaa_1, \dots,
    \N\aaa_{r+s};B)\\
    &= \rho_K h_K \!\!\sum_{\aaa_1, \ldots, \aaa_{r+s-1}}\!\!\A(\theta(\aaa_1, \dots, \aaa_{r+s}),\aaa_{r+s}) \int_1^\infty V(\N\aaa_1, \dots, \N\aaa_{r+s-1}, t_{r+s}; B)\dd t_{r+s}\\
    &+ O_{V,C}(B(\log B)^{r-1}\log\log B)\text.
  \end{align*}
\end{prop}

\begin{proof}
  This is mostly analogous to a special case of \cite[Proposition~3.9,
  Proposition~3.10]{MR2520770}, but we could simplify the third step
  significantly.
 
  We define $T := (\log B)^{d((2C-1)(r+s-1)+s)}$ and proceed in three
  steps:
  \begin{enumerate}
  \item Bound
    \begin{equation*}
      \sum_{\substack{\aaa_1, \ldots, \aaa_{r+s}\\\N\aaa_{r+s} < T}} \theta(\aaa_1, \ldots, \aaa_{r+s})V(\N\aaa_1, \ldots, \N\aaa_{r+s};B)
    \end{equation*}
  \item Bound the sum over $\aaa_1$, $\ldots$, $\aaa_{r+s-1}$ of
    \begin{align*}
      &\sum_{\N\aaa_{r+s} \geq T}\theta(\aaa_1, \ldots, \aaa_{r+s})V(\N\aaa_1,
      \ldots, \N\aaa_{r+s};B)\\ -\ &\rho_Kh_K\A(\theta(\aaa_1, \dots,
      \aaa_{r+s}),\aaa_{r+s}) \int_T^\infty V(\N\aaa_1, \dots, \N\aaa_{r+s-1},
      t_{r+s}; B)\dd t_{r+s}\text.
    \end{align*}
  \item Bound
    \begin{equation*}
      \sum_{\substack{\aaa_1, \ldots, \aaa_{r+s-1}}}\A(\theta(\aaa_1, \dots, \aaa_{r+s}),\aaa_{r+s})\int_{1}^TV(\N\aaa_1, \dots, \N\aaa_{r+s-1}, t_{r+s};B)\dd t_{r+s}\text. 
    \end{equation*}
  \end{enumerate}
  Using $0 \le\theta(\aaa_1, \dots, \aaa_{r+s}) \le 1$ and Lemma
  \ref{lem:3.6}, we see that the expression in step (1) is indeed
  bounded by
  \begin{equation*}
    \sum_{\substack{\aaa_1, \ldots, \aaa_{r+s}\\\N\aaa_{r+s} < T}} V(\N\aaa_1, \ldots, \N\aaa_{r+s};B) \ll \sum_{\N\aaa_{r+s} < T} \frac{B(\log B)^{r-1}}{\N\aaa_{r+s}} \ll B(\log B)^{r-1}\log\log B\text.
  \end{equation*}
  Analogously, since $\A(\theta(\aaa_1, \dots, \aaa_{r+s}),\aaa_{r+s}) \in
  \Theta_{r+s-1}'(2C)$, the expression in step (3) is bounded by
  \begin{align*}
    \sum_{\aaa_1, \ldots, \aaa_{r+s-1}}\int\limits_{1}^T V(\N\aaa_1, \ldots,
    \N\aaa_{r+s-1}, t_{r+s};B)\dd t_{r+s} &\ll \int\limits_{1}^T
    \frac{B(\log B)^{r-1}}{t_{r+s}}\dd t_{r+s}\\ &\ll B(\log
    B)^{r-1}\log\log B\text.
  \end{align*}

  For step (2), we note that in all three cases (a), (b), (c) we have
  \begin{equation*}
    V(t_1, \dots, t_{r+s};B) \ll \frac{B}{t_1\cdots t_{r+s}}.
  \end{equation*}
  By Corollary \ref{cor:6.9}, we may apply Lemma
  \ref{lem:3.1} with $m=1$, $c_1 = (2C)^{\omega(\aaa_1, \ldots,
    \aaa_{r+s-1})}$, $b_1 = 1-1/d$, $k_1=0$, $c_g = B/(\N\aaa_1 \cdots
  \N\aaa_{r+s-1})$, $a=-1$ for the sum over $\aaa_{r+s}$. We obtain an
  error term of order
  \begin{align*}
    \ll_{V, C} \sum_{\aaa_1, \ldots, \aaa_{r+s-1}}\frac{(2C)^{\omega(\aaa_1,
        \ldots, \aaa_{r+s-1})}B}{\N\aaa_1\cdots \N\aaa_{r+s-1}}T^{-1/d}
    &\ll_C B(\log B)^{(2C)(r+s-1)}T^{-1/d}\\ &\ll B(\log
    B)^{r-1}\text.\qedhere
  \end{align*}
\end{proof}

Let $V_{r+1} := V$ be as in cases (b), (c) at the start of this
section. For all $l \in \{0, \ldots, r\}$, we define
\begin{equation*}
  V_l(t_1, \ldots, t_l; B) := \int_{t_{l+1}, \ldots, t_{r+1} \ge 1}V(t_1, \ldots, t_{r+1}; B)\dd t_{l+1} \cdots \dd t_{r+1}\text.
\end{equation*}
For $l \geq 1$, and fixed $t_1$, $\ldots$, $t_{l-1}$, $B$, we
additionally require that there is a partition of $[1, B]$ into at
most $R(V)$ intervals on which $V_l(t_1, \ldots, t_l; B)$, as a
function of $t_l$, is continuously differentiable and monotonic. For
$\theta \in \Theta_{r+1}'(C)$, let
\begin{equation*}
  \theta_l(\aaa_1, \ldots, \aaa_l) := \A(\theta(\aaa_1, \ldots, \aaa_{r+1}), \aaa_{r+1}, \ldots, \aaa_{l+1})\in \Theta_l'(2^{r-l+1}C)\text.
\end{equation*}
The following proposition is analogous to \cite[Proposition~4.3,
Remark~4.4]{MR2520770}.
\begin{prop}\label{prop:4.3}
  Let $V$ be as above and $\theta \in \Theta_{r+1}'(C)$. Then
  \begin{align*}
    &\sum_{\aaa_1, \dots, \aaa_{r+1}} \theta(\aaa_1, \dots, \aaa_{r+1})
    V(\N\aaa_1, \dots, \N\aaa_{r+1};B)\\ &= (\rho_Kh_K)^{r+1}\theta_0
    V_0(B) + O_{V,C}(B(\log B)^{r-1}\log\log B).
  \end{align*}
\end{prop}

\begin{proof}
  By a similar argument as in Lemma \ref{lem:3.6}, we see that, for $l
  \in \{1, \ldots, r\}$,
  \begin{equation*}
    V_l(t_1, \ldots, t_l; B) \ll \frac{B(\log B)^{r-l}}{t_1\cdots t_l}\text.
  \end{equation*}
  Since $\theta_l(\aaa_1, \ldots, \aaa_l)\in \Theta_l'(2^{r-l+1}C)$,
  we can apply Proposition \ref{prop:3.9} inductively to $V_{r+1}$,
  $V_r$, $V_{r-1}/\log B$, $\ldots$, $V_1/(\log B)^{r-1}$.
\end{proof}
Note that $\theta_0$ can be computed by Lemma
\ref{lem:repeated_average}.

\section{The factor $\alpha$}\label{sec:strategy_adjustion}

Let $K$ be an imaginary quadratic field. Let $S$ be a split singular del Pezzo
surface of degree $d=9-r$ over $K$, with minimal desingularization $\tS$. The
final result of our summation process is typically provided by
Proposition~\ref{prop:4.3}. To derive Manin's conjecture as in
Theorem~\ref{thm:a3_main} from this, it remains to compare the integral
$V_0(B)$ with $\alpha(\tS)\pi^{r+1}\omega_\infty B(\log B)^r$. Here,
$\alpha(\tS)$ is a constant defined in \cite[D\'efinition~2.4]{MR1340296},
\cite[Definition~2.4.6]{MR1369408} that is expected be a factor of the leading
constant $c_{S,H}$ in Manin's conjecture (\ref{eq:manin_formula}).

For a split singular del Pezzo surface $S$ of degree $d \le 7$, its value can
be computed by \cite[Theorem~1.3]{MR2377367} as
\begin{equation}\label{eq:alpha}
\alpha(\tS) = \frac{\alpha(S_0)}{|W|}
\end{equation}
where $S_0$ is a split ordinary del Pezzo surface of the same degree and $|W|$
is the order of the Weyl group $W$ associated to the singularities of $S$. For
example, $|W| = (n+1)!$ if $S$ has precisely one singularity whose type is
$\mathbf{A}_n$. The value of $\alpha(S_0)$ can be computed by
\cite[Theorem~4]{MR2318651}, with $\alpha(S_0) = 1/180$ in degree $4$.

To rewrite $\alpha(\tS)$ as an integral, it is most convenient to work with
\cite[Definition~1.1]{arXiv:1202.6287}, giving
\begin{equation*}
  \alpha(\tS) := (r+1)\cdot \vol\{x \in \Lambda_{\mathrm{eff}}^\vee(\tS) \mid
  (x,-K_\tS) \le 1\},
\end{equation*}
since $\Pic(\tS)$ has rank $r+1$, where $\Lambda_{\mathrm{eff}}^\vee(\tS)
\subset (\Pic(\tS) \otimes_\ZZ \RR)^\vee$ is the dual of the effective cone of
$\tS$ (which is generated by the classes of the negative curves since $d \le
7$), $(\cdot,\cdot)$ is the natural pairing between $\Pic(\tS) \otimes_\ZZ
\RR$ and its dual space, and the volume is normalized such that
$\Pic(\tS)^\vee$ has covolume $1$.

Suppose that the negative curves on $\tS$ are $E_1, \dots, E_{r+1+s}$, for
some $s \ge 0$, where $E_1, \dots, E_{r+1}$ are a basis of $\Pic(\tS)$; for
example, this holds in the ordering chosen in
Section~\ref{sec:passage}. Expressing $-K_\tS$ and $E_{r+2},\dots, E_{r+1+s}$
in terms of this basis, we have
\begin{equation*}
  [-K_\tS]= \sum_{j=1}^{r+1} c_j[E_j]
\end{equation*}
and, for $i=1, \dots, s$,
\begin{equation*}
  [E_{r+1+i}] = \sum_{j=1}^{r+1} b_{i,j}[E_j]
\end{equation*}
for some $b_{i,j}, c_j \in \ZZ$. 

\begin{lemma}\label{lem:alpha_integral}
  With the above notation, assume that $c_{r+1} > 0$. Define, for $j = 1,
  \dots, r$ and $i=1, \dots, s$,
  \begin{equation*}
    a_{0,j}:= c_j, \quad a_{i,j}:=b_{i,r+1}c_j-b_{i,j}c_{r+1},\quad 
    A_0:=1,\quad A_i:=b_{i,r+1}.
  \end{equation*}
  Then
  \begin{equation*}
    \alpha(\tS)(\log B)^r = \frac{1}{c_{r+1}\pi^r} \int_{R_1(B)}
    \frac{1}{\abs{\e_1\cdots \e_r}} \dd\e_1\cdots\dd\e_r
  \end{equation*}
  with a domain of integration
  \begin{equation*}
    R_1(B)
    :=\left\{(\e_1, \dots, \e_r) \in \CC^r \WHERE
      \begin{aligned}
        &\abs{\e_1}, \dots, \abs{\e_r} \ge 1,\\
        &\text{$\prod_{j=1}^r
          \abs{\e_j}^{a_{i,j}} \le B^{A_i}$ for all $i \in \{0, \dots, s\}$}
      \end{aligned}
\right\}.
  \end{equation*}
\end{lemma}

\begin{proof}
  Since $[E_1], \dots, [E_{r+1+s}]$ generate the effective cone of $\tS$, the
  value of $\alpha(\tS)$ is
  \begin{equation*}
    (r+1)\cdot \vol\left\{(t_1', \dots, t_{r+1}') \in
      \RRnn^{r+1} \Where \sum_{j=1}^{r+1} b_{i,j}t_j' \ge 0\ (i=1, \dots, s),
      \ \sum_{j=1}^{r+1}c_jt_j' \le 1\right\}.
  \end{equation*}
  We make a linear change of variables $(t_1, \dots,
  t_r,t_{r+1})=(t_1', \dots, t_r', c_1t_1'+\dots+c_{r+1}t_{r+1}')$,
  with Jacobian $c_{r+1}$. This transforms the polytope in the
  previous formula into a pyramid whose base is $R_0 \times \{1\}$ in
  the hyperplace $\{t_{r+1}=1\}$ in $\RR^{r+1}$, and whose apex is the
  origin, where
  \begin{equation*}
    R_0 := \left\{(t_1, \dots, t_r) \in \RRnn^r \Where
      \text{$\sum_{j=1}^r a_{i,j}t_j \le A_i$ for all $i \in \{0, \dots,
        s\}$}\right\}.
  \end{equation*}
  This pyramid has volume $(r+1)^{-1}\vol R_0$ since its height is $1$
  and its dimension is $r+1$. Writing $\vol R_0$ as an integral, we get
  \begin{equation*}
    \alpha(\tS) = \frac{1}{c_{r+1}} \int_{(t_1, \dots, t_r) \in R_0} \dd t_r
    \cdots \dd t_1,
  \end{equation*}
  where the factor $c_{r+1}^{-1}$ appears because of our change of
  coordinates. Now the change of coordinates $\e_i=B^{t_i}$ for $i \in \{1,
  \dots, r\}$ gives a real integral with the factor $(\log B)^r$. The final
  complex integral with the factor $\pi^r$ is obtained via polar coordinates.
\end{proof}

\section{The quartic del Pezzo surface of type $\Athree$ with five
  lines}
\label{sec:A3}

Let $S \subset \PP^4_K$ be the anticanonically embedded del Pezzo
surface defined by \eqref{eq:def_S}. In this section, we apply our
general techniques to prove Manin's conjecture for $S$
(Theorem~\ref{thm:a3_main}).

Our surface $S$ contains precisely one singularity $(0:0:0:0:1)$ (of
type $\Athree$) and the five lines $\{x_0=x_1=x_2=0\}$, $\{x_0=x_2=x_3=0\}$,
  $\{x_0=x_3=x_4=0\}$, $\{x_1=x_2=x_3=0\}$,
$\{x_1=x_3=x_4=0\}$. Let $U$ be the complement of these lines in $S$.

By \cite{MR2753646, arXiv:1212.3518}, $S$ is not an equivariant
compactification of an algebraic group, so that Manin's conjecture does not
follow from \cite{MR1620682, MR1906155, MR2858922}.

\subsection{Passage to a universal torsor}
To parameterize the rational points on $U\subset S$ by integral points
on an affine hypersurface, we apply the strategy described in
Section~\ref{sec:passage}, based on the description of the Cox ring of
its minimal desingularization $\tS$ in \cite{math.AG/0604194}. In
particular, we will refer to the extended Dynkin diagram in
Figure~\ref{fig:A3_dynkin} encoding the configuration of curves $E_1,
\dots, E_9$ corresponding to generators of $\Cox(\tS)$. Here, a vertex
marked by a circle (resp. a box) corresponds to a $(-2)$-curve
(resp. $(-1)$-curve), and there are $([E_j], [E_k])$ edges between the
vertices corresponding to $E_j$ and $E_k$.

\begin{figure}[ht]
  \centering
  \[\xymatrix@R=0.05in @C=0.05in{\li{7} \ar@{-}[r] \ar@{-}[dd] \ar@{-}[dr] & \li{4} \ar@{-}[r] & \ex{1} \ar@{-}[dr]\\
    & E_9 \ar@{-}[r] & \li{5} \ar@{-}[r] & \ex{2}\\
    \li{8} \ar@{-}[r] \ar@{-}[ur] & \li{6} \ar@{-}[r] & \ex{3}
    \ar@{-}[ur]}\]
  \caption{Configuration of curves on $\tS$.}
  \label{fig:A3_dynkin}
\end{figure}
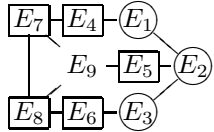

For any given $\classtuple = (C_0, \dots, C_5) \in \classrep^6$, we
define $u_\classtuple := \N(C_0^3 C_1^{-1}\cdots C_5^{-1})$ and
\begin{equation}\label{eq:A3_def_Oj}
  \begin{aligned}
    \OO_1 &:= C_1C_4^{-1},  & \OO_2 &:= C_0C_1^{-1}C_2^{-1}C_3^{-1},  & \OO_3 &:= C_2C_5^{-1}, \\
    \OO_4 &:= C_4,  & \OO_5 &:= C_3,  & \OO_6 &:= C_5,  \\
    \OO_7 &:= C_0C_1^{-1}C_4^{-1},  & \OO_8 &:= C_0 C_2^{-1} C_5^{-1},  &
    \OO_9 &:= C_0 C_3^{-1}.
  \end{aligned}
\end{equation}
Let
\begin{equation*}
  \OO_{j*} :=
  \begin{cases}
    \OO_j^{\neq 0}, & j \in \{1,\ldots, 8\},\\
    \OO_j, & j = 9.
  \end{cases}
\end{equation*}
For $\eta_j \in \OO_j$, we define
\begin{equation*}
  \eI_j := \e_j \OO_j^{-1}\text.
\end{equation*}
For $B \geq 0$, let $\mathcal{R}(B)$ be the set of all $(\e_1, \ldots, \e_8)
\in \CC^8$ with $\e_5\neq 0$ and
\begin{align}
      \abs{\e_1^2\e_2^2\e_3\e_4^2\e_5\e_7} &\leq B,\label{eq:A3_height_1}\\
      \abs{\e_1\e_2^2\e_3^2\e_5\e_6^2\e_8} &\leq B,\label{eq:A3_height_2}\\
      \abs{\e_1^2\e_2^3\e_3^2\e_4\e_5^2\e_6} &\leq B,\label{eq:A3_height_3}\\
      \abs{\e_1\e_2\e_3\e_4\e_6\e_7\e_8} &\leq B,\label{eq:A3_height_4}\\
      \abs{\frac{\e_1\e_4^2\e_7^2\e_8 + \e_3\e_6^2\e_7\e_8^2}{\e_5}} &\leq B\label{eq:A3_height_5}\text.
\end{align}
Moreover, let $M_\classtuple(B)$ be the set of all 
\begin{equation*}
  (\e_1, \ldots, \e_9) \in \OO_{1*} \times \cdots \times \OO_{9*} 
\end{equation*}
that satisfy the \emph{height conditions}
\begin{equation}\label{eq:A3_height}
     (\e_1, \ldots, \e_8) \in \mathcal{R}(u_\classtuple B)\text,
\end{equation}
the \emph{torsor equation}
\begin{equation}\label{eq:A3_torsor}
  \e_1\e_4^2\e_7 + \e_3\e_6^2\e_8 + \e_5\e_9 = 0,
\end{equation}
and the \emph{coprimality conditions}
\begin{equation}\label{eq:A3_coprimality}
  \eI_j + \eI_k = \OO_K \text{ for all distinct nonadjacent vertices $E_j$, $E_k$ in Figure~\ref{fig:A3_dynkin}.}
\end{equation}

\begin{lemma}\label{lem:A3_passage_to_torsor}
  Let $K$ be a imaginary quadratic field. Then
  \begin{equation*}
    N_{U,H}(B) = \frac{1}{\omega_K^6}\sum_{\classtuple \in \classrep^6}|M_\classtuple(B)|\text.
  \end{equation*}
\end{lemma}

\begin{proof}
  We apply the strategy from Section \ref{sec:passage}. We work with
  the data in \cite{math.AG/0604194}. For our surface $S$,
  Claim~\ref{claim:passage} specializes precisely to the statement of
  our lemma (where \eqref{eq:A3_height_5} is $\abs{\e_7\e_8\e_9} \le
  B$ with $\e_9$ eliminated using \eqref{eq:A3_torsor}).

  We prove it via the induction process described in
  Claim~\ref{claim:passage_step}. It is based on the construction of
  the minimal desingularization $\pi: \tS \to S$ by the following
  sequence of blow-ups $\rho=\rho_1\circ\dots\circ\rho_5: \tS \to
  \PP^2_K$. Starting with the curves $E^{(0)}_7:= \{y_0=0\}$,
  $E^{(0)}_8 := \{y_1=0\}$, $E^{(0)}_2 := \{y_2 = 0\}$, $E^{(0)}_9 :=
  \{-y_0 - y_1 = 0\}$ in $\PP^2_K$,
  \begin{enumerate}
  \item blow up $E^{(0)}_2\cap E^{(0)}_7$, giving $E_1^{(1)}$,
  \item blow up $E^{(1)}_2\cap E^{(1)}_8$, giving $E^{(2)}_3$,
  \item blow up $E^{(2)}_2\cap E^{(2)}_9$, giving $E^{(3)}_5$,
  \item blow up $E^{(3)}_1\cap E^{(3)}_7$, giving $E^{(4)}_4$,
  \item blow up $E^{(4)}_3\cap E^{(4)}_8$, giving $E^{(5)}_6$.
  \end{enumerate}

  The inverse $\pi \circ \rho^{-1}: \PP^2_K \rto S$ of the
  projection $\phi = \rho \circ \pi^{-1} : S \rto \PP^2_K$, $(x_0 :
  \cdots : x_4) \mapsto (x_0 : x_1 : x_2)$ is given by
  \begin{equation}\label{eq:A3_parameterization}
    (y_0 : y_1 : y_2) \mapsto (y_0y_2^2 : y_1y_2^2 : y_2^3 : y_0y_1y_2 : -y_0y_1(y_0 + y_1)).
  \end{equation}
  
  In our case, the map $\Psi$ appearing in
  Claim~\ref{claim:passage_step} sends $(\e_1, \ldots, \e_9)$ to
  \begin{equation*}
    (\e_1^2\e_2^2\e_3\e_4^2\e_5\e_7, \e_1\e_2^2\e_3^2\e_5\e_6^2\e_8,
    \e_1^2\e_2^3\e_3^2\e_4\e_5^2\e_6, \e_1\e_2\e_3\e_4\e_6\e_7\e_8, \e_7\e_8\e_9).
  \end{equation*}
 
  Claim~\ref{claim:passage_step} holds for $i=0$ by
  Lemma~\ref{lem:passage_start} since the map $\psi : \PP^2_K \rto S$
  obtained from $\Psi$ by the substitution
  $(\e_1, \dots, \e_8) \mapsto (1, y_2, 1, 1, 1, 1, y_0, y_1, -y_0-y_1)$
  as in \eqref{eq:psii}, \eqref{eq:psi} agrees with $\pi \circ \rho^{-1}$ on $\PP^2_K
  \smallsetminus \{(1:0:0),(0:1:0),(1:-1:0)\}$.

  Since the five blow-ups described above satisfy the assumptions of
  Lemma~\ref{lem:passage_step}, Claim~\ref{claim:passage_step} follows
  by induction for $i=1, \dots, 5$.

  Hence $\Psi$ induces a $\omega_K^6$-to-$1$ map from the set of all
  $(\e_1, \ldots, \e_9) \in \bigcup_{\classtuple \in
    \classrep^6}\OO_{1*}\times\cdots\times\OO_{9*}$ satisfying
  \eqref{eq:A3_torsor}, \eqref{eq:A3_coprimality}, $H(\Psi(\e_1,
  \ldots, \e_9)) \leq B$ to the set of $K$-rational points on $U$ of
  height bounded by $B$.  One easily sees that
  \eqref{eq:A3_coprimality} implies that
  \begin{equation*}
    \e_1^2\e_2^2\e_3\e_4^2\e_5\e_7\OO_K + \ldots +  \e_7\e_8\e_9\OO_K = C_0^3C_1^{-1}\cdots C_5^{-1}.
  \end{equation*}
  As discussed after Claim~\ref{claim:passage_step}, this completes
  the proof of Claim~\ref{claim:passage}.
\end{proof}

\subsection{Summations}
In a direct application of Proposition \ref{prop:first_summation} to
$M_\classtuple(B)$, our height conditions would not
yield sufficiently good estimates for the sum over the error terms, so
we consider two cases: Let $M_{\classtuple}^{(8)}(B)$ be
the set of all $(\e_1, \ldots, \e_9) \in M_\classtuple(B)$ with $\N\eI_8\ge \N\eI_7$, and let $M_{\classtuple}^{(7)}(B)$ be the set of all $(\e_1,\ldots, \e_9) \in M_\classtuple(B)$ with $\N\eI_7 > \N\eI_8$. Moreover, let
\begin{equation*}
  N_8(B) := \frac{1}{\omega_K^6}\sum_{\classtuple \in \classrep^6}|M_{\classtuple}^{(8)}(B)|\text,
\end{equation*}
and define $N_7(B)$ analogously. Then clearly $N_{U,H}(B) = N_8(B) + N_7(B)$.

\subsubsection{The first summation over $\e_8$ in
  $M^{(8)}_\classtuple(B)$ with dependent $\e_9$}
\begin{lemma}\label{lem:A3_first_summation_M1}
  Write $\ee' := (\e_1, \ldots, \e_7)$ and $\eII' := (\eI_1, \ldots,
  \eI_7)$. For $B >0$, $\classtuple \in \classrep^6$, we have
  \begin{equation*}
    |M^{(8)}_\classtuple(B)| = \frac{2}{\sqrt{|\Delta_K|}} \sum_{\ee' \in \OO_{1*} \times \dots \times \OO_{7*}} \theta_8(\eII')V_8(\N\eI_1, \ldots, \N\eI_7; B) + O_\classtuple(B(\log B)^3),
  \end{equation*}
  where
  \begin{equation*}
    V_8(t_1, \ldots, t_7; B) :=  \frac{1}{t_5}\int_{\substack{(\sqrt{t_1}, \ldots, \sqrt{t_7},\e_8)\in \mathcal{R}(B)\\\abs{\e_8}\geq t_7}} \dd \e_8
  \end{equation*}
  with a complex variable $\e_8$, and where
  \begin{equation*}
    \theta_8(\eII') := \prod_{\id p} \theta_{8,\id p}(J_{\id p}(\eII'))
  \end{equation*}
  with $J_{\id p}(\eII') := \{j \in \{1, \dots, 7\}\ :\ \id p \mid
  \eI_j\}$ and
  \begin{equation*}
    \theta_{8,\id p}(J) :=
    \begin{cases}
      1 &\text{ if }J = \emptyset,\{5\},\{6\},\{7\},\\
      1-\frac{1}{\N\id p} &\text{ if }J = \{1\},\{3\},\{4\},\{1,2\}, \{1,4\}, \{2,3\},\{2,5\},\{3,6\},\{4,7\}\\
      1-\frac{2}{\N\id p} &\text{ if }J = \{2\},\\
      0 &\text{otherwise.}
    \end{cases}
  \end{equation*}
  Moreover, the same asymptotic formula holds if we replace the
  condition $\N\eI_8 \geq \N\eI_7$ in the definition of
  $M^{(8)}_\classtuple(B)$ by $\N\eI_8 > \N\eI_7$.
\end{lemma}

\begin{proof}
  We express the condition $\N\eI_8 \geq \N\eI_7$ as
  \begin{equation*}
    \abs{\sqrt{\N\eI_7}} \leq \abs{\sqrt{\N\OO_8^{-1}}\e_8}.
  \end{equation*}
  Let $\ee' \in \OO_{1*} \times \cdots \times \OO_{7*}$. By Lemma
  \ref{lem:preimage_of_unit_disc}, the subset $\mathcal{R}(\ee';u_\classtuple B)
  \subset \CC$ of all $\e_8$ with $(\e_1, \ldots, \e_8) \in
  \mathcal{R}(u_\classtuple B)$ and $\N\eI_8 \geq \N\eI_7$ is of class $m$, where
  $m$ is an absolute constant. Moreover, by Lemma \ref{lem:5.1_4},
  (1), applied to \eqref{eq:A3_height_5} with $u_\classtuple B$ instead of $B$, we see that $\mathcal{R}(\ee';u_\classtuple B)$ is
  contained in the union of $2$ balls of radius
  \begin{equation*}
    R(\ee'; u_\classtuple B) := (u_\classtuple B\abs{\e_3^{-1}\e_5\e_6^{-2}\e_7^{-1}})^{1/4}\ll_\classtuple(B\N\eI_3^{-1}\N\eI_5\N\eI_6^{-2}\N\eI_7^{-1})^{1/4}\text.
  \end{equation*}

  We may sum over all $\e_8 \in \OO_8$ instead of $\e_8 \in
  \OO_{8*}$, since $0 \notin \mathcal{R}(\ee'; u_\classtuple B)$. We apply
  Proposition \ref{prop:first_summation} with $(A_1, A_2, A_0) := (1,
  4, 7)$, $(B_1, B_2, B_0) := (3, 6, 8)$, $(C_1, C_0) := (5, 9)$, $D :=
  2$, and $u_\classtuple B$ instead of $B$. We choose $\Ao
  $ and $\At $ as suggested by Remark
  \ref{rem:first_summation_choice_of_pi}. Then
  \begin{equation*}
    V_1(\ee';u_\classtuple B) = \frac{1}{\N(\eI_5 \OO_8)}\int_{\e_8 \in \mathcal{R}(\ee';u_\classtuple B)} \dd \e_8\text.
  \end{equation*}
  A straightforward computation shows that $\e_8 \in \mathcal{R}(\ee'; u_\classtuple B)$ if and only if
  \begin{equation*}
    (\sqrt{\N\eI_1}, \ldots, \sqrt{\N\eI_7}, \varphi(\e_8)) \in \mathcal{R}(B) \text{ and } \abs{\varphi(\e_8)}\geq \N\eI_7\text,
  \end{equation*}
  where $\varphi : \CC \to \CC$ is given by $z \mapsto e^{i
    \arg(\e_3\e_6^2/(\e_1\e_4^2\e_7))}/\sqrt{\N\OO_8}\cdot
  z$. Therefore, $V_1(\ee';u_\classtuple B) = V_8(\N\eI_1, \ldots,
  \N\eI_7;B)$. Moreover, since $b_0 = 1$, Lemma
  \ref{lem:general_comp_theta_1} shows that $\theta_1(\ee') =
  \theta_1'(\eII') = \theta_8(\eII')$, so the main term is as desired.

  The error term from Proposition \ref{prop:first_summation} is
  \begin{equation*}
    \ll \sum_{\ee', \eqref{eq:first_summation_error_sum} }2^{\omega(\eI_2)+\omega(\eI_1\eI_2\eI_3\eI_4)}\left(\frac{R(\ee'; u_\classtuple B)}{\N(\eI_5)^{1/2}}+1\right)\text.
  \end{equation*}
  Using \eqref{eq:A3_height_1}, \eqref{eq:A3_height_2}, the definitions of $u_\classtuple$ and $\OO_j$, and our
  assumption $\N\eI_8 \geq \N\eI_7$, we see that \eqref{eq:first_summation_error_sum} (with $u_\classtuple B$ instead of $B$) implies
  \begin{align}
    \N\eI_1^2\N\eI_2^2\N\eI_3\N\eI_4^2\N\eI_5\N\eI_7 &\leq B\text{,
      and}\label{eq:A3_first_height_cond_ideals}\\
    \N\eI_1\N\eI_2^2\N\eI_3^2\N\eI_5\N\eI_6^2\N\eI_7 &\leq
    B \label{eq:A3_second_height_cond_ideals_I7}\text{.}
  \end{align}
  Let $\aaa \in \I_K$. Since there are at most $|\OO_K^\times| < \infty$
  elements $\e_j \in \OO_j$ with $\eI_j = \aaa$, we can sum over the
  ideals $I_j \in \I_K$ instead of the $\e_j \in \OO_j$. Moreover, we
  can replace \eqref{eq:first_summation_error_sum} by
  \eqref{eq:A3_first_height_cond_ideals} and
  \eqref{eq:A3_second_height_cond_ideals_I7}, and estimate the error
  term by
  \begin{align*}
      &\ll_\classtuple \sum_{\substack{\eI_1 \ldots, \eI_7\\\eqref{eq:A3_first_height_cond_ideals},\eqref{eq:A3_second_height_cond_ideals_I7}}}2^{\omega(\eI_2)+\omega(\eI_1\eI_2\eI_3\eI_4)}\left(\frac{B^{1/4}}{\N\eI_3^{1/4}\N\eI_5^{1/4}\N\eI_6^{1/2}\N\eI_7^{1/4}}+1\right)\\
      &\ll\sum_{\substack{\eI_1, \dots, \eI_5, \eI_7\\\eqref{eq:A3_first_height_cond_ideals}}}\left(\frac{2^{\omega(\eI_2)+\omega(\eI_1\eI_2\eI_3\eI_4)}B^{1/2}}{\N\eI_1^{1/4}\N\eI_2^{1/2}\N\eI_3^{3/4}\N\eI_5^{1/2}\N\eI_7^{1/2}}+\frac{2^{\omega(\eI_2)+\omega(\eI_1\eI_2\eI_3\eI_4)}B^{1/2}}{\N\eI_1^{1/2}\N\eI_2\N\eI_3\N\eI_5^{1/2}\N\eI_7^{1/2}}\right)\\
      &\ll\sum_{\substack{\eI_1, \dots, \eI_5\\\N\eI_j \leq B}}\left(\frac{2^{\omega(\eI_2)+\omega(\eI_1\eI_2\eI_3\eI_4)}B}{\N\eI_1^{5/4}\N\eI_2^{3/2}\N\eI_3^{5/4}\N\eI_4\N\eI_5}+\frac{2^{\omega(\eI_2)+\omega(\eI_1\eI_2\eI_3\eI_4)}B}{\N\eI_1^{3/2}\N\eI_2^2\N\eI_3^{3/2}\N\eI_4\N\eI_5}\right)\\
      &\ll B(\log B)^3.
  \end{align*}
  For the last estimation, we used Lemma \ref{lem:omega} and Lemma \ref{lem:3.4}.

  Let $M_\classtuple^{(8)'}(B)$ be defined as
  $M_\classtuple^{(8)}(B)$, except that the condition $\N\eI_8 \geq
  \N\eI_7$ is replaced by $\N\eI_8 = \N\eI_7$. We apply Proposition
  \ref{prop:first_summation} in an analogous way as above. Since then
  $V_1(\ee'; u_{\classtuple} B) = 0$, we obtain
  $|M_\classtuple^{(8)'}(B)| \ll B(\log B)^3$. This shows the last
  assertion of the lemma.
\end{proof}

For the second summation, we need another dichotomy: Let
$M_\classtuple^{(87)}(B)$ be the main term in the
expression for $|M^{(8)}_\classtuple(B)|$ given in Lemma
\ref{lem:A3_first_summation_M1} with the additional condition $\N\eI_7
> \N\eI_4$ in the sum, and let $M_\classtuple^{(84)}(B)$
be the same main term with the additional condition $\N\eI_4 \geq
\N\eI_7$ in the sum, so that $|M_\classtuple^{(8)}(B)| =
M_\classtuple^{(87)}(B) +
M_\classtuple^{(84)}(B) + O(B(\log B)^3)$. Moreover, let
\begin{equation*}
  N_{87}(B) := \frac{1}{\omega_K^6}\sum_{\classtuple \in \classrep^6}M^{(87)}_\classtuple(B)
\end{equation*}
and define $N_{84}(B)$ analogously. Then $N_8(B) = N_{87}(B) +
N_{84}(B) + O(B(\log B)^3)$.

\subsubsection{The second summation over $\e_7$ in
  $M^{(87)}_\classtuple(B)$}
\begin{lemma}\label{lem:A3_second_summation_M11}
  Write $\ee'' := (\e_1, \dots, \e_6)$. For $B \geq 3$, $\classtuple \in \classrep^6$, we have
  \begin{align*}
    M_\classtuple^{(87)}(B) &= \left(\frac{2}{\sqrt{|\Delta_K|}}\right)^2\sum_{\ee''\in \OO_{1*} \times \dots \times \OO_{6*}} \mathcal{A}(\theta_8(\eII'), \eI_7)V_{87}(\N\eI_1, \ldots, \N\eI_6; B)\\ &+ O_\classtuple(B(\log B)^3).
  \end{align*}
  For $t_1$, $\ldots$, $t_6 \geq 1$,
  \begin{equation*}
    V_{87}(t_1, \ldots, t_6; B) := \frac{\pi}{t_5}\int_{\substack{(\sqrt{t_1}, \ldots, \sqrt{t_7}, \e_8) \in \mathcal{R}(B)\\t_4 < t_7 \leq \abs{\e_8}}} \dd t_7 \dd \e_8,
  \end{equation*}
  with a real variable $t_7$ and a complex variable $\e_8$.
\end{lemma}

\begin{proof}
  We use the strategy described in Section \ref{sec:second_summation}
  in the case $b_0 = 1$. For $\aaa \in \I_K$, $t \geq 1$, let
  $\vartheta(\aaa) := \theta_8(\eI_1, \ldots, \eI_6, \aaa)$ and $g(t) :=
  V_8(\N\eI_1, \ldots, \N\eI_6, t; B)$. Then
  \begin{equation}\label{eq:A3_sec_sum_M1_start}
    M^{(87)}_\classtuple(B) = \frac{2}{\sqrt{|\Delta_K|}}\sum_{\ee'' \in \OO_{1*} \times \dots \times \OO_{6*}}\sum_{\substack{\e_7 \in \OO_{7*}\\\N\eI_7 > \N\eI_4}}\vartheta(\eI_7)g(\N\eI_7)\text.
  \end{equation}
  By Lemma \ref{lem:general_comp_theta_1} and Lemma \ref{lem:6.8}, $\vartheta$ satisfies
  \eqref{eq:second_sum_thetabound} with $C = 0$ and $c_\vartheta = 2^{\omega_K(\eI_1\eI_2\eI_3\eI_5\eI_6)}$.
 
  The first height condition \eqref{eq:A3_height_1} implies that $g(t)
  = 0$ whenever $t > t_2 := B/(\N\eI_1^2\N\eI_2^2\N\eI_3\N\eI_4^2
  \N\eI_5)$. Moreover, applying Lemma \ref{lem:5.1_4}, (2), to the
  fifth height condition \eqref{eq:A3_height_5}, we see that
  \begin{equation*}
    g(t) \ll \frac{1}{\N\eI_5}\cdot \frac{(B\N\eI_5)^{1/2}}{(\N\eI_3\N\eI_6^2t)^{1/2}} = \frac{B^{1/2}}{\N\eI_3^{1/2}\N\eI_5^{1/2}\N\eI_6}t^{-1/2}\text.
  \end{equation*}

  We may assume that $\N\eI_4 \leq t_2$.  By Lemma \ref{lem:omin}, $g$
  is piecewise continuously diferentiable and monotonic on $[\N\eI_4,
  t_2]$, and the number of pieces can be bounded by an absolute
  constant. Using the notation from Section \ref{sec:second_summation}
  (with $a = -1/2$), we see that the sum over $\e_7$ in
  \eqref{eq:A3_sec_sum_M1_start} is just $S(\N\eI_4, t_2)$, and
  Proposition \ref{prop:second_summation}, applied as suggested by
  Remark \ref{rem:second_summation_no_rho}, yields
  \begin{equation}\label{eq:A3_sec_sum_prop}
  \begin{aligned}
    S(\N\eI_4, t_2) &= \frac{2\pi}{\sqrt{|\Delta_K|}}
    \mathcal{A}(\vartheta(\aaa), \aaa, \OO_K) \int_{t \geq \N\eI_4}g(t)\dd
    t\\ &+
    O\left(\frac{2^{\omega(\eI_1\eI_2\eI_3\eI_5\eI_6)}B^{1/2}(\log B)}{\N\eI_3^{1/2}\N\eI_5^{1/2}\N\eI_6}\right).
  \end{aligned}
\end{equation}

  Clearly, $\pi \int_{t \geq \N\eI_4}g(t)\dd t = V_{87}(\N\eI_1,
  \ldots, \N\eI_6; B)$, so we obtain the correct main term.

  Let us consider the error term. Taking the product of \eqref{eq:A3_height_1} and \eqref{eq:A3_height_3} together with
  $\N\eI_7 > \N\eI_4$ (resp. $t > \N\eI_4$), we see that both the sum and the integral in \eqref{eq:A3_sec_sum_prop} are
  zero unless
  \begin{equation}
    \label{eq:A3_first_third_height_cond_M11}
    \N\eI_1^4\N\eI_2^5\N\eI_3^3\N\eI_4^4\N\eI_5^3\N\eI_6 \leq B^2\text.
  \end{equation}
  Since $|\OO_K^\times| < \infty$, we may sum over the $\eII'':=(\eI_1, \ldots, \eI_6)$ satisfying
  \eqref{eq:A3_first_third_height_cond_M11} instead of the $\ee''$, so
  the total error term is
  \begin{align*}
      &\ll \sum_{\substack{\eI_1, \ldots, \eI_6\in\I_K\\\eqref{eq:A3_first_third_height_cond_M11}}}\frac{2^{\omega(\eI_1\eI_2\eI_3\eI_5\eI_6)} B^{1/2}(\log B)}{\N\eI_3^{1/2}\N\eI_5^{1/2}\N\eI_6}\\
      &\ll \sum_{\substack{\eI_2, \ldots, \eI_6\in\I_K\\\N\eI_j \leq
          B}}\frac{2^{\omega(\eI_2\eI_3\eI_5\eI_6)} B(\log
        B)^2}{\N\eI_2^{5/4}\N\eI_3^{5/4}\N\eI_4\N\eI_5^{5/4}\N\eI_6^{5/4}}
      &\ll B(\log B)^3.
  \end{align*}
  In the summations, we used
  \eqref{eq:A3_first_third_height_cond_M11}, Lemma \ref{lem:omega},
  and Lemma \ref{lem:3.4}.
\end{proof}

\begin{lemma}\label{lem:A3_second_summation_ideals_M11}
  If $\eII''$ runs over all six-tuples $(\eI_1, \ldots, \eI_6)$ of
  nonzero ideals of $\OO_K$, then we have
  \begin{equation*}
    N_{87}(B) = \left(\frac{2}{\sqrt{|\Delta_K|}}\right)^2\sum_{\eII''}\mathcal{A}(\theta_8(\eII'', \eI_7), \eI_7) V_{87}(\N\eI_1, \ldots, \N\eI_6; B) + O(B(\log B)^3)\text. 
  \end{equation*}
\end{lemma}

\begin{proof}
  It follows directly from (\ref{eq:A3_def_Oj}) that $([\OO_1^{-1}],
  \ldots, [\OO_6^{-1}])$ runs through all six-tuples of ideal classes
  whenever $\classtuple$ runs through $\classrep^6$. If $\OO_j^{-1}$
  runs through a set of representatives for the ideal classes and
  $\e_j$ runs through all nonzero elements in $\OO_j$, then $\eI_j =
  \e_j\OO_j^{-j}$ runs through all nonzero integral ideals of $\OO_K$,
  each one occurring $|\OO_K^\times| = \omega_K$ times. This proves the
  lemma.
\end{proof}

\subsubsection{The remaining summations in $N_{87}(B)$}
\begin{lemma}\label{lem:A3_completion_M11}
  We have
  \begin{equation*}
    N_{87}(B) = \pi^6\left(\frac{2}{\sqrt{|\Delta_K|}}\right)^8 \left(\frac{h_K}{\omega_K}\right)^6 \theta_0V_{870}(B) + O(B(\log B)^4\log \log B),
  \end{equation*}
  where
  \begin{equation*}
    V_{870}(B) := \int_{t_1, \dots, t_6 \ge 1}V_{87}(t_1, \dots, t_6;B) \dd t_1 \cdots \dd t_6
  \end{equation*}
  and
  \begin{equation}
    \label{eq:theta_0}
    \theta_0 := \prod_{\p}\left(1-\frac{1}{\N\p}\right)^6\left(1 + \frac{6}{\N\p} + \frac{1}{\N\p^2}\right)\text.
  \end{equation}
\end{lemma}

\begin{proof}
  We start from Lemma
  \ref{lem:A3_second_summation_ideals_M11}. Applying Lemma
  \ref{lem:5.1_4}, (6), to \eqref{eq:A3_height_5}, we see that
  \begin{equation*}
    V_{87}(t_1, \ldots, t_6; B) \ll\frac{1}{t_5}\cdot \frac{B^{2/3}t_5^{2/3}}{t_1^{1/3}t_3^{1/3}t_4^{2/3}t_6^{2/3}}
    =\frac{B}{t_1 \cdots
      t_6}\left(\frac{B}{t_1^2t_2^3t_3^2t_4t_5^2t_6}\right)^{-1/3}\text.
  \end{equation*}
  We apply Proposition~\ref{prop:4.3} with $r=5$ (the assumptions on
  $V = V_{87}$ are satisfied by Lemma~\ref{lem:omin}). By Lemma
  \ref{lem:general_comp_theta_1}, $\theta_0 =
  \mathcal{A}(\theta_8(\eII'), \eII')$ has the desired form.
\end{proof}

\subsubsection{The second summation over $\e_4$ in
  $M^{(84)}_\classtuple(B)$}

\begin{lemma}\label{lem:A3_second_summation_M12}
  Write $\ee'' := (\e_1, \e_2, \e_3, \e_5, \e_6, \e_7)$. Moreover, let $\bO'' :=
  \OO_{1*} \times\OO_{2*} \times\OO_{3*} \times\OO_{5*}\times\OO_{6*}
  \times \OO_{7*}$. We have
  \begin{align*}
    M_\classtuple^{(84)}(B) &= \left(\frac{2}{\sqrt{|\Delta_K|}}\right)^2\sum_{\ee''\in \bO''} \mathcal{A}(\theta_8(\eII'), \eI_4)V_{84}(\N\eI_1, \N\eI_2, \N\eI_3, \N\eI_5, \N\eI_6, \N\eI_7; B)\\
    &+ O_\classtuple(B(\log B)^3),
  \end{align*}
  For $t_1$, $t_2$, $t_3$, $t_5$, $t_6$, $t_7 \geq 1$,
  \begin{equation*}
    V_{84}(t_1, t_2, t_3, t_5, t_6, t_7; B) := \frac{\pi}{t_5}\int_{\substack{(\sqrt{t_1}, \ldots, \sqrt{t_7}, \e_8) \in \mathcal{R}(B)\\t_4 \geq t_7\text{, }\abs{\e_8}\geq t_7}}\dd t_4 \dd \e_8,
  \end{equation*}
  with a real variable $t_4$ and a complex variable $\e_8$.
\end{lemma}

\begin{proof}
  This is similar to Lemma \ref{lem:A3_second_summation_M11}. Let
  $\vartheta(\aaa) := \theta_8(\eI_1, \eI_2, \eI_3, \aaa, \eI_5,
  \eI_6, \eI_7)$ and $g(t) := V_8(\N\eI_1, \N\eI_2, \N\eI_3, t,
  \N\eI_5, \N\eI_6, \N\eI_7; B)$. Then
  \begin{equation}\label{eq:A3_sec_sum_M12_start}
    M^{(84)}_\classtuple(B) = \frac{2}{\sqrt{|\Delta_K|}}\sum_{\ee'' \in \bO''}\sum_{\substack{\e_4 \in \OO_{4*}\\\N\eI_4 \geq \N\eI_7}}\vartheta(\eI_4)g(\N\eI_4) + O(B(\log B)^3)\text.
  \end{equation}
  By Lemma \ref{lem:general_comp_theta_1} and Lemma \ref{lem:6.8},
  $\vartheta$ satisfies \eqref{eq:second_sum_thetabound} with $C =
  0$ and $c_\theta \ll 2^{\omega(\eI_2\eI_3\eI_5\eI_6)}$. By
  \eqref{eq:A3_height_1}, $g(t) = 0$ whenever $t > t_2
  := B^{1/2}/(\N\eI_1\N\eI_2\N\eI_3^{1/2}\N\eI_5^{1/2}
  \N\eI_7^{1/2})$. Moreover, applying Lemma \ref{lem:5.1_4}, (2), to
  \eqref{eq:A3_height_5}, we see that
  \begin{equation*}
    g(t) \ll \frac{1}{\N\eI_5}\cdot \frac{(B\N\eI_5)^{1/2}}{(\N\eI_3\N\eI_6^2\N\eI_7)^{1/2}} = \frac{B^{1/2}}{\N\eI_3^{1/2}\N\eI_5^{1/2}\N\eI_6\N\eI_7^{1/2}} =: c_g\text.
  \end{equation*}

  Clearly, we may assume that $\N\eI_7 \leq t_2$. Using the notation
  from Section \ref{sec:second_summation} (with $a = 0$), the sum over
  $\e_4$ in \eqref{eq:A3_sec_sum_M12_start} is just $S(\N\eI_7, t_2)$,
  and Proposition \ref{prop:second_summation} yields
  \begin{align*}
    S(\N\eI_7, t_2) &= \frac{2\pi}{\sqrt{|\Delta_K|}}
    \mathcal{A}(\vartheta(\aaa), \aaa, \OO_K) \int_{t \geq \N\eI_7}g(t)\dd
    t\\ &+
    O\left(\frac{2^{\omega(\eI_2\eI_3\eI_5\eI_6)}B^{1/2}}{\N\eI_3^{1/2}\N\eI_5^{1/2}\N\eI_6\N\eI_7^{1/2}}\cdot
      \frac{B^{1/4}}{\N\eI_1^{1/2}\N\eI_2^{1/2}\N\eI_3^{1/4}\N\eI_5^{1/4}\N\eI_7^{1/4}}\right).
  \end{align*}
  Now $\pi \int_{t \geq \N\eI_7}g(t)\dd t = V_{84}(\N\eI_1, \N\eI_2,
  \N\eI_3, \N\eI_5, \N\eI_6, \N\eI_7; B)$, so we obtain the correct
  main term.  Let us consider the error term. Height condition
  \eqref{eq:A3_height_3} and $\N\eI_4 \geq \N\eI_7$ imply that both
  the sum and the integral are zero unless
  \begin{equation}
    \label{eq:A3_third_height_cond_M12}
    \N\eI_1^2\N\eI_2^3\N\eI_3^2\N\eI_5^2\N\eI_6\N\eI_7 \leq B\text.
  \end{equation}
  Since $|\OO_K^\times| < \infty$, we may sum over the $\eII'' :=(\eI_1, \eI_2, \eI_3, \eI_5, \eI_6, \eI_7)$ satisfying
  \eqref{eq:A3_third_height_cond_M12} instead of the $\ee''$, so the
  error term is
    \begin{align*}
      &\ll \sum_{\substack{\eI_1, \eI_2, \eI_3, \eI_5, \eI_6, \eI_7\in\I_K\\\eqref{eq:A3_third_height_cond_M12}}}\frac{2^{\omega(\eI_2\eI_3\eI_5\eI_6)} B^{3/4}}{\N\eI_1^{1/2}\N\eI_2^{1/2}\N\eI_3^{3/4}\N\eI_5^{3/4}\N\eI_6\N\eI_7^{3/4}}\\
      &\ll \sum_{\substack{\eI_1, \eI_2, \eI_3, \eI_5, \eI_6\in\I_K\\\N\eI_j \leq B}}\frac{2^{\omega(\eI_2\eI_3\eI_5\eI_6)} B}{\N\eI_1\N\eI_2^{5/4}\N\eI_3^{5/4}\N\eI_5^{5/4}\N\eI_6^{5/4}}\\
      &\ll B(\log B)\text.\qedhere
    \end{align*}
\end{proof}

\begin{lemma}\label{lem:A3_second_summation_ideals_M12}
  If $\eII''$ runs over all six-tuples $(\eI_1, \eI_2, \eI_3, \eI_5,
  \eI_6, \eI_7)$ of nonzero ideals of $\OO_K$, then we have
  \begin{align*}
    N_{84}(B) &= \left(\frac{2}{\sqrt{|\Delta_K|}}\right)^2\sum_{\eII''}\mathcal{A}(\theta_8(\eII'), \eI_4) V_{84}(\N\eI_1, \N\eI_2, \N\eI_3, \N\eI_5, \N\eI_6, \N\eI_7; B)\\
    &+ O(B(\log B)^3)\text.
  \end{align*}
\end{lemma}

\begin{proof}
  This is entirely analogous to the proof of Lemma \ref{lem:A3_second_summation_ideals_M11}.
\end{proof}

\subsubsection{The remaining summations in $N_{84}(B)$}
\begin{lemma}\label{lem:A3_completion_M12}
  We have
  \begin{equation*}
    N_{84}(B) = \pi^6\left(\frac{2}{\sqrt{|\Delta_K|}}\right)^8 \left(\frac{h_K}{\omega_K}\right)^6 \theta_0V_{840}(B) + O(B(\log B)^4\log \log B),
  \end{equation*}
  where
  \begin{equation*}
    V_{840}(B) := \int_{t_1, t_2, t_3, t_5, t_6, t_7 \ge 1}V_{84}(t_1, t_2, t_3, t_5, t_6, t_7;B) \dd t_1 \dd t_2 \dd t_3 \dd t_5 \dd t_6 \dd t_7
  \end{equation*}
  and $\theta_0$ is given in \eqref{eq:theta_0}.
\end{lemma}

\begin{proof}
  We start from Lemma \ref{lem:A3_second_summation_ideals_M12}. Using
  Lemma \ref{lem:5.1_4}, (5), applied to \eqref{eq:A3_height_5}, we
  have
  \begin{align*}
    V_{84}(t_1, t_2, t_3, t_5, t_6, t_7; B) &\ll\frac{1}{t_5}\cdot\frac{B^{3/4}t_5^{3/4}}{t_1^{1/2}t_3^{1/4}t_6^{1/2}t_7^{5/4}}\\
    &=\frac{B}{t_1t_2t_3t_5t_6t_7}\left(\frac{B}{t_1^2t_2^4t_3^3t_5^3t_6^2t_7^{-1}}\right)^{-1/4}\text.
  \end{align*}
  Moreover, using \eqref{eq:A3_height_1} to bound $t_4$ and \eqref{eq:A3_height_2} to bound $\e_8$, we
  have
  \begin{align*}
    V_{84}(t_1, t_2, t_3, t_5, t_6, t_7; B) &\ll\frac{1}{t_5}\cdot\frac{B^{1/2}}{t_1t_2t_3^{1/2}t_5^{1/2}t_7^{1/2}}\cdot\frac{B}{t_1t_2^2t_3^2t_5t_6^2}\\
    &=\frac{B}{t_1t_2t_3t_5t_6t_7}\left(\frac{B}{t_1^2t_2^4t_3^3t_5^3t_6^2t_7^{-1}}\right)^{1/2}\text.
  \end{align*}
  We apply Proposition~\ref{prop:4.3} with $r=5$. Again, we evaluate $\theta_0 = \mathcal{A}(\theta_8(\eII'), \eII')$ using Lemma \ref{lem:general_comp_theta_1}.
\end{proof}

\subsubsection{Combining the summations}
\begin{lemma}\label{lem:A3_completion}
  We have
  \begin{equation*}
    N_{U,H}(B) = \left(\frac{2}{\sqrt{|\Delta_K|}}\right)^8 \left(\frac{h_K}{\omega_K}\right)^6 \theta_0V_0(B) + O(B(\log B)^4\log \log B),
  \end{equation*}
  where $\theta_0$ is given in \eqref{eq:theta_0} and
  \begin{equation*}
    V_0(B) := \int\limits_{\substack{\abs{\e_1}, \ldots, \abs{\e_8} \ge 1\\(\e_1, \ldots, \e_8) \in \mathcal{R}(B)}}\frac{1}{\abs{\e_5}} \dd \e_1 \cdots \dd \e_8\text,
  \end{equation*}
  with complex variables $\e_i$.
\end{lemma}

\begin{proof}
  Similarly as in the proof of Lemma \ref{lem:A3_first_summation_M1},
  we note that $(\e_1, \ldots, \e_8) \in \mathcal{R}(B)$ holds if and
  only if $(|\e_1|, \ldots, |\e_7|, e^{i
    \arg((\e_3\e_6^2)/(\e_1\e_4^2\e_7))}\e_8) \in
  \mathcal{R}(B)$. Using polar coordinates, we obtain
  \begin{align*}
    V_{870}(B) + V_{840}(B) &= \pi\int\limits_{\substack{t_1, \ldots, t_7 \ge 1\text{, }\abs{\e_8}\geq t_7\\(\sqrt{t_1}, \ldots, \sqrt{t_7}, \e_8) \in \mathcal{R}(B)}}\frac{1}{t_5} \dd t_1 \cdots \dd t_7\dd \e_8\\
    &= \pi^{-6} \!\!\int\limits_{\substack{\abs{\e_1}, \ldots, \abs{\e_8} \ge 1\text{, }\abs{\e_8} \ge
        \abs{\e_7}\\(\e_1, \ldots, \e_8)\in \mathcal{R}(B)}}\!\frac{1}{\abs{\e_5}} \dd \e_1 \cdots \dd \e_8 =: \tilde{V}_8(B)\text.
  \end{align*}
  Therefore,
  \begin{equation*}
    N_8(B) = \pi^6\left(\frac{2}{\sqrt{|\Delta_K|}}\right)^8 \left(\frac{h_K}{\omega_K}\right)^6 \theta_0\tilde{V}_8(B) + O(B(\log B)^4\log \log B)\text.
  \end{equation*}
  For the computation of $N_7(B)$, we notice that our height and
  coprimality conditions are symmetric with respect to swapping the
  indices $(1, 4, 7)$ with $(3, 6, 8)$. This allows us to perform the
  first summation over $\e_7$ analogously to Lemma
  \ref{lem:A3_first_summation_M1}, the second summation over $\e_8$
  (resp. $\e_6$) analogously to Lemma
  \ref{lem:A3_second_summation_M11} (resp. Lemma
  \ref{lem:A3_second_summation_M12}), and the remaining summations analogously to Lemma \ref{lem:A3_completion_M11} (resp. Lemma \ref{lem:A3_completion_M12}). We obtain
  \begin{equation*}
    N_7(B) = \pi^6\left(\frac{2}{\sqrt{|\Delta_K|}}\right)^8 \left(\frac{h_K}{\omega_K}\right)^6 \theta_0\tilde V_7(B) + O(B(\log B)^4\log \log B)\text,
  \end{equation*}
  where
  \begin{equation*}
    \tilde V_7(B):= \pi^{-6}\int\limits_{\substack{\abs{\e_1}, \ldots, \abs{\e_8} \ge 1\text{, }\abs{\e_7} \geq \abs{\e_8}\\(\e_1, \ldots, \e_8)\in \mathcal{R}(B)}}\frac{1}{\abs{\e_5}} \dd \e_1 \cdots \dd \e_8\text.
  \end{equation*}
  The lemma follows immediately.
\end{proof}

\subsection{Proof of Theorem \ref{thm:a3_main}}

To compare the result of Lemma \ref{lem:A3_completion} with
Theorem \ref{thm:a3_main}, we introduce the conditions
\begin{align}
  &\abs{\e_1^2\e_3^{2}\e_4\e_5^2\e_6}\leq B,\label{eq:A3_comparison_1}\\
  &\abs{\e_1^2\e_3^{2}\e_4\e_5^2\e_6}\leq B \text{, }\abs{\e_1^{2}\e_3^{-1}\e_4^{4}\e_5^{-1}\e_6^{-2}}\le B,\label{eq:A3_comparison_2}\\
  &\abs{\e_1^2\e_3^{2}\e_4\e_5^2\e_6}\leq B \text{,
  }\abs{\e_1^{2}\e_3^{-1}\e_4^{4}\e_5^{-1}\e_6^{-2}}\le B \text{,
  }\abs{\e_1^{-1}\e_3^2\e_4^{-2}\e_5^{-1}\e_6^4}\le
  B.\label{eq:A3_comparison_3}
\end{align}

\begin{lemma}\label{lem:A3_predicted_volume}
  Let $\omega_\infty$ be as in Theorem \ref{thm:a3_main},
  $\mathcal{R}(B)$ as in
  \eqref{eq:A3_height_1}--\eqref{eq:A3_height_5}, and
  \begin{equation*}
    V_0'(B) := \int_{\substack{(\e_1, \dots, \e_8)\in \mathcal{R}(B)\\\abs{\e_1}, \abs{\e_3}, \ldots, \abs{\e_6} \geq 1\\\eqref{eq:A3_comparison_3}}} \frac{1}{\abs{\e_5}} \dd \e_1 \cdots \dd
      \e_8\text.
  \end{equation*}
  Then $\frac{1}{4320}\pi^6 \omega_\infty B(\log B)^5 = 4 V_0'(B)$.
\end{lemma}

\begin{proof}
  Note that substituting $y_0 = \e_1\e_4^2\e_7$, $y_1 =
  \e_3\e_6^2\e_8$, $y_2 = \e_1\e_2\e_3\e_4\e_5\e_6$, $-(y_0+y_1) =
  \e_5\e_9$ (which are obtained using the substitutions in Section
  \ref{sec:passage}) in \eqref{eq:A3_parameterization} and cancelling
  out $\e_1\e_3\e_4^2\e_5\e_6^2$ gives $\Psi(\e_1, \ldots, \e_9)$ as
  in the proof of Lemma \ref{lem:A3_passage_to_torsor}. This motivates
  the following substitutions in $\omega_\infty$: Let $\e_1$, $\e_3$, $\e_4$, $\e_5$,
  $\e_6 \in \CC\smallsetminus\{0\}$ and $B \in \RR_{>0}$. Let $\e_2$,
  $\e_7$, $\e_8$ be complex variables. With $l :=
  (B\abs{\e_1\e_3\e_4^2\e_5\e_6^2})^{1/2}$, we apply the coordinate transformation
  $y_0 = l^{-1/3}\e_1\e_4^2\cdot \e_7$, $y_1 = l^{-1/3}\e_3\e_6^2
  \cdot \e_8$, $y_2 = l^{-1/3}\e_1\e_3\e_4\e_5\e_6\cdot \e_2$ of
  Jacobi determinant
  \begin{equation}
    \frac{\abs{\e_1\e_3\e_4\e_5\e_6}}{B}\frac{1}{\abs{\e_5}}
  \end{equation}
  and obtain
  \begin{equation}\label{eq:A3_complex_density_torsor}
    \omega_\infty = \frac{12}{\pi}\frac{\abs{\e_1\e_3\e_4\e_5\e_6}}{B} \int_{(\e_1, \ldots, \e_8)\in \mathcal{R}(B)}\frac{1}{\abs{\e_5}}\dd \e_2 \dd \e_7 \dd \e_8\text.
  \end{equation}

  An application of Lemma~\ref{lem:alpha_integral} with exchanged roles of $\e_2$ and
  $\e_6$ gives
  \begin{equation*}
    \alpha(\tS)(\log B)^5 = \frac{1}{3\pi^5} \int_{\substack{\abs{\e_1}, \abs{\e_3}, \ldots, \abs{\e_6} \geq 1\\\eqref{eq:A3_comparison_3}}} \frac{\dd
      \e_1\dd \e_3 \cdots \dd \e_6}{\abs{\e_1\e_3 \cdots \e_6}},
  \end{equation*}
  since $[-K_{\tS}] = [2E_1+3E_2+2E_3+E_4+2E_5+E_6]$, $[E_7] =
  [E_2+E_3-E_4+E_5+E_6]$, and $[E_8] =[E_1+E_2+E_4+E_5-E_6]$. By
  (\ref{eq:alpha}), we have $\alpha(\tS) = 1/4320$.
  
  The lemma follows by substituting this
  and \eqref{eq:A3_complex_density_torsor} in $\frac{1}{4320}\pi^6 \omega_\infty B(\log B)^5$.  
\end{proof}

To finish our proof, we compare $V_0(B)$ from Lemma
\ref{lem:A3_completion} with $V_0'(B)$ from Lemma
\ref{lem:A3_predicted_volume}. Let
\begin{align*}
  \mathcal{D}_0(B) &:= \{(\e_1, \ldots,
\e_8) \in \mathcal{R}(B) \mid \abs{\e_1}, \ldots,  \abs{\e_8} \geq 1\},\\
  \mathcal{D}_1(B) &:= \{(\e_1, \ldots,
\e_8) \in \mathcal{R}(B) \mid \abs{\e_1}, \ldots,  \abs{\e_8} \geq 1\text{, }\eqref{eq:A3_comparison_1}\},\\
  \mathcal{D}_2(B) &:= \{(\e_1, \ldots,
\e_8) \in \mathcal{R}(B) \mid \abs{\e_1}, \ldots,  \abs{\e_8} \geq 1\text{, }\eqref{eq:A3_comparison_2}\},\\
  \mathcal{D}_3(B) &:= \{(\e_1, \ldots,
\e_8) \in \mathcal{R}(B) \mid \abs{\e_1}, \ldots,  \abs{\e_8}\geq 1\text{, }\eqref{eq:A3_comparison_3}\},\\
  \mathcal{D}_4(B) &:= \{(\e_1, \ldots,
\e_8) \in \mathcal{R}(B) \mid \abs{\e_1}, \ldots, \abs{\e_6}, \abs{\e_8} \geq 1\text{, }\eqref{eq:A3_comparison_3}\},\\
  \mathcal{D}_5(B) &:= \{(\e_1, \ldots,
\e_8) \in \mathcal{R}(B) \mid \abs{\e_1}, \ldots, \abs{\e_6} \geq 1\text{, }\eqref{eq:A3_comparison_3}\},\\
  \mathcal{D}_6(B) &:= \{(\e_1, \ldots,
\e_8) \in \mathcal{R}(B) \mid \abs{\e_1}, \abs{\e_3},\ldots,\abs{\e_6} \geq 1\text{, }\eqref{eq:A3_comparison_3}\}.
\end{align*}
Moreover, let
\begin{equation*}
  V_i(B) := \int_{\mathcal{D}_i(B)}\frac{1}{\abs{\e_5}}\dd \e_1 \cdots \dd \e_8 \text.
\end{equation*}
Then clearly $V_0(B)$ is as in Lemma \ref{lem:A3_completion} and
$V_6(B) = V_0'(B)$. We show that, for $i=1, \dots, 6$, $V_i(B) -
V_{i-1}(B) = O(B(\log B)^4)$. This is clear for $i = 1$, since, by
\eqref{eq:A3_height_3} and $t_2 \geq 1$, we have $\mathcal{D}_1 = \mathcal{D}_0$. Moreover,
using Lemma \ref{lem:5.1_4}, (4), and \eqref{eq:A3_height_5} to bound the
integral over $\e_7$ and $\e_8$, we have
\begin{equation*}
  V_2(B) - V_1(B) \ll \int_{\substack{\abs{\e_1}, \ldots, \abs{\e_6}\geq 1\\\abs{\e_1^2\e_2^2\e_3\e_4^2\e_5}\leq B\\\abs{\e_1^2\e_3^{-1}\e_4^4\e_5^{-1}\e_6^{-2}} > B}}\frac{B^{2/3}}{\abs{\e_1\e_3\e_4^{2}\e_5\e_6^{2}}^{1/3}}\dd \e_1 \cdots \dd \e_6 \ll B(\log B)^4\text.
\end{equation*}
An entirely symmetric argument shows that $V_3(B) - V_2(B) \ll B(\log
B)^4$. Using Lemma \ref{lem:5.1_4}, (2), and \eqref{eq:A3_height_5} to
bound the integral over $\e_8$, we obtain
\begin{equation*}
  V_4(B) - V_3(B) \ll \int\limits_{\substack{\abs{\e_1},\ldots,\abs{\e_6}\geq 1\\\abs{\e_7} < 1\text{, }\eqref{eq:A3_comparison_3}\\ \abs{\e_1^2\e_2^3\e_3^2\e_4\e_5^2\e_6} \leq B}}\frac{B^{1/2}}{\abs{\e_3 \e_5 \e_6^2 \e_7}^{1/2}}\dd \e_1 \cdots \dd \e_7 \ll B(\log B)^4\text.  
\end{equation*}
Here, we first integrate over $\e_7$ and $t_2$. Again, an analogous
argument shows that $V_5(B) - V_4(B) \ll B(\log B)^4$. Finally, using
Lemma \ref{lem:5.1_4}, (4), and \eqref{eq:A3_height_5} to bound
the integral over $\e_7$ and $\e_8$, we have
\begin{equation*}
  V_6(B) - V_5(B) \ll \int\limits_{\substack{\abs{\e_1},\ldots,\abs{\e_6}\geq 1\\0 < t_2 < 1\text{, }\eqref{eq:A3_comparison_1}}}\frac{B^{2/3}}{\abs{\e_1\e_3\e_4^{2}\e_5\e_6}^{1/3}}\dd \e_1 \cdots \dd \e_6 \ll B(\log B)^4\text.
\end{equation*}
Thus, $V_0(B) = V_0'(B) + O(B(\log B)^4)$. Using Lemma~\ref{lem:A3_completion} and
Lemma~\ref{lem:A3_predicted_volume}, this implies
Theorem~\ref{thm:a3_main}.

\subsection{Over $\QQ$}\label{sec:A3_Q}

The following result is the analog over $\QQ$ of Theorem \ref{thm:a3_main}.

\begin{theorem}\label{thm:a3_main_Q}
  For the number of $\QQ$-rational points of bounded height on the
  subset $U$ obtained by removing the lines of $S$ and $B \ge 3$, we have
  \begin{equation*}
    N_{U,H}(B) = c_{S, H} B(\log
    B)^5 + O(B(\log B)^4\log \log B),
  \end{equation*}
  where
  \begin{equation*}
    c_{S, H} = \frac{1}{4320} \cdot \prod_p \left(1-\frac{1}{p}\right)^6\left(1+\frac{6}{p}+\frac{1}{p^2}\right) \cdot \omega_\infty 
  \end{equation*}
  with
  \begin{equation*}
    \omega_\infty = \frac{3}{2}\int_{\max\{|y_0y_2^2|, |y_1y_2^2|,
      |y_2^3|, |y_0y_1y_2|, |y_0y_1(y_0+y_1)|\}\leq 1}\dd y_0 \dd y_1
    \dd y_2.
  \end{equation*}
\end{theorem}

\begin{proof}
This is similar to the case of imaginary quadratic $K$ above, so we shall be
very brief.

The parameterization of rational points by integral points on the
universal torsor is as in Lemma~\ref{lem:A3_passage_to_torsor}, here
and everywhere below with $\omega_\QQ=2$, $h_K=1$ so that $\classrep$
contains only the trivial ideal class, with $\OO_j=\ZZ$ for $j=1,
\dots, 9$, $\OO_{1*}=\dots=\OO_{8*}=\ZZ_{\ne 0}$ and $\OO_{9*}=\ZZ$, and
with $\abs{\cdot}$ replaced by the ordinary absolute value $|\cdot|$
on $\RR$ in (\ref{eq:A3_height}).

The proof of the asymptotic formula proceeds as in the imaginary
quadratic case, but using the original techniques over $\QQ$ from
\cite{MR2520770}. In the statements of the intermediate results, we
must always replace $2/\sqrt{|\Delta_K|}$ by $1$, complex by real
integration, $\pi$ by $2$, and $\sqrt{t_i}$ by $t_i$. The computation of
the main terms is always analogous, but less technical. The estimation
of the error terms is often analogous and sometimes easier.

The main changes are as follows. For the first summation, we apply
\cite[Proposition~2.4]{MR2520770}. The error term
$2^{\omega(\e_2)+\omega(\e_1\e_2\e_3\e_4)}$ can be estimated as the
second summand of the error term in
Lemma~\ref{lem:A3_first_summation_M1}.

For the second summation over $\e_7$, we can apply \cite[Lemma~3.1,
Corollary~6.9]{MR2520770}. The error term is
\begin{align*}
  &\sum_{\e_1, \dots, \e_6} 2^{\omega(\e_1\e_2\e_3\e_5\e_6)}
  \sup_{|\e_7|>|\e_4|} \tilde{V}_8(\e_1, \dots, \e_7;B) \ll
  \sum_{\e_1, \dots, \e_6}
  \frac{2^{\omega(\e_1\e_2\e_3\e_5\e_6)}B^{1/2}\log B}{|\e_3|^{1/2}|\e_4|^{1/2}|\e_5|^{1/2}|\e_6|}\\
  \ll{}& \sum_{\e_2, \dots, \e_6} \frac{2^{\omega(\e_2\e_3\e_5\e_6)}B
    (\log
    B)^2}{|\e_2|^{5/4}|\e_3|^{5/4}|\e_4|^{3/2}|\e_5|^{5/4}|\e_6|^{5/4}}
  \ll B (\log B)^2
\end{align*}
where (using $|\e_4|<|\e_7|$)
\begin{equation*}
  |\e_1| \le \left(\frac{B}{|\e_2^2\e_3\e_4^3\e_5|}\right)^{1/4}\left(\frac{B}{|\e_2^3\e_3^2\e_4\e_5^2\e_6|}\right)^{1/4}.
\end{equation*}
For the second summation over $\e_4$, the computation is very similar.

The remaining summations and the completion of the proof of
Theorem~\ref{thm:a3_main_Q} remain essentially unchanged.
\end{proof}

\bibliographystyle{alpha}

\bibliography{counting_imaginary_quadratic_points}

\end{document}